\documentclass[1 [leqno,11pt]{amsart}
\usepackage{amssymb, amsmath,amsmath,latexsym,amssymb,amsfonts,amsbsy, amsthm,mathtools,graphicx,CJKutf8,CJKnumb,CJKulem,color}

\numberwithin{equation}{section}

\allowdisplaybreaks

\let\al=\alpha
\let\b=\beta
\let\g=\gamma

\let\s=\sigma
\let\f=\frac

\let\om=\omega

\let\Om=\Omega

\let\na=\nabla

\let\pa=\partial
\let\ep=\epsilon
\let\ka=\kappa

\def\As{A^{\sigma}}

\def\tu{\widetilde{u}}

\def\R{\mathbf R}

\def\no{\noindent}

\def\eqdef{\buildrel\hbox{\footnotesize def}\over =}

\def\bbT{\mathbb{T}}

\newcommand{\beq}{\begin{equation}}
\newcommand{\eeq}{\end{equation}}
\newcommand{\ben}{\begin{eqnarray}}
\newcommand{\een}{\end{eqnarray}}
\newcommand{\beno}{\begin{eqnarray*}}
\newcommand{\eeno}{\end{eqnarray*}}

\newtheorem{theorem}{Theorem}[section]

\newtheorem{lemma}[theorem]{Lemma}
\newtheorem{proposition}[theorem]{Proposition}

\newtheorem{remark}[theorem]{Remark}

\begin{document}

\title{Stability threshold of the 2D Couette flow in Sobolev spaces}
\author{Nader Masmoudi}
\address{Department of Mathematics, New York University in Abu Dhabi, Saadiyat Island, P.O. Box 129188, Abu Dhabi, United Arab Emirates. Courant Institute of Mathematical Sciences, New York University, 251 Mercer Street, New York, NY 10012, USA,}
\email{masmoudi@cims.nyu.edu}
\author{Weiren Zhao}
\address{Department of Mathematics, New York University in Abu Dhabi, Saadiyat Island, P.O. Box 129188, Abu Dhabi, United Arab Emirates.}
\email{zjzjzwr@126.com, wz19@nyu.edu}

\date{\today}

\maketitle

\begin{abstract}
We study the stability threshold of the 2D Couette flow in Sobolev spaces at high Reynolds number $Re$. 
We prove that if the initial vorticity $\Om_{in}$ satisfies $\|\Om_{in}-(-1)\|_{H^{\s}}\leq \ep Re^{-1/3}$, then the solution of the 2D Navier-Stokes equation approaches to some shear flow which is also close to Couette flow for time $t\gg Re^{1/3}$ by a mixing-enhanced dissipation effect and then converges back to Couette flow when $t\to +\infty$. 
\end{abstract}


\section{Introduction}
In this paper, we consider the 2D incompressible Navier-Stokes equations in $\bbT\times \R$:
\beq
\label{eq:NS}
\left\{
\begin{array}{l}
\pa_tV+V\cdot\nabla V+\nabla P-\nu\Delta V=0,\\
\na\cdot V=0,\\
V|_{t=0}=V_{in}(x,y).
\end{array}\right.
\eeq
where $\nu$ denote the viscosity which is the multiplicative inverse of the Reynolds number $Re$.  $V=(U^1,U^2)$ and $P$ denote the velocity and the pressure of the fluid respectively. 
Let $\Om=\pa_xU^2-\pa_yU^1$ be the vorticity, which satisfies
\beq
\label{eq:vorticity}
\Om_t+V\cdot\nabla \Om-\nu\Delta \Om=0.
\eeq
The Couette $(y,0)$ is a steady solution of \eqref{eq:NS}. 

Now we introduce the perturbation, let $\Om=\om-1$ and $V=(y,0)+(U^x,U^y)$ then $\om=\pa_xU^y-\pa_yU^x$ satisfies 
\ben\label{eq:NS2}
\left\{
\begin{array}{l}
\pa_t\omega+y\pa_x\omega-\nu\Delta\om=-U\cdot\na \om,\\
\om|_{t=0}=\om_{in}(x,y),
\end{array}\right.
\een
and $U=(U^x,U^y)=(-\pa_y\psi, \pa_x\psi)$ with $\Delta\psi=\om$. 

The study of \eqref{eq:NS2} for small perturbations is an old problem in hydrodynamic stability, considered by both Rayleigh \cite{Ray} and Kelvin \cite{Kel}, as well as by many modern authors with new perspectives(see e.g. the classical texts \cite{Dra,Yag} and the references therein). Rayleigh and Kelvin both studied the linearization of \eqref{eq:NS2}, which is simply 
\ben\label{eq:LNS2}
\left\{
\begin{array}{l}
\pa_t\omega+y\pa_x\omega-\nu\Delta\om=0,\\
\Delta\psi=\om,\\
\om|_{t=0}=\om_{in}(x,y).
\end{array}\right.
\een
Indeed, if we denote by $\hat{\om}(t,k,\eta)$ the Fourier transform of $\om(t,x,y)$, then the solution of \eqref{eq:LNS2} can be write as
\beq\label{eq: Lin-sol}
\begin{split}
&\hat{\om}(t,k,\eta)=\hat{\om}_{in}(k,\eta+kt)\exp\left(-\nu\int_0^t|k|^2+|\eta-ks+kt|^2ds\right),\\
&\hat{\psi}(t,k,\eta)=\f{-\hat{\om}_{in}(k,\eta+kt)}{k^2+\eta^2}\exp\left(-\nu\int_0^t|k|^2+|\eta-ks+kt|^2ds\right),
\end{split}
\eeq
which gives that
\ben\label{eq: ED_and_ID}
\begin{split}
&\|\pa_yP_{\neq}\psi\|_{L^2}+\langle t\rangle \|\pa_xP_{\neq}\psi\|_{L^2}\leq C\langle t\rangle^{-1}e^{-c\nu t^3}\|P_{\neq}\om_{in}\|_{H^{2}},\\
&\|P_{\neq}\om\|_{L^2}\leq C\|P_{\neq}\om_{in}\|_{L^2}e^{-c\nu t^3},
\end{split}
\een
here we denote by $P_{\neq}f=f(x,y)-\f{1}{2\pi}\int_{\mathbb{T}}f(x,y)dx$ the projection to nonzero mode of $f$. The first inequality in \eqref{eq: ED_and_ID} is the inviscid damping  and the second one is the enhanced dissipation. These two results both are related to the vorticity mixing effect. 

In \cite{Orr}, Orr observed an important phenomenon that the velocity will tend to 0 as $t\to \infty$, even for a time reversible system such as the Euler equations($\nu=0$). This phenomenon is so-called inviscid damping, which is the analogue in hydrodynamics of Landau damping found by Landau \cite{Lan}, which predicted the rapid decay of the electric field of the linearized Vlasov equation around homogeneous equilibrium. Mouhot and Villani \cite{MV} made a breakthrough and proved nonlinear Landau damping for the perturbation in Gevrey class(see also \cite{BMM}).In this case, the mechanism leading to the damping is the vorticity mixing driven by shear flow or Orr mechanism \cite{Orr}. See \cite{Ry} for similar phenomena in various system. 
We point out that the inviscid damping for general shear flow is a challenge problem even in linear level due to the presence of the nonlocal operator for general shear flow.
For the linear inviscid damping we refer to \cite{Zill,WZZ1,Jia,GNRS2018} for the results of general monotone flows. 
For non-monotone flows such as the Poiseuille flow and the Kolmogorov flow, another dynamic phenomena should be taken into consideration, which is so called the vorticity depletion phenomena, predicted by Bouchet and Morita \cite{BouMor} and later proved by Wei, Zhang and Zhao \cite{WZZ2,WZZ3}. 
Due to possible nonlinear transient growth, it is a challenging task extending linear damping to nonlinear damping. 
Even for the Couette flow there are only few results. Moreover, nonlinear damping is sensitive to the topology of the perturbation. 
Indeed, Lin and Zeng \cite{LZ} proved that nonlinear inviscid damping is not true for the perturbation of the Couette flow in $H^s$ for $s<\f32$. 
Bedrossian and Masmoudi \cite{BM1} proved nonlinear inviscid damping around the Couette flow in Gevrey class $2_-$. 
Recently Deng and Masmoudi \cite{DM}  proved that the instability for initial perturbations in Gevrey class $2_+$. 
We refer to \cite{IJ,Jia2} and references therein for other related interesting results. 
Moreover it is also observed by Orr that, if we rewrite the linearized system by the change of coordinates $f(t,z,y)=\om(t,z+ty,y)$, then the Fourier transform of the stream function $\phi(t,z,y)=\psi(t,z+yt,y)$ is 
\ben\label{eq: stream-Four}
\hat{\phi}(t,k,\eta)=\f{\hat{f}(t,k,\eta)}{(\eta-kt)^2+k^2}, 
\een
The denominator of \eqref{eq: stream-Four} is minimized at $t=\f{\eta}{k}$ which is known as the Orr critical times. 

The second phenomenon --- enhanced dissipation is sometimes referred to modern authors as the ‘shear-diffusion mechanism’. This decay rate is much faster than the diffusive decay of $e^{-\nu t}$. The mechanism leading to the enhanced dissipation is also due to vorticity mixing. 

However, for the nonlinear system, the Orr mechanism is known to interact poorly with the nonlinear term, creating a weakly nonlinear effect referred to as an echo. 
The basic mechanism is straight-forward: a mode which is near its critical time is creating most of the velocity field and at this point can interact with the enstrophy which as already mixed to transfer enstrophy to a mode which is un-mixing. 
When this third mode reaches its critical time, the result of the nonlinear interaction becomes very strong (the time delay explains the terminology ‘echo’). 
There are two necessary ways to control(compete against) the echo cascades. 
One is to assume enough smallness of the initial perturbations such that the rapid growth of the enstrophy may not happen before enhanced dissipative time-scale $\nu^{-\f13}$. 
The other is to assume enough regularity (Gevery class) of the initial perturbations such that one can pay enough regularity to control the growth caused by the echo cascade. 

In this work, we are interested in the first method to stabilize the system and  studying the long time behavior of \eqref{eq:NS2} for small initial perturbations $\om_{in}$. We are aimed at finding the largest perturbation (threshold) in Sobolev spaces below which the Couette flow is stable. More precisely, we are studying the following classical question: 

{\it Given a norm $\|\cdot\|_X$, find a $\beta=\beta(X)$ so that
\beno
&&\|\om_{in}\|_{X}\leq \nu^{\b} \Rightarrow \text{stability},\\
&&\|\om_{in}\|_{X}\gg \nu^{\b} \Rightarrow \text{instability}.
\eeno
}
Another interesting question which is related to this problem is the nonlinear enhanced dissipation and inviscid damping which can be proposed in the following two ways:

1. {\it Given a norm $\|\cdot\|_{X}$($X\subset L^2$), determine a $\beta=\beta(X)$ so that for the initial vorticity $\|\om_{in}\|_{X}\ll \nu^{\beta}$ and for $t>0$
\ben\label{eq: enha-invis}
\|\om_{\neq}\|_{L^2_{x,y}}\leq C\|\om_{in}\|_{X}e^{-c\nu^{\f13}t}\quad \text{and}\quad
\|V_{\neq}\|_{L^2_{t,x,y}}\leq C\|\om_{in}\|_{X},
\een
or the weak enhanced dissipation type estimate
\ben\label{eq: enha-invis-weak}
\|\om_{\neq}\|_{L^2_tL^2_{x,y}}\leq C\nu^{-\f16}\|\om_{in}\|_{X}
\een
holds for the Navier-Stokes equation \eqref{eq:NS2}.}

2. {\it Given $\beta$, is there an optimal function space $X\subset L^2$ so that if the initial vorticity satisfies $\|\om_{in}\|_{X}\ll \nu^{\beta}$, then \eqref{eq: enha-invis} or \eqref{eq: enha-invis-weak} hold for the Navier-Stokes equation \eqref{eq:NS2}?}

These two problems(find the smallest $\beta$ or find the largest function space $X$) are related to each other, since one can gain regularity in a short time by a standard time-weight argument if the initial perturbation is small enough. 

We summarize the results as follows: 
\begin{itemize}
\item For $\beta=0$, Bedrossian, Masmoudi and Vicol \cite{BMV} showed that if $X$ is taken as Gevery-$m$ with $m<2$, then the Couette flow is stable and  \eqref{eq: enha-invis-weak} holds. 
\item For $\beta=\f12$, Bedrossian, Vicol and Wang \cite{BWV} proved the Couette flow is stable as well as the nonlinear enhanced dissipation and inviscid damping for the perturbation of initial vorticity in $H^s, s>1$. 
\item For $\beta=\f12$, recently in \cite{MZ1}, we proved the nonlinear enhanced dissipation and inviscid damping for the perturbation of initial vorticity in the almost critical space $H^{log}_xL^2_{y}\subset L^2_{x,y}$.
\end{itemize}

Let us also mention some other recent progress \cite{BGM1,BGM2,BGM3,CEW2019,DingLin,GNRS2018,LWZ,LX,IMM,WZ2018,WZZ3,ZEW} on the stability problem of different types of shear flows in different domains. 

In this paper, we find a smaller $\beta(=\f13)$ such that the Couette flow is stable and the nonlinear enhanced dissipation and inviscid damping hold, when $X$ takes a Sobolev spaces. Our main result is stated as follows. 
\begin{theorem}\label{Thm: main}
For $\s\geq 40$, $\nu>0$, there exist $0<\ep_0,\nu_0<1$, such that for all $0<\nu\leq \nu_0$ and $0<\ep\leq \ep_0$, if $\om_{in}$ satisfies $\|\om_{in}\|_{H^{\s}}\leq \ep\nu^{\f13}$, then the solution $\om(t)$ of \eqref{eq:NS2} with initial data $\om_{in}$ satisfies the following properties: \\
1. Global stability in $H^{\s}$,  
\beq
\left\|\om\left(t,x+ty+\Phi(t,y),y\right)\right\|_{H^{\s}}\leq C\ep\nu^{\f13},
\eeq
where $\Phi(t,y)$ is given explicitly by 
\beno
\Phi(t,y)=\int_0^te^{\nu(t-\tau)\pa_y^2}\left(\f{1}{2\pi}\int_{\mathbb{T}}U^x(\tau,x,y)dx\right)d\tau.
\eeno
2. Inviscid damping, 
\beq
\left\|P_{\neq}U^x\right\|_2+\langle t\rangle\left\|U^y\right\|_2\leq \f{C\ep\nu^{\f13}}{\langle \nu t^3\rangle}\langle t\rangle^{-1}.
\eeq
3. Weak enhanced dissipation, 
\beq\label{eq: enha-thm}
\left\|P_{\neq}\om(t)\right\|_2\leq \f{C\ep\nu^{\f13}}{\langle \nu t^3\rangle}.
\eeq
The constant $C$ is independent of $\nu$ and $\ep$. 
\end{theorem}
\begin{remark}
By replacing $D(t,\eta)$ by $D(t,\eta)^{\al}$ with $\al\geq 1$ in the proof and assuming $\s$ large enough(depending on $\al$), one can obtain the stronger enhanced dissipation of the following from: 
\ben\label{eq: enhan-stro}
\left\|P_{\neq}\om(t)\right\|_2\leq \f{C\ep\nu^{\f13}}{\langle \nu t^3\rangle^{\al}}.
\een
However, the weak enhanced dissipation of the same decay rate as in the Theorem \ref{Thm: main} is enough for the proof of the Sobolev stability. 
Both \eqref{eq: enha-thm} and \eqref{eq: enhan-stro} are far from the exponential decay of the linear case. 
\end{remark}

Let us now outline the main ideas in the proof of Theorem \ref{Thm: main}. First, we provide a (well chosen) change of variable that adapts to the solution as it evolves and yields a new ‘relative’ velocity which is time-integrable. Second, we will construct a new multiplier which can be regarded as a ghost weight in phase space and helps us control the growth caused by echo cascades. 

\section{Proof of Theorem \ref{Thm: main}}
In this section, we will present several key propositions and complete the proof of Theorem \ref{Thm: main} by admitting those propositions. 

\subsection{Notation and conventions}

See Section \ref{Sec: L-P} for the Fourier analysis conventions we are taking. A convention we generally use is to denote the discrete $x$ (or $z$) frequencies as subscripts. By convention we always use Greek letters such as $\eta$ and $\xi$ to denote frequencies in the $y$ or $v$ direction and lowercase Latin characters commonly used as indices such as $k$ and $l$ to denote frequencies in the $x$ or $z$ direction (which are discrete). Another convention we use is to denote $M,N,K$ as dyadic integers $M,N,K\in \mathbb{D}$ where 
\beno
\mathbb{D}=\left\{\f12,1,2,4,8,...,2^j,...\right\}.
\eeno
When a sum is written with indices $K,M,M',N$ or $N'$ it will always be over a subset of $\mathbb{D}$. 
We will mix use same $A$ for $Af=(A(\eta)\hat{f}(\eta))^{\vee}$ or $A\hat{f}=A(\eta)\hat{f}(\eta)$, where $A$ is a Fourier multiplier. 

We use the notation $f\lesssim g$ when there exists a constant $C>0$ independent of the parameters of interest such that $f\leq Cg$(we analogously $g\gtrsim f$ define). Similarly, we use the notation $f\approx g$ when there exists $C>0$ such that $C^{-1}g\leq f\leq Cg$. 

We will denote the $l^1$ vector norm $|k,\eta|=|k|+|\eta|$, which by convention is the norm taken in our work. Similarly, given a scalar or vector in $\R^n$ we denote 
\beno
\langle v\rangle = (1+|v|^2)^{\f12}. 
\eeno
We use a similar notation to denote the $x$ or $z$ average of a function: $<f>=\f{1}{2\pi}\int f(x,y)dx=f_0$. We also frequently use the notation $f_{\neq}=P_{\neq}f=f-f_0$. We denote the standard $L^2$ norms by $\|\cdot\|_2$. The norm of Sobolev space $H^{\s}$ is given by 
\beno
\|f\|_{H^{\s}}=\left\|(\langle \eta\rangle^{\s}\hat{f})^{\vee}\right\|_2
\eeno
The norm space-time Sobolev space $L^{p}_{T}(H^{\s})$ is given by
\beno
\|f\|_{L^{p}_{T}(H^{\g})}=\left\{
\begin{aligned}
&\sup_{t'\in[1,t]}\|f(t')\|_{H^{\s}},\quad p=\infty,\\
&\left(\int_1^t\|f(t')\|_{H^{\s}}^pdt'\right)^{\f1p},\quad 1\leq p<\infty.
\end{aligned}\right.
\eeno

For $|m|=0,1,2,...$ and $m\eta\geq 0$, let 
\ben
t_{m,\eta}=\f{2\eta}{2m+1}.
\een 
We then use 
\ben
I_{m,\eta}\eqdef [t_{m,\eta},t_{m-1,\eta}],
\een 
for $m=1,2,...,$ to denote any resonant interval and its left and right part with $\eta\geq (2m+1)m$. For $|\eta|\geq 3$, we denote $E(\sqrt{|\eta|})$ the largest integer satisfying $(2E(\sqrt{|\eta|})+1)E(\sqrt{|\eta|})\leq |\eta|$ and then $E(\sqrt{|\eta|})\approx  \sqrt{|\eta|}$. Let $t(\eta)=\f{2\eta}{2E(\sqrt{|\eta|})+1}\approx  \sqrt{|\eta|}$ be the starting of the resonant interval. Then we denote
\ben
I_t(\eta)\eqdef [t(\eta), 2|\eta|]=\bigcup_{m=1}^{E(\sqrt{\eta})}I_{m,\eta}
\een
the whole resonant interval. 

For a statement $Q$, $1_{Q}$ or $\chi^{Q}$ will denote the function that equals $1$ if $Q$ is true and $0$ otherwise. 

\subsection{Coordinate transform}
We will use the same change of coordinates as in \cite{BMV} which allows us to simultaneously 'mod out' by the evolution of the time-dependent background shear flow and treat the mixing of this background shear as a perturbation of the Couette flow (in particular, to understand the nonlinear effect of the Orr mechanism). 

The change of coordinates used is $(t,x,y)\to (t,z,v)$, where
 $z(t,x,y)=x-tv(t,y)$ and $v(t,y)$ satisfies
\beno
(\pa_t-\nu\pa_{yy})\left(t(v(t,y)-y)\right)=< U^x>(t,y),
\eeno
with initial data $\lim\limits_{t\to 0}t(v(t,y)-y)=0$. Where $< U^x>(t,y)=\f{1}{2\pi}\int_{\mathbb{T}}U^x(t,x,y)dx$. 

Define the following quantities
\begin{align}
C(t,v(t,y))&=v(t,y)-y,\label{eq: C}\\
v'(t,v(t,y))&=(\pa_yv)(t,y),\label{eq: v'}\\
v''(t,v(t,y))&=(\pa_{yy}v)(t,y),\label{eq: v''}\\
[\pa_tv](t,v(t,y))&=(\pa_tv)(t,y),\label{eq: v_t}\\
f(t,z(t,x,y),v(t,y))&=\om(t,x,y),\label{eq: om}\\
\phi(t,z(t,x,y),v(t,y))&=\psi(t,x,y),\label{eq: psi}\\
\tu(t,z(t,x,y),v(t,y))&=U^x(t,x,y).\label{eq: U^x}
\end{align}
Thus we get
\beq\label{eq: phi f}
\Delta_t\phi\eqdef\pa_{zz}\phi+(v')^2(\pa_v-t\pa_z)^2\phi+v''(\pa_v-t\pa_z)\phi=f,
\eeq
and
\beq\label{eq: f1}
\begin{split}
&\pa_tf+[\pa_tv]\pa_vf-\nu v''t\pa_zf
+v'\na_{z,v}^{\bot}P_{\neq}\phi\cdot\na_{z,v}f=\nu\Delta_tf,
\end{split}
\eeq
where $\na_{z,v}^{\bot}=(-\pa_v,\pa_z)$, $\na_{z,v}=(\pa_z,\pa_v)$ and $P_{\neq}\phi=\phi-< \phi>$, $\tu_0(t,v)=\f{1}{2\pi}\int_{\mathbb{T}_{2\pi}}\tu(t,z,v)dz$. 

We also obtain that 
\beq\label{eq: tu_0}
\pa_t\tu_0+[\pa_t v]\pa_v\tu_0+< v'\na^{\bot}_{z,v}P_{\neq 0}\phi\cdot\na \tu>=\nu\Delta_t\tu_0.
\eeq

Define the auxiliary function
\beno
g(t,v)=\f{1}{t}(\tu_0(t,v)-C(t,v)),
\eeno
which implies that
\begin{align*}
&[\pa_tv]=g+\nu v'',\\
&v'\pa_vC(t,v)=v'(t,v)-1,\\
&\pa_tC+[\pa_tv]\pa_vC=[\pa_tv],\\
&v'\pa_vv'=v''=\Delta_tC,
\end{align*}
and that $g$ satisfies
\beq\label{eq: g}
\pa_tg+\f{2g}{t}+g\pa_vg=-\f{v'}{t}< \na_{z,v}^{\bot}P_{\neq}\phi\cdot\na_{z,v}\tu>+\nu(v')^2\pa_{vv}g.
\eeq

If we denote $h=v'-1$, we get that
\ben
\pa_th+g\pa_v h=\f{-f_0-h}{t}+\nu\widetilde{\Delta}_th.
\een
Let $\bar{h}=\f{-f_0-h}{t}$, thus we obtain that
\ben
\pa_t\bar{h}+g\pa_v\bar{h}=-\f{2}{t}\bar{h}+\f{v'}{t}< \na_{z,v}^{\bot}P_{\neq}\phi\cdot\na_{z,v}f>+\nu\widetilde{\Delta}_t\bar{h}.
\een

It gives that
\beq\label{eq: f-new}
\pa_tf+u\cdot\na_{z,v}f=\nu\widetilde{\Delta}_tf,
\eeq
where 
\beno
u(t,z,v)=\left(\begin{aligned}0\\g\end{aligned}\right)+v'\na_{z,v}^{\bot}P_{\neq}\phi
\eeno 
and $\widetilde{\Delta}_tf=\pa_{zz}f+(v')^2(\pa_v-t\pa_z)^2f$.

By the changing of the coordinates we deduce our problem to studying the following system: 
\begin{align}\label{eq: main f}
&\left\{
\begin{aligned}
&\pa_tf+u\cdot\na_{z,v}f=\nu\widetilde{\Delta}_tf, \\
&u(t,z,v)=\left(\begin{aligned}0\\g\end{aligned}\right)+v'\na_{z,v}^{\bot}P_{\neq}\phi,\\
&\Delta_t\phi=f, \quad v''=v'\pa_vv', \quad h=v'-1,
\end{aligned}\right.\\
\label{eq: coor}
&\left\{
\begin{aligned}
&\pa_tg+\f{2g}{t}+g\pa_vg=-\f{v'}{t}< \na_{z,v}^{\bot}P_{\neq}\phi\cdot\na_{z,v}\tu>+\nu(v')^2\pa_{vv}g,\\
&\pa_t\bar{h}+\f{2}{t}\bar{h}+g\pa_v\bar{h}=\f{v'}{t}< \na_{z,v}^{\bot}P_{\neq}\phi\cdot\na_{z,v}f>+\nu(v')^2\pa_{vv}\bar{h},\\
&\pa_th+g\pa_v h=\bar{h}+\nu(v')^2\pa_{vv}h,\\
&\tu=-v'(\pa_v-t\pa_z)\phi.
\end{aligned}
\right.
\end{align}

\subsection{Main energy estimate}
In light of the previous section, our goal is to control solution to \eqref{eq: main f} and \eqref{eq: coor} uniformly in a suitable norm as $t\to \infty$. The key idea we use for this is the carefully designed time-dependent norm written as 
\beno
\left\|\As(t,\na)f\right\|_2^2=\sum_k\int_{\eta}\left|\As_k(t,\eta)\hat{f}_k(t,\eta)\right|^2d\eta,
\eeno
where $\As_k(t,\eta)$ is defined in \eqref{eq: Asigma}. 

We also introduce another time-dependent norm for $8\leq s\leq \s-10$,
\beno
\left\|A_E^{s}(t,\pa_k,\pa_v)f\right\|_{2}^2=\sum_{k\neq 0}\int_{\eta}\left|A_E^{s}(t,k,\eta)\hat{f}_k(t,\eta)\right|^2d\eta,
\eeno 
which quantifies the enhanced dissipation effect with 
\beno
A_E^{s}(t,k,\eta)=\langle k,\eta\rangle^sD(t,\eta),
\eeno
with 
\beno
D(t,\eta)=\f13\nu|\eta|^3+\f{1}{24}\nu(t^3-8|\eta|^3)_+. 
\eeno
Here $E$ stands for enhanced dissipation.

We define our higher Sobolev energy: 
\beq\label{eq: Energy}
\mathcal{E}^{\s}(t)=\f{1}{2}\left\|\As(t)f(t)\right\|_2^2+\mathcal{E}_{v}(t),
\eeq
where
\beq
\mathcal{E}_{v}(t)=\|g\|_{H^{\s}}^2+\nu^{\f13}\|h\|_{H^{\s}}^2+\nu^{\f13}\|\bar{h}\|_{H^{\s}}^2+\|h\|_{H^{\s-1}}^2+\|\bar{h}\|_{H^{\s-1}}^2
\eeq

By the well-posedness theory for 2D Navier-Stokes equation in Sobolev spaces we may safely ignore the time interval (say) $[0,1]$ by further restricting the size of the initial data. That is we have the following lemma. 
\begin{lemma}\label{Lem: LWP}
For $\ep>0$, $\nu>0$ and $\s\geq 40$, there exists $\ep'>0$ such that if $\|\om_{in}\|_{H^{\s}}\leq \ep'\nu^{\f13}$, then
\beno
\sup_{t\in [0,1]}\mathcal{E}^{\s}(t)\leq (\ep\nu^{\f13})^2. 
\eeno 
\end{lemma}

We define the following controls referred to in the sequel as the bootstrap hypotheses for $t\geq 1$. 

\no{\bf Higher regularity: main system}
\begin{equation}\label{eq: B1-Main-High}
\|\As f(t)\|_{2}^2
+\nu\int_1^t\left\|\sqrt{-\Delta_L}\As f(t')\right\|_{2}^2dt'
+\int_1^tCK_w(t')dt'\leq (8\ep\nu^{\f13})^2,
\end{equation}
where the $CK$ stands for 'Cauchy-Kovalevskaya' 
\beno
CK_w(t)=\sum_k\int \f{\pa_tw_k(t,\eta)}{w_k(t,\eta)}\left|\As_k(t,\eta)\hat{f}_k(t,\eta)\right|^2d\eta.
\eeno
\no{\bf Higher regularity: coordinate system}
\beq\label{eq: B1-Coo-High}
\begin{split}
&\langle t\rangle \|g\|_{H^{\s}}+\int_1^t\|g(t')\|_{H^{\s}}dt'\leq 8\ep\nu^{\f13},\\
&t^{3}\|\As\bar{h}(t)\|_2^2
+\int_1^tt'^{3}\left\|\sqrt{\f{\pa_tw}{w}}{\As\bar{h}}\right\|^2_2dt'\\
&\quad\quad \quad +\f14\int_1^tt'^{2}\|\As\bar{h}\|_2^2dt'
+\f14\nu\int_1^tt'^{3}\|\pa_v\As\bar{h}\|_2^2dt'
\leq 8\ep(\ep\nu^{\f16})^2,\\
&\|h(t)\|_{H^{\s}}^2+\nu\int_1^t\|\pa_vh(t')\|_{H^{\s}}^2dt'\leq 8(10\ep\nu^{\f16})^2
\end{split}
\eeq
\no{\bf Lower regularity: enhanced dissipation}
\beq\label{eq: B2-enhan}
\|A^s_Ef(t)\|_2^2+\f25\nu \int_1^t\left\|\sqrt{-\Delta_L}A^s_E f(t')\right\|_{2}^2dt'\leq  (8\ep\nu^{\f13})^2.
\eeq
\no{\bf Lower regularity: decay of zero mode}
\beq\label{eq: B2-zero}
\begin{split}
&\langle t\rangle^4\|g(t)\|_{H^{\s-6}}^2+\nu \int_1^tt'^4\|\pa_vg(t')\|_{H^{\s-6}}^2dt'\leq (8\ep\nu^{\f13})^2,\\
&\langle t\rangle^4\|\bar{h}(t)\|_{H^{\s-6}}^2+\nu \int_1^tt'^4\|\pa_v\bar{h}(t')\|_{H^{\s-6}}^2dt'\leq (8\ep\nu^{\f13})^2,\\
&\|f_0\|_{H^s}^2+\f{t\nu}{2}\|\pa_vf_0\|_{H^s}^2+
\nu\int_1^t\left(\|\pa_vf_0(t')\|_{H^s}^2+\f{t'\nu}{2}\|\pa_vf_0(t')\|_{H^s}^2\right)dt'\leq (8\ep\nu^{\f13})^2.
\end{split}
\eeq
\no{\bf Assistant estimates}
\beq\label{eq:assistant}
\begin{split}
&\langle t\rangle \|\bar{h}\|_{H^{\s-1}}+\int_1^t\|\bar{h}(t')\|_{H^{\s-1}}dt'\leq 8\ep\nu^{\f13},\\
&\|h(t)\|_{H^{\s-1}}^2+\nu\int_1^t\|\pa_vh(t')\|_{H^{\s-1}}^2dt'\leq 8(10\ep\nu^{\f13})^2.
\end{split}
\eeq

The following proposition follows from the bootstrap hypotheses, elliptic estimates and the properties of the multipliers: $\As$ and $A_E^s$. 
\begin{proposition}\label{prop: basic estimate}
Under the bootstrap hypotheses, the following inequalities hold: 
\begin{align}
\label{eq: high energy}&\|f\|_{H^{\s}}
+\nu^{\f12} \left\|\sqrt{-\Delta_L} f\right\|_{L^2_{T}(H^{\s})}+\left\|\sqrt{\f{\pa_tw}{w}}f\right\|_{L^2_{T}(H^{\s})}\lesssim \ep \nu^{\f13},\\
\label{eq: enha}&\|f_{\neq}\|_{H^s}+\nu^{\f12}\left\|\sqrt{-\Delta_L} f_{\neq}\right\|_{L^2_T(H^{s})}\lesssim \f{\ep \nu^{\f13}}{\langle \nu t^3\rangle},
\end{align}
and the inviscid damping results
\begin{align}\label{eq: inviscid}
&\|P_{\neq}\phi\|_{H^{\s-4}}\lesssim \f{\ep \nu^{\f13}}{\langle t^2\rangle},\quad
\|\tu_{\neq}\|_{H^{\s-3}}\lesssim \f{\ep \nu^{\f13}}{\langle t\rangle}. 
\end{align}
\end{proposition}
This proposition together with Lemma \ref{lem: composition} implies Theorem \ref{Thm: main}. 
\begin{proof}
By Lemma \ref{lem: total-growth}, we get 
$\As_{k}(t,\eta)\approx \langle k,\eta\rangle^{\s}$. Thus we have $\|\As f\|_2\approx \|f\|_{H^{\s}}$ which implies \eqref{eq: high energy}. 

By Lemma \ref{lem: D-D}, we get $D(t,\eta)\geq \nu t^3$, thus $\|A_E^sf\|_2\gtrsim \nu t^3\|f\|_{H^{s}}$ which gives \eqref{eq: enha}.  

The inviscid damping result \eqref{eq: inviscid} follows form Lemma \ref{lem: lin-inv-dam} and Lemma \ref{lem: tu}. 
\end{proof}

For the enhanced dissipation and the inviscid damping in Sobolev norm, we also have the following remark. 
\begin{remark}\label{Rmk: basic estimate}
Under the bootstrap hypotheses, it holds that
\beno
\|\om_{\neq}(t,x+ty+\Phi(t,y),y)\|_{H^s}\lesssim \f{\ep \nu^{\f13}}{\langle \nu t^3\rangle},
\eeno
and
\begin{align*}
&\left\|U^{y}(t,x+ty+\Phi(t,y),y)\right\|_{H^{\s-4}}\lesssim \f{\ep \nu^{\f13}}{\langle t^2\rangle},\\
&\left\|U^x_{\neq}(t,x+ty+\Phi(t,y),y)\right\|_{H^{\s-4}}\lesssim \f{\ep \nu^{\f13}}{\langle t\rangle}. 
\end{align*}
\end{remark}
Recall that 
\beno
&&f(t,z(t,x,y),v(t,y))=\om(t,x,y) \Rightarrow \om(t,x+ty+\Phi(t,y),y)=f(t,x,v(t,y)),\\
&&\tu(t,z(t,x,y),v(t,y))=U^x(t,x,y) \Rightarrow U^x(t,x+ty+\Phi(t,y),y)=\tu(t,x,v(t,y)),\\
&&\pa_z\phi(t,z(t,x,y),v(t,y))=U^{y}(t,x,y) \Rightarrow 
U^y(t,x+ty+\Phi(t,y),y)=(\pa_z\phi)(t,x,v(t,y)),
\eeno
The remark follows directly from \eqref{eq: enha}, \eqref{eq: inviscid}, the composition Lemma \ref{lem: composition} and the bootstrap hypotheses for the regularity of the coordinate system. \\

By Lemma 2.1, for the rest of the proof we may focus on times $t\geq 1$. Let $I^*$ be the connected set of times $t\geq 1$ such that the bootstrap hypotheses \eqref{eq: B1-Main-High}-\eqref{eq:assistant} are all satisfied. We will work on regularized solutions for which we know $\mathcal{E}^{\s}(t)$ takes values continuously in time, and hence $I^*$ is a closed interval $[1,T^*]$ with $T^*\geq 1$. The bootstrap is complete if we show that $I^*$ is also open, which is the purpose of the following proposition, the proof of which constitutes the majority of this work. 

\begin{proposition}\label{prop: bootstrap}
For $\s\geq 40$, $\nu>0$ and $8\leq s\leq \s-10$, there exist $0<\ep_0,\nu_0<1$, such that for all $0<\nu\leq \nu_0$ and $0<\ep\leq \ep_0$, such that if on $[1,T^*]$ the bootstrap hypotheses \eqref{eq: B1-Main-High}-\eqref{eq:assistant} hold, then for any $t\in [1,T^*]$, \\
1. Vorticity boundedness, 
\begin{align*}
&\|\As f(t)\|_{2}^2
+\nu\int_1^t\|\sqrt{-\Delta_L}\As f(t')\|_{2}^2dt'
+\int_1^tCK_w(t')dt'\leq (6\ep\nu^{\f13})^2,
\end{align*}
2. Control of coordinates system,
\begin{align*}
&\langle t\rangle \|g\|_{H^{\s}}+\int_1^t\|g(t')\|_{H^{\s}}dt'\leq 6\ep\nu^{\f13},\\
&t^{3}\|\As\bar{h}(t)\|_2^2
+\int_1^tt'^{3}\left\|\sqrt{\f{\pa_tw}{w}}{\bar{h}}\right\|^2_{H^{\s}}dt'\\
&\quad\quad \quad +\f14\int_1^tt'^{2}\|\As\bar{h}\|_2^2dt'
+\f14\nu\int_1^tt'^{3}\|\pa_v\As\bar{h}\|_2^2dt'
\leq 6\ep(\ep\nu^{\f16})^2,\\
&\|h(t)\|_{H^{\s}}^2+\nu\int_1^t\|\pa_vh(t')\|_{H^{\s}}^2dt'\leq 6(10\ep\nu^{\f16})^2.
\end{align*}
3. Enhanced dissipation,
\begin{align*}
&\|A^s_Ef(t)\|_2^2+\f25\nu \int_1^t\|\sqrt{-\Delta_L}A^s_E f(t')\|_{2}^2dt'\leq (6\ep\nu^{\f13})^2,
\end{align*}
4. Decay of zero mode, 
\begin{align*}
&\langle t\rangle^4\|g(t)\|_{H^{\s-6}}^2+\nu \int_1^tt'^4\|\pa_vg(t')\|_{H^{\s-6}}^2dt\leq (6\ep\nu^{\f13})^2,\\
&\langle t\rangle^4\|\bar{h}(t)\|_{H^{\s-6}}^2+\nu \int_1^tt'^4\|\pa_v\bar{h}(t')\|_{H^{\s-6}}^2dt\leq (6\ep\nu^{\f13})^2,\\
&\|f_0(t)\|_{H^s}^2+\f{t\nu}{2}\|\pa_vf_0\|_{H^s}^2+
\nu\int_1^t\left(\|\pa_vf_0(t)\|_{H^s}^2+\f{t'\nu}{2}\|\pa_vf_0(t')\|_{H^s}^2\right)dt'\leq (6\ep\nu^{\f13})^2,
\end{align*}
5. Assistant estimate,
\begin{align*}
&\langle t\rangle \|\bar{h}\|_{H^{\s-1}}+\int_1^t\|\bar{h}(t')\|_{H^{\s-1}}dt'\leq 6\ep\nu^{\f13},\\
&\|h(t)\|_{H^{\s-1}}^2+\nu\int_1^t\|\pa_vh(t')\|_{H^{\s-1}}^2dt'\leq 6(10\ep\nu^{\f13})^2,
\end{align*}
from which it follows that $T^*=+\infty$. 
\end{proposition}

The remainder of the paper is devoted to the proof of Proposition \ref{prop: bootstrap}, the primary step being to show that on $[1,T^*]$, we have the following estimates: 
\begin{align}
\nonumber\|\As f(t)\|_{2}^2
+\nu\int_1^t\|\sqrt{-\Delta_L}\As f(t')\|_{2}^2dt'
&+\int_1^tCK_w(t')dt'\label{eq: aim1}\\
&\leq 2\|\As f(1)\|_{2}^2+C\ep^3\nu^{\f23},\\
\label{eq: aim2}
\langle t\rangle \|g\|_{H^{\s}}
+\int_1^t\|g(t')\|_{H^{\s}}dt'
&\leq 2\|g(1)\|_{H^{\s}}+C\ep^2\nu^{\f13},\\
t^{3}\|\As\bar{h}(t)\|_2^2
+\int_1^tt'^{3}\left\|\sqrt{\f{\pa_tw}{w}}{\bar{h}}\right\|^2_{H^{\s}}dt'
&+\f14\int_1^tt'^{2}\|\As\bar{h}\|_2^2dt'\nonumber\\
+\f14\nu\int_1^tt'^{3}\|\pa_v\As\bar{h}\|_2^2dt'
&\leq 2\|\bar{h}(1)\|_{H^{\s}}+C\ep^4\nu^{\f13},\label{eq: aim3}\\
\|h(t)\|_{H^{\s}}^2+\nu\int_1^t\|\pa_vh(t')\|_{H^{\s}}^2dt'
&\leq 2\|{h}(1)\|_{H^{\s}} +C\ep^3\nu^{\f13},\label{eq: aim4}\\
\label{eq: aim5}
\|A^s_Ef(t)\|_2^2+\f25\nu \int_1^t|\sqrt{-\Delta_L}A^s_E f(t')\|_{2}^2dt'
&\leq 2\|A^s_Ef(1)\|_2^2+C\ep^3\nu^{\f23},\\
\label{eq: aim6}
\langle t\rangle^4\|g(t)\|_{H^{\s-6}}^2
+\nu \int_1^tt'^4\|\pa_vg(t')\|_{H^{\s-6}}^2dt'
&\leq 2\|g(1)\|_{H^{\s-6}}^2+ C\ep(\ep\nu^{\f13})^2,\\
\label{eq: aim7}
\langle t\rangle^4\|\bar{h}(t)\|_{H^{\s-6}}^2
+\nu \int_1^tt'^4\|\pa_v\bar{h}(t')\|_{H^{\s-6}}^2dt
&\leq 2\|\bar{h}(1)\|_{H^{\s-6}}^2+C\ep(\ep\nu^{\f13})^2,\\
\|f_0(t)\|_{H^s}^2+\f{t\nu}{2}\|\pa_vf_0\|_{H^s}^2+
\nonumber
\nu\int_1^t\Big(\|\pa_vf_0(t)\|_{H^s}^2&+\f{t'\nu}{2}\|\pa_vf_0(t')\|_{H^s}^2\Big)dt'\\
&\leq 2\|f_0(1)\|_{H^s}^2+\nu\|\pa_vf_0(1)\|_{H^s}^2+ C\ep^3\nu^{\f23},\label{eq: aim8}\\
\label{eq: aim9}
\langle t\rangle \|\bar{h}\|_{H^{\s-1}}+\int_1^t\|\bar{h}(t')\|_{H^{\s-1}}dt'
&\leq 2\|\bar{h}(1)\|_{H^{\s-1}}
+C\ep^2\nu^{\f13},\\
\label{eq: aim10}
\|h(t)\|_{H^{\s-1}}^2+\nu\int_1^t\|\pa_vh(t')\|_{H^{\s-1}}^2dt'
&\leq 2\|{h}(1)\|_{H^{\s-1}}^2+8\|\bar{h}\|_{L^1_T(H^{\s-1})}^2+C\ep^2\nu^{\f13}.
\end{align}
for some constant $C$ independent of $\ep,\nu$ and $T^*$. If $\ep$ is sufficiently small then \eqref{eq: aim1}-\eqref{eq: aim10} implies Proposition \ref{prop: bootstrap}. 

It is natural to compute the time evolution of the following quantities:
\beno
\mathcal{E}_{H,f}=\|\As f(t)\|_{2}^2,\quad 
\mathcal{E}_{H,g}=t\|g\|_{H^{\s}},\quad 
\mathcal{E}_{H,\bar{h}}=t^3\|\As\bar{h}\|_{H^{\s}}^2,\quad
\mathcal{E}_{H,h}=\|h(t)\|_{H^{\s}}^2,
\eeno 
and
\beno
\mathcal{E}_{L,\neq}=\|A^s_Ef(t)\|_2^2,\quad
\mathcal{E}_{L,g}=t^4\|g(t)\|_{H^{\s-6}}^2,\quad
\mathcal{E}_{L,\bar{h}}=t^4\|\bar{h}\|_{H^{\s-6}}^2,\quad
\mathcal{E}_{L,0}=\|f_0(t)\|_{H^s}^2+\f{t\nu}{2}\|\pa_vf_0\|_{H^s}^2,
\eeno
and
\beno
\mathcal{E}_{as,\bar{h}}=t^2\|\bar{h}\|_{H^{\s-1}}^2,\quad
\mathcal{E}_{as,h}=\|h(t)\|_{H^{\s-1}}^2,
\eeno
where $H$ stands for the highest regularity, $L$ stands for the lower regularity, $as$ stands for assistant. 

The most difficult part in the proof is to control the energy $\mathcal{E}_{H,f}$. 
Here we present the calculations of the time evolution of $\mathcal{E}_{H,f}$. 
The calculations of the time evolution of $\mathcal{E}_{H,g}$, $\mathcal{E}_{as,\bar{h}}$, $\mathcal{E}_{as,{h}}$, $\mathcal{E}_{H,\bar{h}}$ and $\mathcal{E}_{H,h}$ are in Section \ref{Sec: higher regular controls}. 
The calculations of the time evolution of $\mathcal{E}_{L,g}$ and $\mathcal{E}_{L,\bar{h}}$ are in Section \ref{Sec: Lower energy estimate}. 
The calculations of the time evolution of $\mathcal{E}_{L,\neq}$ and $\mathcal{E}_{L,0}$ are in Section \ref{Sec: Decay estimate of vorticity}. 

The rest part of this section will give an outline of the proof of \eqref{eq: aim1}. 

The proof of \eqref{eq: aim2} can be found in Section \ref{sec: g-high};

The proof of \eqref{eq: aim3} and \eqref{eq: aim4} can be found in Section \ref{sec: high-h-barh}; 

The proof of \eqref{eq: aim5} can be found in Section \ref{sec: f-low-enha}; 

The proof of \eqref{eq: aim6} can be found in Section \ref{sec: g-low};

The proof of \eqref{eq: aim7} can be found in Section \ref{sec: bar h};

The proof of \eqref{eq: aim8} can be found in Section \ref{sec: f-low-0};

The proof of \eqref{eq: aim9} and \eqref{eq: aim10} can be found in Section \ref{sec: barh and h-high}.

Form the time evolution of $\mathcal{E}_{H,f}$ we get
\beq\label{eq: dten}
\f{1}{2}\f{d}{dt}\int_{\mathbb{T}\times \R}\left|\As f(t)\right|^2dvdz
=-CK_{w}
-\int \As f\As(u\na f)dzdv
+\nu \int \As f\As\left(\tilde{\Delta}_tf\right)dzdv.
\eeq
where the $CK$ stands for 'Cauchy-Kovalevskaya' 
\ben
CK_w=\sum_k\int \f{\pa_tw_k(t,\eta)}{w_k(t,\eta)}\left|\As_k(t,\eta)\hat{f}_k(t,\eta)\right|^2d\eta.
\een

To treat the second term in \eqref{eq: dten}, we have
\begin{align*}
\int \As f\As(u\na_{z,v}f)dzdv=-\f12\int \na\cdot u|\As f|^2dvdz
+\int \As f\left[\As(u\cdot\na f)-u\cdot\na \As f\right]dzdv. 
\end{align*}
Notice that the relative velocity is not divergence free: 
\beno
\na\cdot u=\pa_vg+\pa_z\phi\pa_vv'=\pa_vg+\pa_zP_{\neq}\phi\pa_vh.
\eeno
The first term is controlled by the bootstrap hypothesis \eqref{eq: B2-zero}. For the second term we use the elliptic estimates, Lemma \ref{lem: lin-inv-dam}, which shows that under the bootstrap hypotheses we have
\beq
\|P_{\neq}\phi\|_{H^{\s-4}}\lesssim \f{\ep\nu^{\f13}}{\langle t\rangle^2}. 
\eeq
Therefore, by the Sobolev embedding, $\s>40$ and the bootstrap hypotheses,
\beq\label{eq: in-by-part-term}
\begin{split}
\left|\int \na\cdot u|\As f|^2dvdz\right|
&\lesssim \|\na u\|_{L^{\infty}}\|\As f\|_2^2\\
&\lesssim \big(\|g\|_{H^2}+(1+\|h\|_{H^{2}})\|P_{\neq}\phi\|_{H^3}\big)\|\As f\|_2^2\lesssim \f{\ep\nu^{\f13}}{\langle t\rangle^2}\|\As f\|_2^2.
\end{split}
\eeq
To handle the commutator, $\int \As f\left[\As(u\cdot\na f)-u\cdot\na \As f\right]dzdv$, we use a paraproduct decomposition. Precisely, we define three main contributors: transport, reaction and remainder: 
\ben\label{eq: com-bony}
\int \As f\left[\As(u\cdot\na f)-u\cdot\na \As f\right]dzdv
=\f{1}{2\pi} \sum_{N\geq 8}T_N+\f{1}{2\pi} \sum_{N\geq 8}R_N+\f{1}{2\pi}\mathcal{R},
\een
where 
\begin{align*}
&T_N=2\pi \int \As f\left[\As(u_{<N/8}\cdot \na f_N)-u_{<N/8}\cdot\na \As f_N\right]dzdv\\
&R_N=2\pi \int \As f\left[\As (u_{N}\cdot \na f_{<N/8})-u_{N}\cdot\na \As f_{<N/8}\right]dzdv\\
&\mathcal{R}=2\pi\sum_{N\in \mathbb{D}}\sum_{\f18N\leq N'\leq 8N} \int \As f\left[\As(u_{N}\cdot \na f_{N'})-u_{N}\cdot\na \As f_{N'}\right]dzdv
\end{align*}
Here $N\in \mathbb{D}=\{\f12,1,2,4,...,2^j,...\}$ and $g_N$ denote the $N$-th Littlewood-Paley projection and $g_{<N}$ means the Littlewood-Paley projection onto frequencies less than $N$. 

For the last term, we get
\begin{align}
\nonumber
\nu \int \As f\As \left(\tilde{\Delta}_tf\right)dzdv
&=\nu \int \As f\As \left({\Delta}_Lf\right)dzdv
-\nu \int \As f\As \left((1-(v')^2)(\pa_v-t\pa_z)^2f\right)dzdv\\
\nonumber
&=-\nu \left\|\sqrt{-\Delta_L}\As f\right\|_2^2\\
\nonumber
&\quad -\nu \int \As f_{\neq}\As \left((1-(v')^2)(\pa_v-t\pa_z)^2f_{\neq}\right)dzdv\\
\nonumber
&\quad -\nu \int \As f_{0}\As \left((1-(v')^2)\pa_v^2f_{0}\right)dv\\
&=-\nu \left\|\sqrt{-\Delta_L}\As f\right\|_2^2
+E^{\neq}+E^0. \label{eq: E_0andEneq}
\end{align} 

The next four propositions together with \eqref{eq: in-by-part-term} imply \eqref{eq: aim1}. At first, we deal with the dissipation term. In Section \ref{Sec: dissipation}, we will prove the following proposition. 
\begin{proposition}\label{prop: dissipation}
Under the bootstrap hypotheses, 
\beno
\nu \int_1^{t}\left(\int \As f\As \left(\tilde{\Delta}_tf\right)dzdv\right) dt'\leq -\f78\nu \int_1^{t} \left\|\sqrt{-\Delta_L}\As f(t')\right\|_2^2dt'+C\ep^3\nu^{\f23}.
\eeno
\end{proposition}

Next we control the transport part. In Section \ref{Sec: Transport}, we will prove the following proposition. 

\begin{proposition}\label{Prop: transport}
Under the bootstrap hypotheses, 
\beno
\int_1^t\sum_{N\geq 8}|T_N(t')|dt'
\lesssim \ep\sup_{t'\in[1,t]}\|\As f(t')\|_2^2. 
\eeno
\end{proposition}

Next we control the remainder part. In Section \ref{Sec: Remainder}, we will prove the following proposition. 

\begin{proposition}\label{Prop: Remainder}
Under the bootstrap hypotheses, 
\beno
|\mathcal{R}(t)| \lesssim \f{\ep\nu^{\f13}}{\langle t\rangle^2}\|\As f\|_2^2. 
\eeno
\end{proposition}

At last, we control the reaction part. In Section \ref{Sec: reaction}, we will prove the following proposition. 
\begin{proposition}\label{prop: reaction}
Under the bootstrap hypotheses, 
\beno
\int_1^t\sum_{N\geq 8}|R_N(t')|dt'
\lesssim \ep\sup_{t'\in[1,t]}\|\As f(t')\|_2^2+\ep \int_1^tCK_w(t')dt'+\ep^3\nu^{\f23}. 
\eeno
\end{proposition}

Let us admit the above propositions and finish the proof of \eqref{eq: aim1}. 
\begin{proof}
We then get by \eqref{eq: dten} that
\begin{align*}
&\|\As f(t)\|_2^2+2\int_1^tCK_{w}(t')dt'\\
&=\|\As f(1)\|_2^2-2\int_1^t\int \As f\As(u\na f)dzdvdt'
+\nu 2\int_1^t\int \As f\As\left(\tilde{\Delta}_tf\right)dzdvdt'\\
&\leq \|\As f(1)\|_2^2-\f{7}{4}\nu \int_1^{t} \left\|\sqrt{-\Delta_L}\As f(t')\right\|_2^2dt'+C\ep^3\nu^{\f23}\\
&\quad+C\int_1^t\bigg[\left|\int \na\cdot u|\As f|^2dvdz\right|
+\sum_{N\geq 8}|T_N(t')|+|\mathcal{R}(t')| +\sum_{N\geq 8}|R_N(t')|\bigg]dt'.
\end{align*}
Thus by \eqref{eq: in-by-part-term} and the above propositions, we have
\begin{align*}
&\|\As f(t)\|_2^2+2\int_1^tCK_{w}(t')dt'+\f{7}{4}\nu \int_1^{t} \left\|\sqrt{-\Delta_L}\As f(t')\right\|_2^2dt'\\
&\leq \|\As f(1)\|_2^2+C\ep^3\nu^{\f23}
+C\ep \sup_{t'\in [1,t]}\|\As f(t')\|_2^2
+C\ep \int_1^tCK_{w}(t')dt'.
\end{align*}
Thus by taking $\ep$ small enough, we proved \eqref{eq: aim1}. 
\end{proof}

\section{Toy model and the nonlinear growth}
\subsection{The toy model}
According to the change of coordinate, the relative velocity now is time integrable. The growth may come from the reaction term. In each time interval $I_{m,\eta}$ which contain only one Orr critical time $t=\f{\eta}{m}$, 
it is necessary to study the following toy model
\begin{align*}
\pa_t\widehat{f}(t,m,\eta)&+\nu(k^2+(\eta-mt)^2)\widehat{f}(t,m,\eta)\\
&=\int_{|\eta-\xi|\leq 1}\sum_{m-l=\pm 1}\f{\pm \eta}{l^2+(\eta-lt)^2}\widehat{f}(t,l,\eta)\widehat{f}(t,\pm 1,\xi-\eta)d\eta.\\
\pa_t\widehat{f}(t,m\pm 1,\eta)&+\nu((m\pm1)^2+(\eta-(m\pm1)t)^2)\widehat{f}(t,m,\eta)\\
&=\int_{|\eta-\xi|\leq 1}\f{\eta}{m^2+(\eta-mt)^2}\widehat{f}(t,m,\eta)\widehat{f}(t,\pm 1,\xi-\eta)d\eta.
\end{align*}
Since the $\widehat{f}(t,\pm 1,\xi-\eta)$ is restricted to the lower frequency $|\xi-\eta|\leq 1$, we can regard it as a constant in $\xi$ valuable. Moreover, $\widehat{f}(t,\pm 1,\xi-\eta)$ also has enhanced dissipation. 
As $t\in I_{m,\eta}$, $(m\pm1)^2+(\eta-(m\pm1)t)^2\approx \f{\eta^2}{m^2}$, thus we deduce to the following simplified toy model
\begin{align}
\pa_t\widehat{f}(t,m,\eta)
+\nu(m^2+(\eta-mt)^2)\widehat{f}(t,m,\eta)
&=\f{\ka e^{-c\nu^{\f13}t}m^2}{|\eta|}\widehat{f}(t,m\pm 1,\eta),\label{eq: toy-1}\\
\pa_t\widehat{f}(t,m\pm 1,\eta)
+\f{\nu\eta^2}{m^2}\widehat{f}(t,m\pm 1,\eta)
&=\f{\ka|\eta| e^{-c\nu^{\f13}t}}{m^2(1+(\f{\eta}{m}-t)^2)}\widehat{f}(t,m,\eta)\label{eq: toy-2}.
\end{align}
where $\ka$ stands for the smallness assumption of the initial data. Our goal is to find the largest $\ka$ such that we can control the total growth caused by the toy model. Thus we assume the enhanced dissipation is $e^{-c\nu^{\f13}t}$. 

The next step of simplification bases on the following observations: 
\begin{itemize}
\item When $t\gg \nu^{-\f13}$, the enhanced dissipation will offer a small coefficient which makes the Orr mechanism weaker. So we focus on the time region $t\lesssim \nu^{-\f13}$. The resonant time region is $I_t(\eta)\approx [\sqrt{|\eta|},2|\eta|]$. So we are interested in the case $|\eta|\lesssim \nu^{-\f23}$ so that $I_t(\eta)\cap [1,C\nu^{-\f13}]\neq \emptyset$. During this time region $e^{-c\nu^{\f13}t}\approx 1$. 
\item The rapid growth of $\widehat{f}(t,m\pm 1,\eta)$ happens when $|t-\f{\eta}{m}|\approx 1$. 
\item The coefficient in front of $\widehat{f}(t,m,\eta)$ on the right hand side of \eqref{eq: toy-2} is much bigger than the coefficient in front of $\widehat{f}(t,m\pm 1,\eta)$ on the right hand side of \eqref{eq: toy-1}. It means that $\widehat{f}(t,m\pm 1,\eta)$ will grow faster than $\widehat{f}(t,m,\eta)$. We may replace $\widehat{f}(t,m,\eta)$ by $\widehat{f}(t,m\pm 1,\eta)$ in the second equation. 
\item Since $|\eta|\lesssim \nu^{-\f23}$, when $|t-\f{\eta}{m}|\approx 1$, the dissipation coefficient in \eqref{eq: toy-2} $\f{\nu\eta^2}{m^2}\lesssim \f{\nu^{\f13}\eta}{m^2}$ is not bigger than the coefficient of the right hand sider if $\ka\approx \nu^{\f13}$. Thus we can remove the dissipation term. 
\end{itemize}

Thus we deduce to the following toy model 
\begin{align*}
\pa_t\widehat{f}(t,m\pm 1,\eta)
=\f{\ka|\eta| e^{-c\nu^{\f13}t}}{m^2(1+(\f{\eta}{m}-t)^2)}\widehat{f}(t,m\pm1,\eta).
\end{align*}

For $t\in I_{m,\eta}$, let $\tau=t-\f{\eta}{m}$, then $\tau\in [-D_{m,\eta}^-,D_{m,\eta}^+]$ where $D_{m,\eta}^-=\f{\eta}{(2m+1)m}=\f{\eta}{m}-t_{m,\eta}$ and $D_{m,\eta}^+=\f{\eta}{(2m-1)m}=t_{m-1,\eta}-\f{\eta}{m}$ for $m\geq 1$, then $D_{m,\eta}^{\pm}\approx  \f{\eta}{m^2}$. 

At last we use the following model to control the entropy growth in each critical time region. 
\beq\label{eq: toy model}
\left\{\begin{aligned}
&\pa_{\tau}g_m= \langle \nu^{\f13}t_{m,\eta}\rangle^{-(1+\b)}\f{\nu^{\f13}\f{\eta}{m^2}}{1+\tau^2}g_m,\\
&g_m(-D_{m,\eta}^-)=1.
\end{aligned}\right.
\eeq
We need to point out that in the toy model $e^{-c\nu^{\f13}t}$ is replaced by $\langle \nu^{\f13}t_{m,\eta}\rangle^{-(1+\b)}$ with $0<\b\leq \f12$ due to some technical reasons when we deal with zero mode (see \eqref{eq: beta<1/2}). The condition $\b>0$ ensures the total growth is bounded (see Lemma \ref{lem: total-growth}). 
 
For $m\eta>0$ and $|m|\in [1,E(\sqrt{|\eta|})]$, with $|\eta|\geq 3$, we define for $0<\b\leq \f12$, 
\beq\label{eq: g_m}
g_m(\tau,\eta)=\exp\Big( \langle\nu^{\f13}t_{m,\eta}\rangle^{{-(1+\b)}}\f{\nu^{\f13}\eta}{m^2}\big(\arctan(\tau)+\arctan(D_{m,\eta}^-)\big)\Big),
\eeq
then $g_m$ solves \eqref{eq: toy model}.

Then we have 
\beno
g_m(D_{m,\eta}^+,\eta)=G_{m}(\eta)\underbrace{g_m(-D_{m,\eta}^-,\eta)}_{=1}.
\eeno
with 
\beno
G_{m}(\eta)=\exp\Big( \langle \nu^{\f13}t_{m,\eta}\rangle^{-(1+\b)}\f{\nu^{\f13}\eta}{m^2}\big(\arctan(D_{m,\eta}^+)+\arctan(D_{m,\eta}^-)\big)\Big)
\eeno
Otherwise, we let $g_{m}(\tau,\eta)=1$.

\subsection{Key multiplier}
In this subsection, we will define the key multiplier which govern the growth. 

We define $w(t,\eta)$ in the following way: 
\begin{itemize}
\item For $t\leq t(\eta)$, $w(t,\eta)=1$;
\item For $t\in I_{j,\eta}$ with $|j|\in [1, E(\sqrt{|\eta|})]$ and $j\eta>0$, $w(t,\eta)=w(t_{j,\eta},\eta)g_{j}(t-\f{\eta}{j},\eta)$;
\item For $t\geq 2|\eta|$, $w(t,\eta)=w(2|\eta|,\eta)$.
\end{itemize}

According the definition of $g_{m}$, we get
\beq\label{eq: w_t/w}
\f{\pa_tw(t,\eta)}{w(t,\eta)}\approx \f{\langle \nu^{\f13}t\rangle^{-(1+\b)}\nu^{\f13}\f{\eta}{m^2}}{1+(t-\f{\eta}{m})^2}1_{t\in I_{m,\eta}}\approx \f{\langle \nu^{\f13}t\rangle^{-\b}m^{-1}}{1+(t-\f{\eta}{m})^2}1_{t\in I_{m,\eta}}
\eeq

Next for $m\eta>0$ and $|m|\in [1,E(\sqrt{|\eta|})]$, with $|\eta|\geq 3$, we will construct a continuous function $\varrho(m,\eta)$ taking value almost as $\f{m}{|m|}\max\{|m|,|\eta|\}$. First let $\rho(x)$ be a bounded smooth function such that 
\beq\label{eq: rho}
\rho(x)=\left\{\begin{aligned}
&1, \quad x\geq \f{1}{10},\\
&\text{ smooth connected}, \quad x\in \big[\f{1}{20}, \f{1}{10}\big],\\
&0, \quad x\leq \f{1}{20},
\end{aligned}\right.
\eeq
We also let $\rho$ satisfy 
\beno
\int_{\f{1}{20}}^{\f{1}{10}}\rho(x)dx=\f{1}{20}. 
\eeno
Let $\rho_k(x)=\rho(\f{x}{k})$ and 
\beq
w_{k}(t,\eta)=w\left(t, \varrho(k,\eta)\right).
\eeq
where 
\beq\label{eq: rho(k,eta)}
\varrho(k,\eta)=\left\{
\begin{aligned}
&\f{k}{20}+\int_{0}^{\eta}\rho_k(x)dx, \quad k\neq 0, \\
&\eta, \quad k=0.
\end{aligned}\right.
\eeq 
Then we get that for $|\eta|\leq \f{|k|}{20}$, $\varrho(k,\eta)=\f{k}{20}$ and 
$w_{k}(t,\eta)=w\left(t,\f{k}{20}\right)$; 
and for $|\eta|\geq \f{|k|}{10}$, $\varrho(k,\eta)=\eta$ and $w_{k}(t,\eta)=w(t,\eta)$. 

\begin{lemma}\label{lem: rho}
It holds that 
\beno
\varrho(k,\eta)\approx \langle k,\eta\rangle.
\eeno
For $|k-l,\xi-\eta|\leq \f{1}{100}|l,\xi|$, it holds that
\beno
|\varrho(k,\eta)-\varrho(l,\xi)|\lesssim |k-l,\xi-\eta|.
\eeno
\end{lemma}
\begin{proof}
It is easy to obtain that $\varrho(k,\eta)\lesssim \langle k,\eta\rangle$. 
The lower bound follows from the fact that for $\f{|k|}{20}\leq |\eta|\leq \f{|k|}{10}$, $\varrho(k,\eta)\gtrsim \f{|k|}{20}\geq \f{\eta}{2}$. 

If $|\xi|\geq |l|$, then $|k-l,\xi-\eta|\leq \f{1}{50}|\xi|$, 
$|\eta|\geq \f{49}{50}|\xi|$ and $|k|\leq |k-l|+|l|\leq \f{51}{50}|\xi|\leq 2|\eta|$. Thus 
\beno
|\varrho(k,\eta)-\varrho(l,\xi)|=|\xi-\eta|.
\eeno
If $|\xi|\leq \f{|l|}{100}$, then $|k-l,\xi-\eta|\leq \f{101}{10000}|l|$, $|k|\geq \f{9899}{10000}|l|$ and $|\eta|\leq |\xi|+|\xi-\eta|\leq \f{201}{10000}|l|\leq \f{|k|}{20}$. Thus 
\beno
|\varrho(k,\eta)-\varrho(l,\xi)|=\f{1}{20}|k-l|.
\eeno
Then we only need to focus on $|\xi|\approx |l|\approx |\eta|\approx |k|$. Thus
\begin{align*}
|\varrho(k,\eta)-\varrho(l,\xi)|
&=\left|\f{k}{20}-\f{l}{20}+k\int_{0}^{\f{\eta}{k}}\rho(x)dx-l\int_{0}^{\f{\xi}{l}}\rho(x)dx\right|\\
&\lesssim |k-l|+|k-l|\int_{0}^{\f{\eta}{k}}\rho(x)dx+\left|l\int_{\f{\eta}{k}}^{\f{\xi}{l}}\rho(x)dx\right|\\
&\lesssim |k-l|+\f{|\eta l-\xi k|}{|k|}
\lesssim |k-l|+\f{|\eta,k||k-l,\eta-\xi|}{|k|}\\
&\lesssim |k-l,\xi-\eta|.
\end{align*}
Thus we proved the lemma. 
\end{proof}

With $w_{k}(t,\eta)$, now we can define our key multiplier $\As_k(t,\eta)$, 
\beq\label{eq: Asigma}
\As_k(t,\eta)=\f{\langle k,\eta\rangle^{\s}}{w_k(t,\eta)}.
\eeq

\subsection{Basic estimate for the multiplier}
The following lemma expresses the well-separation of critical times. 
\begin{lemma}[\cite{BM1}]\label{lem: 3.2}
Let $\xi,\eta$ be such that there exists some $\al\geq 1$ with $\al^{-1}|\xi|\leq |\eta|\leq \al |\xi|$ and let $k,n$ be such that $t\in I_{k,\eta}\cap I_{n,\xi}$, then $k\approx n$ and moreover at least one of the following holds:  
\begin{itemize}
\item [(a)] $k=n$; 
\item [(b)] $|t-\f{\eta}{k}|\geq \f{1}{10\al}\f{\eta}{k^2}$ and $|t-\f{\xi}{n}|\geq \f{1}{10\al}\f{\xi}{n^2}$;
\item [(c)] $|\eta-\xi|\gtrsim_{\al} \f{|\eta|}{|n|}$. 
\end{itemize}
\end{lemma}
Now we will present a lemma about the upper and lower bounds estimates of $w(t,\eta)$. 

\begin{lemma}\label{lem: total-growth}
It holds that
\begin{align*}
w(t,\eta)\approx 1.
\end{align*}
As a consequence, $A^{\s}_k(t,\eta)\approx \langle k,\eta\rangle^{\s}$. 
\end{lemma}
\begin{proof}
We have for any $t,\eta$, 
\begin{align*}
1&\leq w(t,\eta)
\leq \prod_{m=E(\sqrt{|\eta|})}^{1}G_{m}(\eta)
\leq \exp\left(\sum_{m=E(\sqrt{|\eta|})}^{1}\f{\pi\nu^{\f13}\eta}{m^2}(1+\nu^{\f13}t_{m,\eta})^{-(1+\b)}\right)\\
&\lesssim \left\{\begin{aligned}
&\exp\left(\sum_{m=E(\sqrt{|\eta|})}^{1}\f{\nu^{\f13}\eta}{m^2}\right)\quad \nu^{\f13}\eta\leq 1\\
&\exp\left(\sum_{m=\nu^{\f13}\eta}^{E(\sqrt{|\eta|})}\f{\nu^{\f13}\eta}{m^2}
+\sum_{m=\nu^{\f13}\eta}^{1}\f{m^{-1+\b}}{(\nu^{\f13}\eta)^{\b}}\right)\quad 1\leq \nu^{\f13}\eta\leq E(\sqrt{|\eta|})\\
&\exp\left(\sum_{m=E(\sqrt{|\eta|})}^{1}\f{m^{-1+\b}}{(\nu^{\f13}\eta)^{\b}}\right)\quad 1\leq E(\sqrt{|\eta|})\leq \nu^{\f13}\eta
\end{aligned}\right.\\
&\lesssim  \left\{\begin{aligned}
\nu^{\f13}\eta\lesssim &1,\quad |\eta|\lesssim \nu^{-\f13},\\
&1, \quad  \nu^{-\f13}\lesssim |\eta|\lesssim \nu^{-\f23},\\
\f{1}{(\nu^{\f13}\sqrt{|\eta|})^{\b}}\lesssim &1, \quad \nu^{-\f23}\lesssim|\eta|.
\end{aligned}\right.
\end{align*}
Thus we proved the lemma. 
\end{proof}
The above lemma gives that for all $t$,
\beq\label{eq: As}
\As_{k}(t,\eta)\approx \langle k,\eta\rangle^{\s}.
\eeq

Next we introduce several lemma related to the properties of $D$. The first lemma can be found in \cite{BMV} which will be useful in the proof of the commutator estimate in Section \ref{Sec: Decay estimate of vorticity}. 
\begin{lemma}[\cite{BMV}]\label{lem: D-D}
Uniformly in $\nu$, $\eta,\xi$ and $t\geq 1$ we have: 
\beno
D(t,\eta)\gtrsim \nu \max\{|\eta|^3,t^3\},
\eeno
and
\beno
\f{D(t,\xi)}{D(t,\eta)}\lesssim \langle \eta-\xi\rangle^3,\quad
|D(t,\xi)-D(t,\eta)|\lesssim \f{D(t,\xi)}{\langle \xi\rangle+\langle \eta\rangle}\langle \eta-\xi\rangle^3
\eeno
\end{lemma}

Next lemma we will introduce the product lemma related to $D$ which is a Sobolev type estimates comparing to the Lemma 3.7 in \cite{BMV}. 
\begin{lemma}\label{lem: product}
The following holds for all $q^1$ and $q^2$ and $\g>1$, 
\beno
\|D(q^1q^2)\|_{H^{\g}}\lesssim \|q^1\|_{H^{\g+3}}\|Dq^2\|_{H^{\g}\g},
\eeno
and
\beno
\|D(\na^{\bot} q^1\cdot \na q^2)\|_{H^{\g}}\lesssim \|q^1\|_{H^{\g+5}}\|Dq^2\|_{H^{\g}}+\|D q_1\|_{H^{\g}}\|q_2\|_{H^{\g+5}}
\eeno
\end{lemma} 
\begin{proof}
We use the dual method. By Lemma \ref{lem: D-D}, we get
\begin{align*}
&\|D(q^1q^2)\|_{H^{\g}}
=\|\langle \na\rangle^{\g}D(q^1q^2)\|_{L^2}\\
&=\sup_{\|\varphi\|_{L^2}=1}
\left|\sum_{k,l}\int_{\eta,\xi}\hat{\varphi}_{k}(\eta)
\langle k,\eta \rangle^{\g}D(\eta)\hat{q}^1_{k-l}(\eta-\xi)\hat{q}^2_{l}(\xi)d\xi d\eta\right|\\
&\lesssim \sup_{\|\varphi\|_{L^2}=1}
\sum_{k,l}\int_{\eta,\xi}|\hat{\varphi}_{k}(\eta)|
\langle k,\eta \rangle^{\g}\langle\xi-\eta\rangle^3 |\hat{q}^1_{k-l}(\eta-\xi)||D(\xi)\hat{q}^2_{l}(\xi)|d\xi d\eta\\
&\lesssim \sup_{\|\varphi\|_{L^2}=1}
\sum_{k,l}\int_{\eta,\xi}1_{|k-l,\eta-\xi|\leq |l,\xi|}|\hat{\varphi}_{k}(\eta)|
\langle\xi-\eta\rangle^3 |\hat{q}^1_{k-l}(\eta-\xi)||D(\xi)\langle l,\xi \rangle^{\g}\hat{q}^2_{l}(\xi)|d\xi d\eta\\
&\quad+\sup_{\|\varphi\|_{L^2}=1}
\sum_{k,l}\int_{\eta,\xi}1_{|k-l,\eta-\xi|>|l,\xi|}|\hat{\varphi}_{k}(\eta)|
\langle k-l,\eta-\xi \rangle^{\g}\langle\xi-\eta\rangle^3 |\hat{q}^1_{k-l}(\eta-\xi)||D(\xi)\hat{q}^2_{l}(\xi)|d\xi d\eta\\
&\lesssim \|\varphi\|_{L^2}\|Dq^2\|_{H^{\g}}\|q^1\|_{H^{\g+3}}.
\end{align*}
The last inequality we use the fact that $\|\widehat{q}\|_{L^{1}}\lesssim \|\langle k,\eta\rangle^{\g}\widehat{q}\|_{L^2}\|\langle k,\eta\rangle^{-\g}\|_{L^2}\lesssim \|q\|_{H^{\g}}$ for $\g>1$. 

We also have
\begin{align*}
&\|D(\na^{\bot}q^1\cdot\na q^2)\|_{H^{\g}}\\
&=\sup_{\|\varphi\|_{L^2}=1}
\left|\sum_{k,l}\int_{\eta,\xi}\hat{\varphi}_{k}(\eta)
\langle k,\eta \rangle^{\g}D(\eta)\hat{q}^1_{k-l}(\eta-\xi)\hat{q}^2_{l}(\xi)(-\eta+\xi,k-l)\cdot (l,\xi)d\xi d\eta\right|\\
&\lesssim \sup_{\|\varphi\|_{L^2}=1}
\sum_{k,l}\int_{\eta,\xi}1_{|k-l,\eta-\xi|\leq |l,\xi|}|\hat{\varphi}_{k}(\eta)|
 |D(\eta-\xi)\hat{q}^1_{k-l}(\eta-\xi)||\langle\xi\rangle^3\langle l,\xi \rangle^{\g+2}\hat{q}^2_{l}(\xi)|d\xi d\eta\\
&\quad+\sup_{\|\varphi\|_{L^2}=1}
\sum_{k,l}\int_{\eta,\xi}1_{|k-l,\eta-\xi|>|l,\xi|}|\hat{\varphi}_{k}(\eta)|
\langle k-l,\eta-\xi \rangle^{\g+2}\langle\xi-\eta\rangle^3 |\hat{q}^1_{k-l}(\eta-\xi)||D(\xi)\hat{q}^2_{l}(\xi)|d\xi d\eta\\
&\lesssim \|\varphi\|_{L^2}\|Dq^2\|_{H^{\g}}\|q^1\|_{H^{\g+5}}+\|\varphi\|_{L^2}\|Dq^1\|_{H^{\g}}\|q^2\|_{H^{\g+5}}.
\end{align*}
Thus we proved the lemma. 
\end{proof}

\section{Elliptic estimate}\label{Elliptical estimate}
The purpose of this section is to provide a thorough analysis of $\Delta_t$. 

\begin{lemma}\label{lem: lin-inv-dam}
Under the bootstrap hypotheses, for $\nu$ sufficiently small and $s'\in [0,2]$, it holds that for $2\leq \g\leq \s-1$
\beno
\|P_{\neq}\phi\|_{H^{\g-s'}}\lesssim \f{1}{\langle t\rangle^{s'}}\|\langle\pa_z\rangle^{-s'}f_{\neq}\|_{H^{\g}},
\eeno
and for $\gamma\leq \s-1$
\beno
\|\Delta_L\Delta_t^{-1}P_{\neq}f\|_{H^{\gamma}}=\|\Delta_LP_{\neq}\phi\|_{H^{\gamma}}\lesssim \|P_{\neq}f\|_{H^{\gamma}}.
\eeno
\end{lemma}
\begin{proof}
We get that for $s'\in [0,2]$ and $s\geq 0$ that
\begin{align}
\|P_{\neq}\phi\|_{H^{s}}^2
\nonumber&=\sum_{k\neq 0}\int_{\eta}\langle k,\eta\rangle^{2s} |\hat{\phi}(k,\eta)|^2d\eta\\
\nonumber&\leq \sum_{k\neq 0}\int_{\eta}\f{\langle k,\eta\rangle^{2s}\langle \f{\eta}{k}\rangle^{2s'}}{\langle \f{\eta}{k}\rangle^{2s'}(k^2+(\eta-kt)^2)^2} |\widehat{\Delta_L\phi}(k,\eta)|^2d\eta\\
\nonumber&\lesssim \sum_{k\neq 0}\int_{\eta}\f{\langle k,\eta\rangle^{2s+2s'}}{k^{2s'}(1+t^2)^{s'}} |\widehat{\Delta_L\phi}(k,\eta)|^2d\eta\\
&\lesssim \f{1}{(1+t^2)^{s'}}\|\langle\pa_z\rangle^{-s'}{\Delta_LP_{\neq}\phi}\|_{H^{s+s'}}^2.\label{eq: Lin-inv-dam}
\end{align}
We write $\Delta_t$ as a perturbation of $\Delta_L$ via
\beno
\Delta_LP_{\neq}\phi=P_{\neq}f
+(1-(v')^2)(\pa_v-t\pa_z)^2P_{\neq}\phi
-v''(\pa_v-t\pa_z)P_{\neq}\phi.
\eeno
Thus we get 
\begin{align*}
\|\Delta_LP_{\neq}\phi\|_{H^{\g}}
&\leq \|P_{\neq}f\|_{H^{\g}}
+C\|(1-(v')^2)(\pa_v-t\pa_z)^2P_{\neq}\phi\|_{H^{\g}}
+C\|v''(\pa_v-t\pa_z)P_{\neq}\phi\|_{H^{\g}}.
\end{align*}
then by using the fact that $v''=(h+1)\pa_vh$, \eqref{eq: prod1} and the bootstrap hypotheses, we get
\begin{align*}
\|\Delta_LP_{\neq}\phi\|_{H^{\g}}
&\leq \|P_{\neq}f\|_{H^{\g}}+
C\|h\|_{H^{\s-1}}(1+\|h\|_{\s-1})\|\Delta_LP_{\neq}\phi\|_{H^{\g}}\\
&\quad+C(1+\|h\|_{\s-1})\|h\|_{H^{\s}}\|\Delta_LP_{\neq}\phi\|_{H^{\g}}\\
&\lesssim \|P_{\neq}f\|_{H^{\g}}+C\ep\nu^{\f16}\|\Delta_LP_{\neq}\phi\|_{H^{\g}}
\end{align*}
which implies $\|\Delta_LP_{\neq}\phi\|_{H^{\g}}\lesssim \|P_{\neq}f\|_{H^{\g}}$. The lemma follows from \eqref{eq: Lin-inv-dam} with $s=\s-2-s'$.
\end{proof}

As $(1-(v')^2)$ and $v''$ are zero mode, by the same argument as the proof, we can easily get that for $\g\leq \s-1$
\beq\label{eq: ellip-z}
\|\langle\pa_z\rangle^{\s-\g}\langle\pa_v\rangle^{\g}\Delta_{L}\Delta_t^{-1}f_{\neq}\|_{2}\lesssim \|f_{\neq}\|_{H^{\s}}\lesssim \|\As f\|_2.
\eeq

\begin{lemma}\label{lem: tu}
Under the bootstrap hypotheses, it holds that
\beno
\|\na_{L}P_{\neq}\phi\|_{H^{\s-2}}+\|\tu_{\neq}\|_{H^{\s-2}}\lesssim \f{1}{\langle t\rangle }\|f_{\neq}\|_{H^{\s-1}},
\eeno
and $\g\leq \s-1$
\beno
\|\na_{L}\tu_{\neq}\|_{H^{\g}}\lesssim \|f_{\neq}\|_{H^{\g}}. 
\eeno
\end{lemma}
\begin{proof}
By the definition of $\tu$ we get 
\begin{align*}
\tu_{\neq}=-(1+h)(\pa_v-t\pa_z)P_{\neq}\phi. 
\end{align*}
Here we use the same argument as \eqref{eq: Lin-inv-dam} and get that
\begin{align}\label{eq: Lin-inv-dam2}
\|(\pa_v-t\pa_z)P_{\neq}\phi\|_{H^{s}}^2
\nonumber&=\sum_{k\neq 0}\int_{\eta}\langle k,\eta\rangle^{2s}|\eta-kt|^2 |\hat{\phi}(k,\eta)|^2d\eta\\
\nonumber&\leq \sum_{k\neq 0}\int_{\eta}
\f{\langle k,\eta\rangle^{2s}\langle \f{\eta}{k}\rangle^{2}
|\eta-kt|^2}
{\langle \f{\eta}{k}\rangle^{2}(k^2+(\eta-kt)^2)^2} |\widehat{\Delta_L\phi}(k,\eta)|^2d\eta\\
\nonumber&\leq \sum_{k\neq 0}\int_{\eta}
\f{\langle k,\eta\rangle^{2s}\langle \f{\eta}{k}\rangle^{2}}
{\langle \f{\eta}{k}\rangle^{2}(k^2+(\eta-kt)^2)} |\widehat{\Delta_L\phi}(k,\eta)|^2d\eta\\
\nonumber&\lesssim \sum_{k\neq 0}\int_{\eta}\f{\langle k,\eta\rangle^{2s+2}}{k^{2}(1+t^2)} |\widehat{\Delta_L\phi}(k,\eta)|^2d\eta\\
&\lesssim \f{1}{1+t^2}\|\langle\pa_z\rangle^{-1}{\Delta_LP_{\neq}\phi}\|_{H^{s+1}}^2.
\end{align}
Then by Lemma \ref{lem: lin-inv-dam} and the bootstrap hypotheses, we have 
\begin{align*}
\|\tu_{\neq}\|_{H^{\s-2}}
&\lesssim (1+\|h\|_{H^{\s-2}})\|(\pa_v-t\pa_z)P_{\neq}\phi\|_{H^{\s-2}}\\
&\lesssim \f{(1+\|h\|_{H^{\s-2}})}{\langle t\rangle}\|\Delta_LP_{\neq}\phi\|_{H^{\s-1}}
\lesssim \f{1}{\langle t\rangle}\|f_{\neq}\|_{H^{\s-1}}.
\end{align*}
The first inequality follows from \eqref{eq: Lin-inv-dam2} with $s=\s-2$. 

We also have
\beno
&&\pa_z\tu_{\neq}=-(1+h)(\pa_v-t\pa_z)\pa_zP_{\neq}\phi,\\
&&(\pa_v-t\pa_z)\tu_{\neq}=-(1+h)(\pa_v-t\pa_z)^2P_{\neq}\phi-\pa_vh(\pa_v-t\pa_z)P_{\neq}\phi.
\eeno
Therefore by Lemma \ref{lem: lin-inv-dam} and the bootstrap hypotheses, we get
\begin{align*}
\|\na_{L}\tu_{\neq}\|_{H^{\g}}
\lesssim (1+\|h\|_{H^{\s-1}})\|\Delta_{L}P_{\neq}\phi\|_{H^{\g}}\lesssim \|f_{\neq}\|_{H^{\g}}.
\end{align*}
Thus we proved the lemma.  
\end{proof}

\begin{lemma}\label{lem: elliptical highest}
Under the bootstrap hypotheses, it holds that
\beno
\|\na_{L}\tu_{\neq}\|_{H^{\s}}\lesssim \|\Delta_{L}\Delta_t^{-1}f_{\neq}\|_{H^{\s}}\lesssim \|f_{\neq}\|_{H^{\s}}+\f{\ep\nu^{\f13}}{\langle t\rangle \langle \nu t^3\rangle}\|\pa_vh\|_{H^{\s}},
\eeno
and
\beno
\left\|\sqrt{\f{\pa_tw}{w}}\chi_R\Delta_{L}\Delta_t^{-1}f_{\neq}\right\|_{H^{\s}}\lesssim \left\|\sqrt{\f{\pa_tw}{w}}f_{\neq}\right\|_{H^{\s}}+\f{\ep^2\nu^{\f12}}{\langle \nu t^3\rangle}.
\eeno
\end{lemma}
\begin{proof}
We have 
\beno
\Delta_{L}\Delta_t^{-1}f_{\neq}=\Delta_LP_{\neq}\phi=P_{\neq}f
+(1-(v')^2)(\pa_v-t\pa_z)^2P_{\neq}\phi
-v''(\pa_v-t\pa_z)P_{\neq}\phi.
\eeno
Thus we get 
\begin{align*}
\|\Delta_{L}\Delta_t^{-1}f_{\neq}\|_{H^{\s}}
&\lesssim \|f_{\neq}\|_{H^{\s}}
+\|(1-(v')^2)\|_{H^3}\|(\pa_v-t\pa_z)^2P_{\neq}\phi\|_{H^{\s}}\\
&\quad+\|(1-(v')^2)\|_{H^{\s}}\|(\pa_v-t\pa_z)^2P_{\neq}\phi\|_{H^{3}}\\
&\quad+\|(1+h)\pa_vh\|_{H^3}\|\na_{L}P_{\neq}\phi\|_{H^{\s}}
+\|(1+h)\pa_vh\|_{H^{\s}}\|\na_{L}P_{\neq}\phi\|_{H^{3}}\\
&\lesssim \|f_{\neq}\|_{H^{\s}}
+\ep\nu^{\f13}\|\Delta_{L}\Delta_t^{-1}f_{\neq}\|_{H^{\s}}
+\ep\nu^{\f16}\|f_{\neq}\|_{H^3}
+\ep\nu^{\f16}\|\na_{L}P_{\neq}\phi\|_{H^{3}}\\
&\quad+\|\pa_vh\|_{H^{\s}}\|\na_{L}P_{\neq}\phi\|_{H^{3}}\\
&\lesssim \|f_{\neq}\|_{H^{\s}} 
+\ep\nu^{\f16}\|\Delta_{L}\Delta_t^{-1}f_{\neq}\|_{H^{\s}}
+\langle t\rangle^{-1}\|\pa_vh\|_{H^{\s}}\|\Delta_{L}P_{\neq}\phi\|_{H^{4}}\\
&\lesssim \|f_{\neq}\|_{H^{\s}} 
+\ep\nu^{\f16}\|\Delta_{L}\Delta_t^{-1}f_{\neq}\|_{H^{\s}}
+\f{\nu^{\f13}}{\langle t\rangle\langle\nu t^3\rangle}\|\pa_vh\|_{H^{\s}}.
\end{align*}
We also have
\beno
&&\pa_z\tu_{\neq}=-(1+h)(\pa_v-t\pa_z)\pa_zP_{\neq}\phi,\\
&&(\pa_v-t\pa_z)\tu_{\neq}=-(1+h)(\pa_v-t\pa_z)^2P_{\neq}\phi-\pa_vh(\pa_v-t\pa_z)P_{\neq}\phi.
\eeno
Therefore by Lemma \ref{lem: lin-inv-dam} and the bootstrap hypotheses, we get
\begin{align*}
\|\na_{L}\tu_{\neq}\|_{H^{\s}}
\lesssim (1+\|h\|_{H^{\s}})\|\Delta_{L}P_{\neq}\phi\|_{H^{\s}}.
\end{align*}
By taking $\ep$ small enough, we get the first inequality. 

In what follows we use the shorthand 
\beno
G(\xi)=\widehat{1-(v')^2}(\xi).
\eeno
and then 
\begin{align*}
&\sqrt{\f{\pa_tw_k(t,\eta)}{w_{k}(t,\eta)}}1_{t\in I_{k,\eta}}1_{k\neq 0}(k^2+(\eta-kt)^2)\phi_k(t,\eta)\\
&=\sqrt{\f{\pa_tw_k(t,\eta)}{w_{k}(t,\eta)}}1_{t\in I_{k,\eta}}1_{k\neq 0}f_{k}(t,\eta)\\
&\quad-\sqrt{\f{\pa_tw_k(t,\eta)}{w_{k}(t,\eta)}}1_{t\in I_{k,\eta}}1_{k\neq 0}
\int_{|\xi|\geq |\eta-\xi|}{G}(\xi)(\eta-\xi-kt)^2\widehat{\phi}_k(t,\eta-\xi)d\xi\\
&\quad-\sqrt{\f{\pa_tw_k(t,\eta)}{w_{k}(t,\eta)}}1_{t\in I_{k,\eta}}1_{k\neq 0}
\int_{|\xi|\geq |\eta-\xi|}{G}(\eta-\xi)(\xi-kt)^2\widehat{\phi}_k(t,\xi)d\xi\\
&\quad-i\sqrt{\f{\pa_tw_k(t,\eta)}{w_{k}(t,\eta)}}1_{t\in I_{k,\eta}}1_{k\neq 0}
\int_{|\xi|\geq |\eta-\xi|}\widehat{v''}(\xi)(\eta-\xi-kt)\widehat{\phi}_k(t,\eta-\xi)d\xi\\
&\quad-i\sqrt{\f{\pa_tw_k(t,\eta)}{w_{k}(t,\eta)}}1_{t\in I_{k,\eta}}1_{k\neq 0}
\int_{|\xi|\geq |\eta-\xi|}\widehat{v''}(\eta-\xi)(\xi-kt) \widehat{\phi}_k(t,\xi)d\xi\\
&=\sqrt{\f{\pa_tw_k(t,\eta)}{w_{k}(t,\eta)}}1_{t\in I_{k,\eta}}1_{k\neq 0}f_{k}(t,\eta)
+E_{HL}^1+E_{LH}^1
+E_{HL}^2+E_{LH}^2
\end{align*}
We have $t\approx t_{k,\eta}\approx \f{\eta}{k}$ and then
\beno
\f{\pa_tw_{k}(t,\eta)}{w_{k}(t,\eta)}=\f{\langle\nu^{\f13}t_{k,\eta}\rangle^{-(1+\b)}\nu^{\f13}\f{\eta}{k^2} }{1+(\f{\eta}{k}-t)^2}\lesssim \f{1}{k}\langle\nu^{\f13}t\rangle^{-(1+\b)}\nu^{\f13}t\lesssim \f1k.
\eeno
Thus we get
\begin{align*}
\|E_{HL}^1\|_{H^{\s}}
&\lesssim \|G\|_{H^{\s}}\|(\pa_v-t\pa_z)^2P_{\neq}\phi\|_{H^4}\\
&\lesssim \|h\|_{H^{\s}}(\|h\|_{H^3}+1)\|f_{\neq}\|_{H^4}\\
&\lesssim \f{\ep^2\nu^{\f12}}{\langle\nu t^3\rangle}.
\end{align*}
For $E_{LH}^1$, we get
\begin{align*}
\f{\pa_tw_{k}(t,\eta)}{w_{k}(t,\eta)}\f{w_{k}(t,\xi)}{\pa_tw_{k}(t,\xi)}\approx \f{1+|\f{\xi}{k}-t|^2}{1+|\f{\eta}{k}-t|^2}\lesssim \langle\eta-\xi\rangle^2
\end{align*}
Then we get 
\begin{align*}
\|E_{LH}^1\|_{H^{\s}}
&\lesssim \left\|\langle k,\eta\rangle^{\s}\sqrt{\f{\pa_tw_k(t,\eta)}{w_{k}(t,\eta)}}1_{t\in I_{k,\eta}}1_{k\neq 0}
\int_{|\xi|\geq |\eta-\xi|}\widehat{G}(\eta-\xi)(\xi-kt)^2\widehat{\phi}_k(t,\xi)d\xi\right\|_{L^2}\\
&\lesssim \left\|1_{t\in I_{k,\eta}}1_{k\neq 0}
\int_{|\xi|\geq |\eta-\xi|}\langle\eta-\xi\rangle\widehat{G}(\eta-\xi)(\xi-kt)^2\langle k,\xi\rangle^{\s}\sqrt{\f{\pa_tw_k(t,\xi)}{w_{k}(t,\xi)}}\widehat{\phi}_k(t,\xi)d\xi\right\|_{L^2}\\
&\lesssim \|G\|_{H^6}\left\|\sqrt{\f{\pa_tw}{w}}\chi_R\Delta_{L}\Delta_t^{-1}f_{\neq}\right\|_{H^{\s}}
\lesssim \ep \nu^{\f13}\left\|\sqrt{\f{\pa_tw}{w}}\chi_R\Delta_{L}\Delta_t^{-1}f_{\neq}\right\|_{H^{\s}},
\end{align*}
similarly we have
\begin{align*}
\|E_{LH}^2\|_{H^{\s}}
&\lesssim \left\|\langle k,\eta\rangle^{\s}\sqrt{\f{\pa_tw_k(t,\eta)}{w_{k}(t,\eta)}}1_{t\in I_{k,\eta}}1_{k\neq 0}
\int_{|\xi|\geq |\eta-\xi|}\widehat{v''}(\eta-\xi)|\xi-kt|\widehat{\phi}_k(t,\xi)d\xi\right\|_{L^2}\\
&\lesssim \left\|1_{t\in I_{k,\eta}}1_{k\neq 0}
\int_{|\xi|\geq |\eta-\xi|}\langle\eta-\xi\rangle\widehat{v''}(\eta-\xi)|\xi-kt|\langle k,\xi\rangle^{\s}\sqrt{\f{\pa_tw_k(t,\xi)}{w_{k}(t,\xi)}}\widehat{\phi}_k(t,\xi)d\xi\right\|_{L^2}\\
&\lesssim \|v''\|_{H^6}\left\|\sqrt{\f{\pa_tw}{w}}\chi_R\Delta_{L}\Delta_t^{-1}f_{\neq}\right\|_{H^{\s}}
\lesssim \ep \nu^{\f13}\left\|\sqrt{\f{\pa_tw}{w}}\chi_R\Delta_{L}\Delta_t^{-1}f_{\neq}\right\|_{H^{\s}}.
\end{align*}
At last we deal with $T_{HL}^2$, we have $\sqrt{\f{\pa_tw_k(t,\xi)}{w_{k}(t,\xi)}}\lesssim \f{kt}{\eta}$ and then get
\begin{align*}
\|E_{HL}^2\|_{H^{\s}}
&\lesssim \left\|\langle k,\eta\rangle^{\s}\f{kt}{\eta}1_{t\in I_{k,\eta}}1_{k\neq 0}
\int_{|\xi|\geq |\eta-\xi|}\widehat{v''}(\xi)(\eta-\xi-kt) \widehat{\phi}_k(t,\eta-\xi)d\xi\right\|_{L^2}\\
&\lesssim \left\|1_{t\in I_{k,\eta}}1_{k\neq 0}
\int_{|\xi|\geq |\eta-\xi|}\langle \xi\rangle^{\s-1}|\widehat{v''}(\xi)||kt||\eta-\xi-kt||\widehat{\phi}_k(t,\eta-\xi)|d\xi\right\|_{L^2}\\
&\lesssim t\|v''\|_{H^{\s-1}}\|\na_{L}\phi_{\neq}\|_{H^4}\\
&\lesssim  \|h\|_{H^{\s}}\|f_{\neq}\|_{H^{5}}
\lesssim \f{\ep^2\nu^{\f12}}{\langle \nu t^3\rangle}.
\end{align*}
Thus we proved the lemma. 
\end{proof}

By the fact that $u=(0,g)^{T}+(1+h)\na^{\bot}_{z,v}P_{\neq}\phi$, Lemma \ref{lem: lin-inv-dam} and under the bootstrap hypotheses, it holds that
\beq\label{eq: u H^s}
\|u\|_{H^s}\lesssim \|g\|_{H^s}+\|P_{\neq}\phi\|_{H^{s+1}}\lesssim \f{\ep\nu^{\f13}}{\langle t^2\rangle}.
\eeq

\begin{lemma}\label{lem: loss-elliptic-A_E^s}
Under the bootstrap hypotheses for $\ep$ sufficiently small, for $s\leq \s-7$ it holds that
\beno
\|A^s_{E}(P_{\neq}\phi)\|_2\lesssim \f{1}{\langle t\rangle^2}(\|A^s_{E}f\|_2+\|f\|_{H^{\s}}). 
\eeno
\end{lemma}
\begin{proof}
By Lemma \ref{lem: D-D}, we have
\begin{align*}
\|A^s_{E}(\na^{\bot}P_{\neq}\phi)\|_2^2
&\approx \|\nu\max\{t^3,\eta^3\}\hat{\phi}_{\neq}\|_{H^{s}}^2\\
&\lesssim \nu^2\sum_{k\neq 0}\int_{2|\eta|\geq t} \langle k,\eta\rangle^{2s+6}|\hat{\phi}_k(t,\eta)|^2d\eta
+\nu^2\sum_{k\neq 0}\int_{2|\eta|< t} t^6\langle k,\eta\rangle^{2s}|\hat{\phi}_k(t,\eta)|^2d\eta\\
&=\Pi_1+\Pi_2.
\end{align*}
By Lemma \ref{lem: lin-inv-dam}, we get 
\begin{align*}
|\Pi_1|\lesssim \|P_{\neq}\phi\|_{H^{s+3}}^2\lesssim \f{1}{\langle t\rangle^4}\|f_{\neq}\|_{H^{s+5}}^2\lesssim \f{1}{\langle t\rangle^4}\|f_{\neq}\|_{H^{\s}}^2,
\end{align*}
and 
\begin{align*}
|\Pi_2|\lesssim \nu^2t^6\|P_{\neq}\phi\|_{H^{s}}^2
&\lesssim \nu^2\sum_{k\neq 0}\int_{2|\eta|< t} t^6\f{\langle k,\eta\rangle^{2s}}{(k^2+|\eta-kt|^2)^2}|\widehat{\Delta_{L}\phi}_k(t,\eta)|^2d\eta\\
&\lesssim \nu^2\sum_{k\neq 0}\int_{2|\eta|< t} t^6\f{\langle k,\eta\rangle^{2s}}{(k^2+|\eta|^2+k^2t^2)^2}|\widehat{\Delta_{L}\phi}_k(t,\eta)|^2d\eta\\
&\lesssim \f{1}{\langle t\rangle^4}\|\nu t^3\Delta_{L}\phi_{\neq}\|_{H^{s}}^2.
\end{align*}
By Lemma \ref{lem: lin-inv-dam}, we then obtain that
\begin{align*}
|\Pi_2|\lesssim
\f{1}{\langle t\rangle^4}\|\nu t^3f_{\neq}\|_{H^{s}}^2\lesssim \f{1}{\langle t\rangle^4}\|A_E^sf_{\neq}\|_2^2.
\end{align*}
Thus we proved the lemma.
\end{proof}

\section{Dissipation error term}\label{Sec: dissipation}
In this section, we will deal with the dissipation error term in \eqref{eq: E_0andEneq}. 

\subsection{Treatment of the zero mode}
By the fact that $\As_0(\eta)\approx \langle \eta\rangle^{\s}\approx 1+|\eta|^{\s}$ and $|\eta|\leq |\xi|+|\eta-\xi|\lesssim \max\{|\xi|,|\eta-\xi|\}$, we get that
\begin{align*}
|E^0|&\lesssim \int_{\xi,\eta}\langle \eta\rangle^{2\s}
|\bar{\hat{f}}_0(\eta)|\left|(\widehat{1-(v')^2}(\eta-\xi))|\xi|^2\hat{f}_0(\xi)\right|d\xi d\eta\\
&\lesssim \int_{\xi,\eta}1_{|\eta|\leq 1}
|\bar{\hat{f}}_0(\eta)|\left|(\widehat{1-(v')^2}(\eta-\xi))|\xi|(|\eta|+|\xi-\eta|)\hat{f}_0(\xi)\right| d\eta d\xi\\
&\quad+\int_{\xi,\eta}1_{|\eta|\geq 1}1_{|\xi-\eta|\geq |\xi|}|\eta|^{2\s}
|\bar{\hat{f}}_0(\eta)|\left|(\widehat{1-(v')^2}(\eta-\xi))|\xi|^2\hat{f}_0(\xi)\right|d\xi d\eta\\
&\quad+\int_{\xi,\eta}1_{|\eta|\geq 1}1_{|\xi-\eta|<|\xi|}|\eta|^{2\s}
|\bar{\hat{f}}_0(\eta)|\left|(\widehat{1-(v')^2}(\eta-\xi))|\xi|^2\hat{f}_0(\xi)\right|d\xi d\eta\\
&\lesssim \int_{\xi,\eta}1_{|\eta|\leq 1}
|\bar{\hat{\pa_vf}}_0(\eta)|\left|(\widehat{1-(v')^2}(\eta-\xi))|\xi|\hat{f}_0(\xi)\right| d\eta d\xi\\
&\quad+\int_{\xi,\eta}1_{|\eta|\leq 1}
|\bar{\hat{f}}_0(\eta)|\left|(\widehat{\pa_v(1-(v')^2)}(\eta-\xi))|\xi|\hat{f}_0(\xi)\right| d\eta d\xi\\
&\quad+\int_{\xi,\eta}1_{|\eta|\geq 1}1_{|\xi-\eta|< |\xi|}
|\eta|^{\s+1}
|\bar{\hat{f}}_0(\eta)|\left|(\widehat{1-(v')^2}(\eta-\xi))|\xi|^{\s+1}\hat{f}_0(\xi)\right|d\xi d\eta\\
&\quad+\int_{\xi,\eta}1_{|\eta|\geq 1}1_{|\xi-\eta|\geq |\xi|}|\eta|^{\s+1}
|\bar{\hat{f}}_0(\eta)|\left|(\widehat{1-(v')^2}(\eta-\xi))|\eta-\xi|^{\s-1}|\xi|^2\hat{f}_0(\xi)\right|d\xi d\eta\\
&\lesssim \nu\|\pa_vf_0\|_{H^2}^2\|1-(v')^2\|_{2}
+\nu\|\pa_vf_0\|_{H^2}\|\pa_v(1-(v')^2)\|_{2}\|f_0\|_{2}\\
&\quad+\nu\|\pa_vf_0\|_{H^{\s}}^2\|1-(v')^2\|_{H^2}
+\nu\|\pa_vf_0\|_{H^{\s}}\|\pa_v(1-(v')^2)\|_{H^{\s-2}}\|f_0\|_{H^4}.
\end{align*}
The purpose of above estimate is to obtain the homogeneous derivative. 
By the fact that $(v')^2-1=(1-(v'))^2+2(v'-1)=h^2+2h$ and 
\beno
\|h^2\|_{H^{s}}\lesssim \|h\|_{H^1}\|h\|_{H^{s}},\quad \|\pa_vh^2\|_{H^s}\lesssim \|h\|_{H^1}\|\pa_vh\|_{H^{s}}\quad s\geq 1,
\eeno
we obtain by the bootstrap hypotheses that
\beq\label{eq:E^0}
\begin{split}
|E^0|&\lesssim \nu (\|h\|_{H^2}+1)\Big(\|\pa_v \As f_{0}\|_2^2\|h\|_{H^{2}}
+\|\pa_vh\|_{H^{\s-2}}\|\pa_vf_0\|_{H^{\s}}\|f_0\|_{H^4}\Big).\\
&\lesssim \ep\nu^{\f13}\nu\|\pa_v \As f_{0}\|_2^2+\ep\nu^{\f13}\nu\|\pa_vh\|_{H^{\s-1}}^2.
\end{split}
\eeq

\subsection{Treatment of the non-zero mode}
We use a paraproduct decomposition in $v$. Then we have 
\begin{align*}
E^{\neq}=E^{\neq}_{LH}+E^{\neq}_{HL}+E^{\neq}_{HH},
\end{align*}
where
\begin{align*}
&E^{\neq}_{LH}=-\sum_{M\geq 8}\nu \int \As f_{\neq}\As \left((1-(v')^2)_{<M/8}(\pa_v-t\pa_z)^2(f_{\neq})_{M}\right)dzdv\\
&E^{\neq}_{HL}=-\sum_{M\geq 8}\nu \int \As f_{\neq}\As \left((1-(v')^2)_{M}(\pa_v-t\pa_z)^2(f_{\neq})_{<M/8}\right)dzdv\\
&{E}^{\neq}_{HH}=-\nu \sum_{M\in \mathbb{D}}\sum_{\f18M\leq M'\leq 8M}\int \As f_{\neq}\As \left((1-(v')^2)_{M}(\pa_v-t\pa_z)^2(f_{\neq})_{M'}\right)dzdv.
\end{align*}
\subsubsection{Treatment of $E^{\neq}_{LH}$}
We have 
\begin{align*}
E^{\neq}_{LH}\lesssim \nu\sum_{M\geq 8} \sum_{k\neq 0}\int_{\eta,\xi}\As |\bar{\hat{f}}_k(\eta)|\As_{k}(\eta)
|\widehat{(1-(v')^2)}(\eta-\xi)_{<M/8}|\xi-kt|^2f_{k}(\xi)_{M}d\xi d\eta.
\end{align*}
By the fact that $\xi\approx \eta\approx M$, $|k,\eta|\approx |k,\xi|$ and 
\beno
|\xi-kt|\lesssim |\xi-\eta|+|\eta-kt|\lesssim \langle \xi-\eta\rangle \sqrt{k^2+|\eta-kt|^2},
\eeno
we have 
\begin{align*}
E^{\neq}_{LH}&\lesssim 
\nu\sum_{M\geq 8} \sum_{k\neq 0}\int_{\eta,\xi}\sqrt{k^2+|\eta-kt|^2}\As |\bar{\hat{f}}_k(\eta)|
|\widehat{\langle\pa_v\rangle(1-(v')^2)}(\eta-\xi)_{<M/8}|\xi-kt|\As_{k}(\xi)f_{k}(\xi)_{M}d\xi d\eta\\
&\lesssim \nu\sum_{M\geq 8}\|(\sqrt{-\Delta_L}\As f_{\neq})_{\sim M}\|_2\|(\sqrt{-\Delta_L}\As f_{\neq})_{M}\|_2\|(1-(v')^2)\|_{H^4},
\end{align*}
which gives 
\begin{align*}
E^{\neq}_{LH}&\lesssim \nu\|(\sqrt{-\Delta_L}\As f_{\neq})\|_2^2\|(1-(v')^2)\|_{H^4},
\end{align*} 

\subsubsection{Treatment of $E^{\neq}_{HL}$}
We have 
\begin{align*}
E^{\neq}_{HL}&\lesssim \nu\sum_{M\geq 8} \sum_{k\neq 0}\int_{\eta,\xi}
\left[1_{|\eta|\leq 16|k|}+1_{|\eta|>16|k|}\right]
\As |\bar{\hat{f}}_k(\eta)|\As_{k}(\eta)\\
&\quad\quad\quad\quad\quad\quad\times 
|\widehat{(1-(v')^2)}(\eta-\xi)_{M}|\xi-kt|^2f_{k}(\xi)_{<M/8}d\xi d\eta\\
&=E^{\neq,z}_{HL}+E^{\neq,v}_{HL}
\end{align*}
For $E^{\neq,z}_{HL}$, we have $|k,\eta|\approx |k|\approx |k,\xi|$ and 
\beno
|\xi-kt|\lesssim |\xi-\eta|+|\eta-kt|\lesssim \langle \xi-\eta\rangle \sqrt{k^2+|\eta-kt|^2},
\eeno
which then implies
\begin{align*}
E^{\neq,z}_{HL}&\lesssim \nu\sum_{M\geq 8} \sum_{k\neq 0}\int_{\eta,\xi}
1_{|\eta|\leq 16|k|}
\As\sqrt{k^2+|\eta-kt|^2} |\bar{\hat{f}}_k(\eta)|\\
&\quad\quad\quad\quad\quad\quad\times 
|\langle\eta-\xi\rangle\widehat{(1-(v')^2)}(\eta-\xi)_{M}|\xi-kt||k|^{\s}f_{k}(\xi)_{<M/8}d\xi d\eta\\
&\lesssim \nu\sum_{M\geq 8} M^{-2}\|(1-(v')^2)_{M}\|_{H^5}\|(\sqrt{-\Delta_L}\As f_{\neq})\|_2^2.
\end{align*}
Thus we have
\beno
E^{\neq,z}_{HL}\lesssim \nu\|(1-(v')^2)\|_{H^5}\|(\sqrt{-\Delta_L}\As f_{\neq})\|_2^2
\eeno

We turn to $E^{\neq,v}_{HL}$. In this case 
$|k,\eta|\approx |\eta|\approx |\eta-\xi|\approx M$, then we get
\begin{align*}
|E^{\neq,v}_{HL}|&\lesssim \nu\sum_{M\geq 8} \sum_{k\neq 0}\int_{\eta,\xi}
1_{|\eta|>16|k|}
\As |\bar{\hat{f}}_k(\eta)||\eta-\xi|\langle\eta-\xi\rangle^{\s-1}|\widehat{(1-(v')^2)}(\eta-\xi)_{M}|\\
&\quad\quad\quad\quad\quad\quad\times 
|\xi-kt|^2f_{k}(\xi)_{<M/8}d\xi d\eta\\
&\lesssim \nu\sum_{M\geq 8} 
\|(f_{\neq})_{\sim M}\|_{H^{\s}}\|\pa_v(1-(v')^2)_{M}\|_{H^{\s-1}}
\langle t\rangle^2\|f_{\neq}\|_{H^{5}}\\
&\lesssim \nu
\|f_{\neq}\|_{H^{\s}}\|\pa_v(1-(v')^2)\|_{H^{\s-1}}
\langle t\rangle^2\|f_{\neq}\|_{H^{5}}.
\end{align*}
\subsubsection{Treatment of $E^{\neq}_{HH}$}
In this case, it holds that $|\eta-\xi|\approx |\xi|\approx M'$. 
We divide into two parts: 
\begin{align*}
|{E}^{\neq}_{HH}|&\lesssim 
\nu \sum_{M\in \mathbb{D}}\sum_{\f18M\leq M'\leq 8M}
\sum_{k\neq 0}\int_{\eta,\xi}
\left[1_{|k|\geq 16|\xi|}+1_{|k|<16|\xi|}\right]\\
&\quad\quad\quad\quad\quad\quad\times
\As |\bar{\hat{f}}_k(\eta)|\As_{k}(\eta)
|\widehat{(1-(v')^2)}(\eta-\xi)_{M'}|\xi-kt|^2f_{k}(\xi)_{M}d\xi d\eta\\
&={E}^{\neq,z}_{HH}+{E}^{\neq,v}_{HH}.
\end{align*}

To treat ${E}^{\neq,z}_{HH}$, we have 
\beno
|k|\lesssim |k,\eta| \lesssim |k|+|\eta-\xi|+|\xi|\lesssim |k|,
\eeno
and
\beno
|\xi-kt|\lesssim |\xi-\eta|+|\eta-kt|\lesssim \langle \xi-\eta\rangle \sqrt{k^2+|\eta-kt|^2}.
\eeno
Therefore we get that
\begin{align*}
{E}^{\neq,z}_{HH}&\lesssim 
\nu \sum_{M\in \mathbb{D}}\sum_{\f18M\leq M'\leq 8M}
\sum_{k\neq 0}\int_{\eta,\xi}
1_{|k|\geq 16|\xi|}
\As\sqrt{k^2+|\eta-kt|^2} |\bar{\hat{f}}_k(\eta)|\\
&\quad\quad\quad\quad\quad\quad\quad\times
\langle \xi-\eta\rangle|\widehat{(1-(v')^2)}(\eta-\xi)_{M'}|\xi-kt||k|^{\s}f_{k}(\xi)_{M}d\xi d\eta\\
&\lesssim \nu \sum_{M\in \mathbb{D}}\|\sqrt{-\Delta_L}\As f\|_{2}\|(\sqrt{-\Delta_L}\As f)_{M}\|_{2}\|(1-(v')^2)_{\sim M}\|_{H^3}\\
&\lesssim \nu \|\sqrt{-\Delta_L}\As f\|_{2}^2\|(1-(v')^2)\|_{H^3}.
\end{align*}

Next we turn to ${E}^{\neq,v}_{HH}$, in which case we have 
\beno
|k,\eta|\lesssim |k|+|\eta|\lesssim |k|+|\eta-\xi|+|\xi| \approx |k|+|\xi|\lesssim |\xi|\approx |\xi-\eta|,
\eeno
and
\beno
|\xi-kt|\lesssim |\xi-\eta|+|\eta-kt|\lesssim \langle \xi-\eta\rangle \sqrt{k^2+|\eta-kt|^2}.
\eeno
Therefore we get that
\begin{align*}
{E}^{\neq,v}_{HH}&\lesssim 
\nu \sum_{M\in \mathbb{D}}\sum_{\f18M\leq M'\leq 8M}
\sum_{k\neq 0}\int_{\eta,\xi}
1_{|k|< 16|\xi|}
\As\sqrt{k^2+|\eta-kt|^2} |\bar{\hat{f}}_k(\eta)|\\
&\quad\quad\quad\quad\quad\quad\quad\times
\langle \xi-\eta\rangle|\widehat{(1-(v')^2)}(\eta-\xi)_{M'}|\xi-kt||\xi|^{\s}f_{k}(\xi)_{M}d\xi d\eta\\
&\lesssim \nu \sum_{M\in \mathbb{D}}\|\sqrt{-\Delta_L}\As f\|_{2}\|(\sqrt{-\Delta_L}\As f)_{M}\|_{2}\|(1-(v')^2)_{\sim M}\|_{H^3}\\
&\lesssim \nu \|\sqrt{-\Delta_L}\As f\|_{2}^2\|(1-(v')^2)\|_{H^3}.
\end{align*}

By the fact that $(v')^2-1=h^2+2h$, the bootstrap hypotheses and \eqref{eq: enha}, we obtain that
\beq\label{eq: E_neq}
\begin{split}
|E^{\neq}|&\lesssim \nu
\|f_{\neq}\|_{H^{\s}}\|\pa_v(1-(v')^2)\|_{H^{\s-1}}
\langle t\rangle^2\|f_{\neq}\|_{H^{5}}
+\nu \|\sqrt{-\Delta_L}\As f\|_{2}^2\|(1-(v')^2)\|_{H^4}\\
&\lesssim \nu(1+\|h\|_{H^2})
\Big(\|f_{\neq}\|_{H^{\s}}\|\pa_vh\|_{H^{\s-1}}
\langle t\rangle^2\|f_{\neq}\|_{H^{5}}
+\|\sqrt{-\Delta_L}\As f\|_{2}^2\|h\|_{H^4}\Big)\\
&\lesssim \nu \ep\nu^{\f13}\|\sqrt{-\Delta_L}\As f\|_{2}^2
+(\ep\nu^{\f13})^2\nu\|\pa_vh\|_{H^{\s-1}}\f{\langle t\rangle^2}{\langle \nu t^3\rangle}.
\end{split}
\eeq

We end the section by proving Proposition \ref{prop: dissipation}. 
\begin{proof}
We get by \eqref{eq: E_0andEneq} that
\begin{align*}
&\int_1^t\left(\nu \int \As f(t')\As \left(\tilde{\Delta}_tf(t')\right)dzdv\right)dt'\\
&\leq -\int_1^t\nu \left\|\sqrt{-\Delta_L}\As f(t')\right\|_2^2dt'
+\int_1^t|E^{\neq}(t')|+|E^0(t')|dt'. 
\end{align*} 
Then by \eqref{eq:E^0} and \eqref{eq: E_neq}, we obtain that
\begin{align*}
&\int_1^t\left(\nu \int \As f(t')\As \left(\tilde{\Delta}_tf(t')\right)dzdv\right)dt'\\
&\leq -\int_1^t\nu \left\|\sqrt{-\Delta_L}\As f(t')\right\|_2^2dt'+C\ep\nu^{\f13}\int_1^t\nu \left\|\sqrt{-\Delta_L}\As f(t')\right\|_2^2dt'\\
&\quad+C\ep\nu^{\f13}\nu\|\pa_vh\|_{L^2_T(H^{\s-1})}^2
+C\int_1^t(\ep\nu^{\f13})^2\nu\|\pa_vh(t')\|_{H^{\s-1}}\f{\langle t\rangle^2}{\langle \nu t'^3\rangle}dt'.
\end{align*} 
Thus by taking $\ep$ small enough and using Proposition \ref{prop: basic estimate}, we get
\begin{align*}
&\int_1^t\left(\nu \int \As f(t')\As \left(\tilde{\Delta}_tf(t')\right)dzdv\right)dt'\\
&\leq -\f78\int_1^t\nu \left\|\sqrt{-\Delta_L}\As f(t')\right\|_2^2dt'
+\ep^2\nu  \|\pa_vh(t')\|_{L^2_{T}H^{\s-1}}\left(\int_1^t\f{1}{\langle \nu^{\f13} t'\rangle^2}dt'\right)^{\f12}\\
&\leq -\f78\int_1^t\nu \left\|\sqrt{-\Delta_L}\As f(t')\right\|_2^2dt'+C\ep^3\nu^{\f23}.
\end{align*}
Thus we proved the proposition. 
\end{proof}

\section{Transport}\label{Sec: Transport}
To treat the transport term, we need consider the commutator. The following lemma gives the key commutator estimate. 
\begin{lemma}\label{lem: w-w}
Assume that $|\xi-\eta|\leq \f{1}{10}|\eta|$, then it holds that
\beno
|w(t,\eta)-w(t,\xi)|\lesssim \f{|\xi-\eta|}{\langle\eta\rangle}\times\left\{\begin{aligned}
&\nu^{-\f13},\quad t\lesssim \nu^{-\f13},\\
&\nu^{\f13\b}t^{1-\b},\quad t\gtrsim \nu^{-\f13}.
\end{aligned}\right.
\eeno
\end{lemma}
Let us admit the lemma and finish the estimate of transport term first. Then proof of the lemma will be present at the end of this section. 

We write 
\begin{align*}
T_{N}&=i\sum_{k,l}\int_{\eta,\xi}\As_k(\eta)\bar{\hat{f}}_k(\eta)
\hat{u}_{k-l}(\eta-\xi)_{<N/8}\cdot (l,\xi)\As_l(\xi)\hat{f}_l(\xi)_{N}\left(\f{\As_k(\eta)}{\As_l(\xi)}-1\right)d\xi d\eta\\
&=i\sum_{k,l}\int_{\eta,\xi}\As_k(\eta)\bar{\hat{f}}_k(\eta)
\hat{u}_{k-l}(\eta-\xi)_{<N/8}\cdot (l,\xi)\As_l(\xi)\hat{f}_l(\xi)_{N}\left(\f{\langle k,\eta\rangle^{\s}}{\langle l,\xi\rangle^{\s}}-1\right)\f{w_{l}(t,\xi)}{w_{k}(t,\eta)}d\xi d\eta\\
&\quad+i\sum_{k,l}\int_{\eta,\xi}\As_k(\eta)\bar{\hat{f}}_k(\eta)
\hat{u}_{k-l}(\eta-\xi)_{<N/8}\cdot (l,\xi)\As_l(\xi)\hat{f}_l(\xi)_{N}\left(\f{w_l(t,\xi)}{w_k(t,\eta)}-1\right)d\xi d\eta\\
&=T_N^1+T_N^2.
\end{align*}
For the first term, we get 
\beno
\left|\f{\langle k,\eta\rangle^{\s}}{\langle l,\xi\rangle^{\s}}-1\right|\lesssim \f{\langle k-l,\eta-\xi\rangle}{\langle l,\xi\rangle},
\eeno
which gives 
\beno
|T_N^1|\lesssim \|\As f_{\sim N}\|_2\|\As f_{N}\|_2\|u\|_{H^{4}}. 
\eeno

Next we will deal with $T_{N}^2$. 
By the support of the integrand, we get 
\beno
N/16\leq |k-l,\xi-\eta|\leq 3N/16,\quad
N/2\leq |l,\xi|\leq 3N/2. 
\eeno
We then set more restrictions on the support of the integrand to make $k,\eta$ and $l,\xi$ be closer. We get 
\begin{align*}
T_{N}^2&=i\sum_{k,l}\int_{\eta,\xi}(\chi^D+(1-\chi^D))
A_k(\eta)\bar{\hat{f}}_k(\eta)
\hat{u}_{k-l}(\eta-\xi)_{<N/8}\cdot (l,\xi)A_l(\xi)\hat{f}_l(\xi)_{N}\left(\f{w_l(t,\xi)}{w_k(t,\eta)}-1\right)d\xi d\eta\\
&=T_{N,D}^2+T_{N,*}^2,
\end{align*}
where $\chi^D$ is a characteristic function (the indicator function) of the set 
\beno
D=\left\{(k,l,\xi,\eta):\  |k-l,\xi-\eta|\leq \f{1}{1000}|l,\xi|\right\}.
\eeno 
Then we get by $|\f{w(t,\xi)}{w(t,\eta)}|\lesssim 1$ that
\beno
|T_{N,*}^2|\lesssim \|A^{\s}f_{\sim N}\|_2\|A^{\s}f_{N}\|_2\|u\|_{H^{4}}. 
\eeno

We rewrite $T_{N,D}^2$ as follows. 
\begin{align*}
T_{N,D}^2&=i\sum_{k\neq l}\int_{\eta,\xi}\chi^D
A_k(\eta)\bar{\hat{f}}_k(\eta)
\hat{u}_{k-l}(\eta-\xi)_{<N/8}\cdot (l,\xi)A_l(\xi)\hat{f}_l(\xi)_{N}\left(\f{w_l(t,\xi)}{w_k(t,\eta)}-1\right)d\xi d\eta\\
&\quad+i\sum_{l}\int_{\eta,\xi}\chi^D
A_k(\eta)\bar{\hat{f}}_l(\eta)
\hat{u}_{0}(\eta-\xi)_{<N/8}\cdot (l,\xi)A_l(\xi)\hat{f}_l(\xi)_{N}\left(\f{w_l(t,\xi)}{w_l(t,\eta)}-1\right)d\xi d\eta\\
&=T_{N,\neq}^2+T_{N,=}^2.
\end{align*}

\subsection{Treatment of $T_{N,=}^2$}
We get by the fact that $u_0=(0,g)$ and that for $|l|\geq 20\max\{|\xi|,|\eta|\}$, $w_l(t,\xi)=w_l(t,\eta)=w(t,\f{l}{20})$ and then 
\begin{align*}
T_{N,=}^2=i\sum_{0\neq l\leq 20\max\{|\xi|,|\eta|\}}\int_{\eta,\xi}\chi^D
A_k(\eta)\bar{\hat{f}}_l(\eta)
\hat{g}(\eta-\xi)_{<N/8}\cdot \xi A_l(\xi)\hat{f}_l(\xi)_{N}\left(\f{w_l(t,\xi)}{w_l(t,\eta)}-1\right)d\xi d\eta. 
\end{align*}
Due to the fact that $0\neq l\leq 20\max\{|\xi|,|\eta|\}\approx |\xi|$, we get $\varrho(l,\eta)\approx |\eta|$. Thus by Lemma \ref{lem: w-w} and Lemma \ref{lem: rho}, we obtain that
\begin{align*}
\left|\f{w_l(t,\xi)}{w_l(t,\eta)}-1\right|
&\lesssim |w(t,\varrho(l,\xi))-w(t,\varrho(l,\eta))|\\
&\lesssim \left(\nu^{-\f13}\chi_{t\lesssim \nu^{-\f13}}(t)+\nu^{\f13\b}t^{1-\b}\chi_{t\gtrsim \nu^{-\f13}}(t)\right)\f{|\varrho(l,\xi)-\varrho(l,\eta)|}{|\varrho(l,\eta)|}\\
&\lesssim \left(\nu^{-\f13}\chi_{t\lesssim \nu^{-\f13}}(t)+\nu^{\f13\b}t^{1-\b}\chi_{t\gtrsim \nu^{-\f13}}(t)\right)\f{|\eta-\xi|}{\eta}.
\end{align*}
Therefore we get
\beno
\begin{split}
|T_{N,=}^2|&\lesssim \|\As f_{\sim N}\|_2\|\As f_{N}\|_2\|g\|_{H^{4}}\nu^{-\f13}\chi_{t\lesssim \nu^{-\f13}}(t)\\
&\quad +\|\As f_{\sim N}\|_2\|\As f_{N}\|_2\| g\|_{H^{4}}\nu^{\f13\b}t^{1-\b}\chi_{t\gtrsim \nu^{-\f13}}(t). 
\end{split}
\eeno

\subsection{Treatment of $T_{N,\neq}^2$}
By the definition of $\varrho(k,\eta)$, we have for $(l,k,\xi,\eta)\in D$ that
\beno
|\varrho(k,\eta)|\approx |\varrho(l,\xi)|\approx |l,\xi|. 
\eeno
We get by Lemma \ref{lem: w-w} and Lemma \ref{lem: rho} that
\begin{align*}
\left|\f{w_l(t,\xi)}{w_k(t,\eta)}-1\right|
&\lesssim |w(t,\varrho(l,\xi))-w(t,\varrho(k,\eta))|\\
&\lesssim \f{|\varrho(l,\xi)-\varrho(k,\eta)|}{|\varrho(k,\eta)|}\left(\nu^{-\f13}\chi_{t\lesssim \nu^{-\f13}}(t)+\nu^{\f13\b}t^{1-\b}\chi_{t\gtrsim \nu^{-\f13}}(t)\right)\\
&\lesssim \f{|l-k,\xi-\eta|}{|l,\xi|}\left(\nu^{-\f13}\chi_{t\lesssim \nu^{-\f13}}(t)+\nu^{\f13\b}t^{1-\b}\chi_{t\gtrsim \nu^{-\f13}}(t)\right),
\end{align*}
which implies that
\begin{align*}
|T_{N,\neq}^2|
\lesssim 
\|\As f_{\sim N}\|_2\|\As f_{N}\|_2
\|u_{\neq}\|_{H^4}\left(\nu^{-\f13}\chi_{t\lesssim \nu^{-\f13}}(t)+\nu^{\f13\b}t^{1-\b}\chi_{t\gtrsim \nu^{-\f13}}(t)\right).
\end{align*}

By the fact that 
\beno
u_{\neq}=h\na_{z,v}^{\bot}P_{\neq}\phi+\na_{z,v}^{\bot}P_{\neq}\phi.
\eeno
We then get by Lemma \ref{lem: lin-inv-dam}(by taking $s'=2$ in the lemma) that
\begin{align*}
\|u_{\neq}\|_{H^4}\lesssim (1+\|h\|_{H^4})\|\na_{z,v}^{\bot}P_{\neq}\phi\|_{H^4}\lesssim \f{1}{\langle t\rangle^2}(1+\|h\|_{H^4})\|f_{\neq}\|_{H^7}.
\end{align*}
Therefore we get 
\beq\label{eq: T_N-est}
\begin{split}
|T_{N}|
&\lesssim 
\|\As f_{\sim N}\|_2\|\As f_{N}\|_2
\big(\|g\|_{H^4}+\|u_{\neq}\|_{H^4}\big)\left(\nu^{-\f13}\chi_{t\lesssim \nu^{-\f13}}(t)+\nu^{\f13\b}t^{1-\b}\chi_{t\gtrsim \nu^{-\f13}}(t)\right)\\
&\lesssim \ep\|\As f_{\sim N}\|_2\|\As f_{N}\|_2
\left(\f{\chi_{t\lesssim \nu^{-\f13}}(t)}{\langle t\rangle^2}+\f{\nu^{\f13\b+\f13}\chi_{t\gtrsim \nu^{-\f13}}(t)}{\langle t\rangle^{1+\b}}\right).
\end{split}
\eeq

Now we are able to prove Proposition \ref{Prop: transport}. 
\begin{proof}
We get by \eqref{eq: T_N-est} and Proposition \ref{prop: basic estimate}, \eqref{eq: L-P-ortho} and \eqref{eq: L-P-ortho2} that
\begin{align*}
\int_1^t\sum_{N\geq 8}|T_N(t')|dt'
&\lesssim 
\ep\int_1^t\sum_{N\geq 8}\|\As f_{\sim N}\|_2\|\As f_{N}\|_2
\left(\f{\chi_{t'\lesssim \nu^{-\f13}}(t')}{\langle t'\rangle^2}+\f{\nu^{\f13\b+\f13}\chi_{t'\gtrsim \nu^{-\f13}}(t')}{\langle t'\rangle^{1+\b}}\right)dt'\\
&\lesssim \ep \sup_{t'\in[1,t]}\|\As f(t')\|_2^2.
\end{align*}
Thus we proved the proposition.
\end{proof}

\subsection{Proof of Lemma \ref{lem: w-w}}
We end this section by proving Lemma \ref{lem: w-w}. 
\begin{proof}
Without loss of the generality, we assume $0<\eta<\xi$. Then according to the relation between $t$ and $\xi,\eta$, we need to consider following 5 cases.

\no{\bf Case 1.} For $t\leq t(\eta)$, $w(t,\eta)=w(t,\xi)=1$. 

\no{\bf Case 2.} For $t(\eta)\leq t\leq t(\xi)$, there exists $l\in [1,E(\sqrt{|\eta|})]$ such that $t\in I_{l,\eta}$, then $|l-E(\sqrt{|\eta|})|\lesssim \sqrt{\xi}-\sqrt{\eta}$ and
\begin{align*}
&|w(t,\eta)-w(t,\xi)|=|w(t,\eta)-1|\\
&\leq \left|\prod_{m=E(\sqrt{|\eta|})}^{l+1}G_m(\eta)\exp\Big( \langle\nu^{\f13}t_{l,\eta}\rangle^{-(1+\b)}\f{\nu^{\f13}\eta}{l^2}\big(\arctan(t-\f{\eta}{l})+\arctan(D_{l,\eta}^-)\big)\Big)-1\right|\\
&\leq\left\{\begin{aligned}
&\exp\left(\sum_{m=E(\sqrt{|\eta|})}^{l}\f{C\nu^{\f13}\eta}{m^2}\right)-1
\leq \nu^{\f13}|\sqrt{\xi}-\sqrt{\eta}|,\quad \sqrt{\eta}\leq  \nu^{-\f13},\\
&\exp\left(\sum_{m=E(\sqrt{|\eta|})}^{l}\f{Cl^{-1+\b}}{(\nu^{\f13}\eta)^{\b}}\right)-1\leq \f{|\sqrt{\xi}-\sqrt{\eta}|}{\sqrt{|\xi|}}\left(\f{\sqrt{|\xi|}}{\nu^{\f13}\eta}\right)^{\b}, \quad \sqrt{\eta}>  \nu^{-\f13},
\end{aligned}\right.\\
&\lesssim {|\eta-\xi|}{\langle \xi\rangle^{-1}}.
\end{align*}
Here we use the fact that $|e^{x}-1|\lesssim |x|$ for $|x|\lesssim 1$. 

\no{\bf Case 3.} For $t(\xi)\leq t\leq 2\eta$, there exist $k,l$ such that $t\in I_{k,\eta}\cap I_{l,\xi}$. By Lemma \ref{lem: 3.2}, we need to consider the following three cases. 

(3a.) $k=l$, let $F_1(m,\eta)=\nu^{\f13} \langle\nu^{\f13}t_{m,\eta}\rangle^{-(1+\b)}\f{\eta}{m^2}$ and $F_2^{\pm}(m,\eta)=\nu^{\f13} \langle\nu^{\f13}t_{m,\eta}\rangle^{-(1+\b)}\f{\eta}{m^2}\arctan(D_{m,\eta}^{\pm})$ then $G_m(\eta)=e^{F^{+}(m,\eta)+F^{-}(m,\eta)}$ and then we get
\beq\label{eq: F_1,F_2}
\begin{split}
&|\pa_{\eta}F_1(m,\eta)|\lesssim \f{F_1(m,\eta)}{\langle\eta\rangle},\quad
|\pa_{\eta}F_2^{\pm}(m,\eta)|\lesssim \f{F_2^{\pm}(m,\eta)}{\langle\eta\rangle},\\
&|\arctan(t-\f{\eta}{l})-\arctan(t-\f{\xi}{l})\big)|
\lesssim \min\left\{\f{|\xi-\eta|}{l},1\right\},\\
&e^{x}-1\leq (e^{x}+1)|x| \lesssim |x|,\quad\text{for}\quad |x|\lesssim 1.
\end{split}
\eeq
Therefore, we obtain that
\begin{align*}
&|w(t,\eta)-w(t,\xi)|=w(t,\xi)\left|\f{w(t,\eta)}{w(t,\xi)}-1\right|\\
&\lesssim 
\left|\prod_{m=E(\sqrt{|\xi|})}^{m=E(\sqrt{|\eta|})+1}
\f{1}{G_m(\xi)}
\prod_{m=E(\sqrt{|\eta|})}^{l+1}
\f{G_m(\eta)}{G_m(\xi)}
\f{\exp\Big(\nu^{\f13} \langle\nu^{\f13}t_{l,\eta}\rangle^{-(1+\b)}\f{\eta}{l^2}\big(\arctan(t-\f{\eta}{l})+\arctan(D_{l,\eta}^-)\big)\Big)}{\exp\Big(\nu^{\f13} \langle\nu^{\f13}t_{l,\xi}\rangle^{-(1+\b)}\f{\xi}{l^2}\big(\arctan(t-\f{\xi}{l})+\arctan(D_{l,\xi}^-)\big)\Big)}-1\right|\\
&\lesssim \left|\prod_{m=E(\sqrt{|\xi|})}^{m=E(\sqrt{|\eta|})+1}
\f{1}{G_m(\xi)}-1\right|
+\left|\prod_{m=E(\sqrt{|\eta|})}^{l+1}
\f{G_m(\eta)}{G_m(\xi)}-1\right|\\
&\quad+\left|\f{\exp\Big(\nu^{\f13} \langle\nu^{\f13}t_{l,\eta}\rangle^{-(1+\b)}\f{\eta}{l^2}\big(\arctan(t-\f{\eta}{l})+\arctan(D_{l,\eta}^-)\big)\Big)}{\exp\Big(\nu^{\f13} \langle\nu^{\f13}t_{l,\xi}\rangle^{-(1+\b)}\f{\xi}{l^2}\big(\arctan(t-\f{\xi}{l})+\arctan(D_{l,\xi}^-)\big)\Big)}-1\right|\\
&\lesssim \left|\prod_{m=E(\sqrt{|\xi|})}^{m=E(\sqrt{|\eta|})+1}\f{1}{G_m(\xi)}-1\right|
+\left|\prod_{m=E(\sqrt{|\eta|})}^{l+1}\f{G_m(\eta)}{G_m(\xi)}-1\right|\\
&\quad+\left|\f{\exp\Big(\nu^{\f13} \langle\nu^{\f13}t_{l,\eta}\rangle^{-(1+\b)}\f{\eta}{l^2}\arctan(D_{l,\eta}^-)\Big)}{\exp\Big(\nu^{\f13} \langle\nu^{\f13}t_{l,\xi}\rangle^{-(1+\b)}\f{\xi}{l^2}\arctan(D_{l,\xi}^-)\Big)}-1\right|
+\left|\f{\exp\Big(\nu^{\f13} \langle\nu^{\f13}t_{l,\eta}\rangle^{-(1+\b)}\f{\eta}{l^2}\Big)}{\exp\Big(\nu^{\f13} \langle\nu^{\f13}t_{l,\xi}\rangle^{-(1+\b)}\f{\xi}{l^2}\Big)}-1\right|\\
&\quad+\left|\exp\Big(\nu^{\f13} \langle\nu^{\f13}t_{l,\eta}\rangle^{-(1+\b)}\f{\eta}{l^2}\big(\arctan(t-\f{\eta}{l})-\arctan(t-\f{\xi}{l})\big)\Big)-1\right|.
\end{align*}
Here and the rest of the proof we will always use the fact that for $x,y\lesssim 1$, $|xy-1|\lesssim |x||y-1|+|x-1|\lesssim |x-1|+|y-1|$.

Then the lemma follows from the following inequalities which follow from \eqref{eq: F_1,F_2} and the fact that $|f(x)-f(y)|\lesssim \|f'(z)\|_{L^{\infty}}|x-y|$.
\beq\label{eq:sum E-E}
\left|\prod_{m=E(\sqrt{|\xi|})}^{m=E(\sqrt{|\eta|})+1}\f{1}{G_m(\xi)}-1\right|
\lesssim |\sqrt{\xi}-\sqrt{\eta}|\f{\nu^{\f13}}{\langle \nu^{\f13}\eta^{\f12}\rangle^{1+\b}}\lesssim |\eta-\xi|\langle\xi\rangle^{-1},
\eeq
and
\beq\label{eq:sum E-l}
\left|\prod_{m=E(\sqrt{|\eta|})}^{l+1}\f{G_m(\eta)}{G_m(\xi)}-1\right|
\lesssim \exp\left(C\f{|\xi-\eta|}{\langle\eta\rangle}\sum_{m=E(\sqrt{|\eta|})}^{l+1}(F^+_2+F^-_2)(m,\eta)\right)-1\lesssim \f{|\xi-\eta|}{\langle\eta\rangle},
\eeq
and
\begin{align*}
&\left|\f{\exp\Big(\nu^{\f13} \langle\nu^{\f13}t_{l,\eta}\rangle^{-(1+\b)}\f{\eta}{l^2}\arctan(D_{l,\eta}^-)\Big)}{\exp\Big(\nu^{\f13} \langle\nu^{\f13}t_{l,\xi}\rangle^{-(1+\b)}\f{\xi}{l^2}\arctan(D_{l,\xi}^-)\Big)}-1\right|
\lesssim \exp\left(C\f{|\xi-\eta|}{\langle\eta\rangle}\right)-1\lesssim \f{|\xi-\eta|}{\langle\eta\rangle},\\
&\left|\f{\exp\Big(\nu^{\f13} \langle\nu^{\f13}t_{l,\eta}\rangle^{-(1+\b)}\f{\eta}{l^2}\Big)}{\exp\Big(\nu^{\f13} \langle\nu^{\f13}t_{l,\xi}\rangle^{-(1+\b)}\f{\xi}{l^2}\Big)}-1\right|
\lesssim \exp\left(C\f{|\xi-\eta|}{\langle\eta\rangle}\right)-1\lesssim \f{|\xi-\eta|}{\langle\eta\rangle},\\
&\left|\exp\Big(\nu^{\f13} \langle\nu^{\f13}t_{l,\eta}\rangle^{-(1+\b)}\f{\eta}{l^2}\big(\arctan(t-\f{\eta}{l})-\arctan(t-\f{\xi}{l})\big)\Big)-1\right|\\
&\lesssim \nu^{\f13} \langle\nu^{\f13}t_{l,\eta}\rangle^{-(1+\b)}\f{\eta}{l^2}\big(\arctan(t-\f{\eta}{l})-\arctan(t-\f{\xi}{l})\big)\\
&\lesssim \nu^{\f13} \langle\nu^{\f13}t_{l,\eta}\rangle^{-(1+\b)}\f{\eta}{l^2}\min\left\{\f{|\xi-\eta|}{l},1\right\}
\lesssim \left\{\begin{aligned}
&\nu^{-\f13}\f{|\xi-\eta|}{\eta},\quad t\approx \f{\eta}{l}\lesssim \nu^{-\f13},\\
&\nu^{\f13\b}t^{1-\b}\f{|\xi-\eta|}{\eta},\quad t\approx \f{\eta}{l}\gtrsim \nu^{-\f13}.
\end{aligned}\right. 
\end{align*}
 
(3b.) $k\neq l$, $|t-\f{\eta}{k}|\gtrsim \f{\eta}{k^2}$ and $|t-\f{\xi}{l}|\gtrsim \f{\xi}{l^2}$ with $k<l$. We have  
\begin{align}\label{eq:w-w}
\nonumber&|w(t,\eta)-w(t,\xi)|=w(t,\xi)\left|\f{w(t,\eta)}{w(t,\xi)}-1\right|\\
\nonumber&\lesssim \left|\prod_{m=E(\sqrt{|\xi|})}^{m=E(\sqrt{|\eta|})+1}\f{1}{G_m(\xi)}-1\right|+
\left|\prod_{m=E(\sqrt{|\eta|})}^{\max\{l,k\}+1}\f{G_m(\eta)}{G_m(\xi)}-1\right|
+\left|\prod_{m=\max\{l,k\}}^{\min\{k,l\}+1}G_m(\eta)-1\right|\\
&\quad+\left|\f{\exp\Big(\nu^{\f13} \langle\nu^{\f13}t_{k,\eta}\rangle^{-(1+\b)}\f{\eta}{k^2}\big(\arctan(t-\f{\eta}{k})+\arctan(D_{k,\eta}^-)\big)\Big)}{\exp\Big(\nu^{\f13} \langle\nu^{\f13}t_{l,\xi}\rangle^{-(1+\b)}\f{\xi}{l^2}\big(\arctan(t-\f{\xi}{l})+\arctan(D_{l,\xi}^-)\big)\Big)}-1\right|.
\end{align}
Let $F_3(t,k,\eta)=\nu^{\f13} \langle\nu^{\f13}t_{k,\eta}\rangle^{-(1+\b)}\f{\eta}{k^2}\big(\arctan(t-\f{\eta}{k})+\arctan(D_{k,\eta}^-)\big)$, we get 
\beno
|F_3|\lesssim k^{-1},\quad 
|\pa_kF_3|\lesssim \f{1}{\langle k\rangle},\quad 
|\pa_{\eta}F_3|\lesssim \f{1}{\langle \eta\rangle},
\eeno
where we use the fact that $\eta\gtrsim k^2$. Then the lemma follows from \eqref{eq:sum E-E}, \eqref{eq:sum E-l} and the following two inequalities
\begin{align*}
&\left|\f{\exp\Big(\nu^{\f13} \langle\nu^{\f13}t_{k,\eta}\rangle^{-(1+\b)}\f{\eta}{k^2}\big(\arctan(t-\f{\eta}{k})+\arctan(D_{k,\eta}^-)\big)\Big)}{\exp\Big(\nu^{\f13} \langle\nu^{\f13}t_{l,\xi}\rangle^{-(1+\b)}\f{\xi}{l^2}\big(\arctan(t-\f{\xi}{l})+\arctan(D_{l,\xi}^-)\big)\Big)}-1\right|\\
&\lesssim e^{F_3(t,k,\eta)-F_3(t,l,\xi)}-1\lesssim |F_3(t,k,\eta)-F_3(t,l,\xi)|\\
&\lesssim \f{1}{\langle k\rangle}|k-l|+\f{1}{\langle \eta\rangle}|\xi-\eta|\lesssim \f{|\xi-\eta|}{\langle \eta\rangle},
\end{align*}
and 
\begin{align*}
\left|\prod_{m=\max\{k,l\}}^{\min\{k,l\}+1}G_m(\eta)-1\right|
\lesssim |l-k|\f{\f{\nu^{\f13}\eta}{l^2}}{\langle\f{\nu^{\f13}\eta}{l}\rangle^{1+\b}}\lesssim \f{|k-l|}{l}\lesssim \f{|\xi-\eta|}{\langle \eta\rangle},
\end{align*}
where we use the fact that $|k-l|\lesssim \langle \f{\xi-\eta}{t}\rangle$, $l\approx \f{\eta}{t}$. 

(3c.) $|\xi-\eta|\gtrsim \f{\xi}{l}\approx\f{\eta}{k}$. 
In this case, we have
\begin{align*}
|w(t,\xi)-w(t,\eta)|\lesssim 1\lesssim k\lesssim \f{|\xi-\eta|}{\langle \eta\rangle}. 
\end{align*}

\no{\bf Case 4.} For $2\eta\leq t\leq 2\xi$, then $t\in I_{1,\xi}$ and 
\begin{align*}
&|w(t,\eta)-w(t,\xi)|=w(t,\xi)\left|\f{w(2\eta,\eta)}{w(t_{l,\xi},\xi)g_{l}(t-\f{\xi}{l},\xi)}-1\right|\\
&\lesssim \left|\f{w(2\eta,\eta)}{w(t_{l,\xi},\xi)g_{l}(t-\f{\xi}{l},\xi)}-1\right|\\
&\lesssim \left|\prod_{m=E(\sqrt{|\xi|})}^{m=E(\sqrt{|\eta|})+1}\f{1}{G_m(\xi)}-1\right|
+\left|\prod_{m=E(\sqrt{|\eta|})}^{2}\f{G_m(\eta)}{G_m(\xi)}-1\right|\\
&\quad+\left|\f{\exp\Big(\nu^{\f13} \langle\nu^{\f13}t_{1,\eta}\rangle^{-(1+\b)}\eta\big(\arctan(\eta)+\arctan(\f13\eta)\big)\Big)}{\exp\Big(\nu^{\f13} \langle\nu^{\f13}t_{1,\xi}\rangle^{-(1+\b)}\xi\big(\arctan(t-\xi)+\arctan(\f13\xi)\big)\Big)}-1\right|.
\end{align*}
Thus the lemma follows from \eqref{eq:sum E-E}, \eqref{eq:sum E-l} and the following inequalities 
\begin{align*}
&\left|\exp\Big(\nu^{\f13} \langle\nu^{\f13}t_{1,\xi}\rangle^{-(1+\b)}\xi\big(\arctan(t-\xi)+\arctan(D_{1,\xi}^-)\big)\Big)-1\right|\\
&\lesssim \bigg|\nu^{\f13} \langle\nu^{\f13}t_{1,\xi}\rangle^{-(1+\b)}\xi\big(\arctan(t-\xi)+\arctan(\f13\xi)\big)\\
&\quad-\Big(\nu^{\f13} \langle\nu^{\f13}t_{1,\eta}\rangle^{-(1+\b)}\eta\big(\arctan(\eta)+\arctan(\f13\eta)\big)\Big)\bigg|\\
&\lesssim \left|\arctan(t-\xi)-\arctan(\eta)\right|+|\xi-\eta|\langle \xi\rangle^{-1}\\
&\lesssim \max\{\arctan(\xi)-\arctan(\eta),\arctan(\eta)-\arctan(2\eta-\xi)\}+|\xi-\eta|\langle \xi\rangle^{-1}\\
&\lesssim |\xi-\eta|\langle \xi\rangle^{-1}. 
\end{align*}
\no{\bf Case 5.} For $t\geq 2\xi$. We get by \eqref{eq:sum E-E} and \eqref{eq:sum E-l} that
\begin{align*}\label{eq:w-w}
\nonumber&|w(2\eta,\eta)-w(2\xi,\xi)|=w(2\xi,\xi)\left|\f{w(2\eta,\eta)}{w(2\xi,\xi)}-1\right|\\
\nonumber&\lesssim \left|\prod_{m=E(\sqrt{|\xi|})}^{m=E(\sqrt{|\eta|})+1}\f{1}{G_m(\xi)}-1\right|+
\left|\prod_{m=E(\sqrt{|\eta|})}^{1}\f{G_m(\eta)}{G_m(\xi)}-1\right|\\
&\lesssim |\xi-\eta|\langle \xi\rangle^{-1}.
\end{align*}
Thus we proved the lemma. 
\end{proof}

\section{Remainder}\label{Sec: Remainder}
In this section we deal with the remainder and prove Proposition \ref{Prop: Remainder}. Now the commutator can not gain us anything so we may as well treat each term separately. We rewrite both terms on the Fourier side:
\begin{align*}
\mathcal{R}&=\sum_{N\in \mathbb{D}}\sum_{N'\approx N}\sum_{k,l}\int_{\eta,\xi} \As\bar{\hat{f}}_{k}(\eta)\As_{k}(\eta)\hat{u}_{l}(\xi)_{N}\widehat{\na f}_{k-l}(\eta-\xi)_{N'}d\eta d\xi\\
&\quad+\sum_{N\in \mathbb{D}}\sum_{N'\approx N}\sum_{k,l}\int_{\eta,\xi} \As\bar{\hat{f}}_{k}(\eta)\hat{u}_{l}(\xi)_{N}\As_{k-l}(\eta-\xi)\widehat{\na f}_{k-l}(\eta-\xi)_{N'}d\eta d\xi.
\end{align*}
On the support of the integrand, $|l,\xi|\approx |k-l,\eta-\xi|$ thus 
\begin{align*}
\As_{k}(\eta)\approx \langle k,\eta\rangle^{\s}\lesssim \langle l,\xi\rangle^{\s}+\langle k-l,\eta-\xi\rangle^{\s}\approx \langle l,\xi\rangle\langle k-l,\eta-\xi\rangle^{\s-1}\approx \As_{k-l}(\eta-\xi),
\end{align*}
which implies that
\begin{align*}
|\mathcal{R}|\lesssim \sum_{N\in \mathbb{D}}\|\As f\|_2\|u_N\|_{H^{3}}\|f_{\sim N}\|_{H^{\s}}. 
\end{align*}

Therefore we get by \eqref{eq: u H^s}
\beq\label{eq: Reminder}
|\mathcal{R}|\lesssim \|f\|_{H^{\s}}^2\|u\|_{H^{3}}\lesssim \f{\ep\nu^{\f13}}{\langle t^2\rangle}\|\As f\|_{2}^2,
\eeq
which gives Proposition \ref{Prop: Remainder}.

\section{Reaction}\label{Sec: reaction}
In this section, we deal with the reaction term and prove Proposition \ref{prop: reaction}. 
Focus first on an individual frequency shell and divide each into several natural pieces
\begin{align*}
R_{N}=R_{N}^1+R_{N}^{\ep,1}+R_{N}^2+R_{N}^3,
\end{align*}
where 
\begin{align*}
R_{N}^1&=\sum_{k,l\neq 0}\int_{\eta,\xi}\As \bar{\hat{f}}_{k}(\eta)\As_{k}(\eta)(\eta l-\xi k)\hat{\phi}_{l}(\xi)_{N}\hat{f}_{k-l}(\eta-\xi)_{<N/8}d\eta d\xi\\
R_{N}^{\ep,1}&=-\sum_{k,l\neq 0}\int_{\eta,\xi}\As\bar{\hat{f}}_{k}(\eta)\As_{k}(\eta)\left[\widehat{(1-v')\na^{\bot}\phi_l}\right](\xi)_N\cdot\widehat{\na f}_{k-l}(\eta-\xi)_{<N/8}d\eta d\xi\\
R^2_N&=\sum_{k}\int_{\eta,\xi}\As\bar{\hat{f}}_{k}(\eta)\As_{k}(\eta)\widehat{g}(\xi)_N\widehat{\pa_vf}_k(\eta-\xi)_{<N/8}d\eta d\xi\\
R_N^3&=-\sum_{k,l}\int_{\eta,\xi}\As\bar{\hat{f}}_{k}(\eta)\As_{k-l}(\eta-\xi)\hat{u}_{l}(\xi)_{N}\widehat{\na f}_{k-l}(\eta-\xi)_{<N/8}d\eta d\xi.
\end{align*}
\subsection{Main contribution}
The main contribution comes from $R_{N}^1$. We subdivide this integral depending on whether or not $(l,\xi)$ and/or $(k,\eta)$ are resonant as each combination requires a slightly different treatment. Define the partition: 
\begin{align*}
1&=1_{t\notin I_{l,\eta}, t\notin I_{l,\xi}}+1_{t\notin I_{l,\eta}, t\in I_{l,\xi}}+1_{t\in I_{l,\eta}, t\notin I_{l,\xi}}+1_{t\in I_{l,\eta}, t\in I_{l,\xi}}\\
&=\chi^{NR,NR}+\chi^{NR,R}+\chi^{R,NR}+\chi^{R,R},
\end{align*}
where the $NR$ and $R$ denotes 'non-resonant' and 'resonant' respectively referring to $(k,\eta)$ and $(l,\xi)$. Correspondingly, denote 
\begin{align*}
R_{N}^1&=\underbrace{\sum_{l\neq 0}\int_{\eta,\xi}\chi^D
 \As\bar{\hat{f}}_{l}(\eta)\As_{l}(\eta)(\eta l-\xi l)\hat{\phi}_{l}(\xi)_{N}\hat{f}_{0}(\eta-\xi)_{<N/8}d\eta d\xi}_{R_{N,D}}\\
&\quad+\sum_{l\neq 0}\int_{\eta,\xi}\left[\chi^{NR,NR}+\chi^{NR,R}+\chi^{R,NR}+\chi^{R,R}\right](1-\chi^D)\\
&\quad\quad\underbrace{\quad\quad\quad \times
 \As\bar{\hat{f}}_{l}(\eta)\As_{l}(\eta)(\eta l-\xi l)\hat{\phi}_{l}(\xi)_{N}\hat{f}_{0}(\eta-\xi)_{<N/8}d\eta d\xi}_{R_{N,=}^{NR,NR}+R_{N,=}^{NR,R}+R_{N,=}^{R,NR}+R_{N,=}^{R,R}}\\
&\quad+\underbrace{\sum_{k,l\neq 0,k\neq l}\int_{\eta,\xi}
(1-\chi^{D_1})\As\bar{\hat{f}}_{k}(\eta)\As_{k}(\eta)(\eta l-\xi k)\hat{\phi}_{l}(\xi)_{N}\hat{f}_{k-l}(\eta-\xi)_{<N/8}d\eta d\xi}_{R_{N,\neq,*}}\\
&\quad+\sum_{k,l\neq 0,k\neq l}\int_{\eta,\xi}
\left[\chi^{NR,NR}+\chi^{NR,R}+\chi^{R,NR}+\chi^{R,R}\right]\chi^{D_1}\\
&\quad\quad\underbrace{\quad\quad\quad \times \As\bar{\hat{f}}_{k}(\eta)\As_{k}(\eta)(\eta l-\xi k)\hat{\phi}_{l}(\xi)_{N}\hat{f}_{k-l}(\eta-\xi)_{<N/8}d\eta d\xi}_{R_{N}^{NR,NR}+R_{N}^{NR,R}+R_{N}^{R,NR}+R_{N}^{R,R}}\\
&=R_{N,D}+R_{N,=}^{NR,NR}+R_{N,=}^{NR,R}+R_{N,=}^{R,NR}+R_{N,=}^{R,R}\\
&\quad +R_{N,\neq,*}+R_{N}^{NR,NR}+R_{N}^{NR,R}+R_{N}^{R,NR}+R_{N}^{R,R},
\end{align*}
where $\chi^D$ is a characteristic function (the indicator function) of the set 
\beno
D=\left\{(l,\xi):\ |l|\geq \f54|\xi|\right\}, 
\eeno
and $\chi^{D_1}$ is a characteristic function (the indicator function) of the set 
\beno
D_1=\left\{(l,k,\xi,\eta):\ |l|\leq |\xi|,\ |l-k,\xi-\eta|\leq \f{1}{1000}|l,\xi|\right\}, 
\eeno

\subsubsection{Treatment of $R_{N,D}$}
For the case $|l|\geq \f54|\xi|$, we get for $t\geq 1$, 
\beno
\f{\As_{l}(\eta)|l|}{l^2+|lt-\xi|^2}\lesssim \f{\As_l(\xi)|l|}{l^2+t^2l^2}\lesssim \f{|l|^{\s-1}}{1+t^2},
\eeno
which implies that
\ben\label{eq: R_N,D}
\begin{split}
|R_{N,D}|&\lesssim \f{1}{\langle t^2\rangle}\|\As f_{\sim N}\|_{L^2}\|\langle\pa_z\rangle^{\s-1} \Delta_L\Delta_t^{-1}P_{\neq}f_N\|_{L^2}
\|f_{0}\|_{H^{3}}\\
&\lesssim \f{\ep \nu^{\f13}}{\langle t^2\rangle}\|\As f\|_2^2.
\end{split}
\een

The next 4 subsections we will use the fact that for $|l|\leq \f54|\xi|$ and $|\xi-\eta|\leq \f{3}{8}|l,\xi|\leq \f{27}{40}|\xi|$, we have $|\eta|\geq \f{13}{40}|\xi|\geq \f{13}{50}|l|$, which gives that 
\beq\label{eq: w_l=w}
w_l(t,\eta)=w(t,\eta), \quad w_l(t,\xi)=w(t,\xi). 
\eeq
\subsubsection{Treatment of the zero mode $R_{N,=}^{NR,NR}$}
We have 
\begin{align*}
R_{N,=}^{NR,NR}&\lesssim 
\sum_{l\neq 0}\int_{\eta,\xi}\chi^{NR,NR}(1-\chi^D)
\left|\As\bar{\hat{f}}_{l}(\eta)\f{\As_{l}(\eta)|l|}{l^2+|lt-\xi|^2}\widehat{\Delta_L\Delta_t^{-1}f}_{l}(\xi)_{N}\widehat{\pa_vf}_{0}(\eta-\xi)_{<N/8}\right|d\eta d\xi
\end{align*}

According to the relation between $t$ and $\xi$, we have the following 5 cases. 

{\bf Case 1. } $t\leq \max\{t(\xi),t(\eta)\}\approx \sqrt{|\xi|}\approx \sqrt{N}$. Then in this case
\ben\label{eq: tbd}
\f{\As_{l}(\eta)|l|}{l^2+|lt-\xi|^2}\lesssim \f{\As_l(\xi)|l|}{l^2(1+\f{|\xi|^2}{l^4})}\lesssim \f{\As_l(\xi)}{\sqrt{N}},
\een
which implies
\begin{align*}
|R_{N,=}^{NR,NR}|
&\lesssim \f{1}{\langle t\rangle}\|\As f_{\sim N}\|_{L^2}\|\As\Delta_L\Delta_t^{-1}P_{\neq}f_N\|_{L^2}\|\pa_vf_{0}\|_{H^{3}}\\
&\lesssim \f{\ep\nu^{\f13}}{\langle \nu^{\f12}t^{\f12}\rangle\langle t\rangle}\|\As f_{\sim N}\|_{L^2}\|\As\Delta_L\Delta_t^{-1}P_{\neq}f_N\|_{L^2}.
\end{align*}

{\bf Case 2. } $t\geq 2|\xi|$ or $t\geq 2|\eta|$. Then in this case, 
\ben\label{eq:big1}
\f{|l|}{l^2+|lt-\xi|^2}\lesssim \f{1}{1+t^2},
\een
and
\begin{align}
\f{|l|}{l^2+|lt-\xi|^2}
\nonumber&\lesssim \f{|l|}{l^2+|lt-\xi|^2}-\f{|l|}{l^2+|lt-\eta|^2}+\f{|l|}{l^2+|lt-\eta|^2}\\
\nonumber&\lesssim \f{|l||\eta-\xi|(2|\xi-lt|+|\xi-\eta|)}{(l^2+|lt-\eta|^2)(l^2+|lt-\xi|^2)}+\f{|l|}{l^2+|lt-\eta|^2}\\
&\lesssim \f{|l|\langle\xi-\eta\rangle^2}{l^2+|lt-\eta|^2}
\lesssim \f{\langle\xi-\eta\rangle^2}{1+t^2},\label{eq:big2}
\end{align}
which implies that
\begin{align*}
|R_{N,=}^{NR,NR}|
&\lesssim \f{1}{\langle t^2\rangle}\|\As f_{\sim N}\|_{L^2}\|\As\Delta_L\Delta_t^{-1}P_{\neq}f_N\|_{L^2}
\|f_{0}\|_{H^{5}}\\
&\lesssim \f{\ep\nu^{\f13}}{\langle t^2\rangle}\|\As f_{\sim N}\|_{L^2}\|\As\Delta_L\Delta_t^{-1}P_{\neq}f_N\|_{L^2}.
\end{align*}

{\bf Case 3. } $t\in I_t(\xi)\cap I_t(\eta)$. In this case there exists $k,l'$ so that $t\in I_{k,\eta}\cap I_{l',\xi}$. By Lemma \ref{lem: 3.2}, we get $k\approx l'$. 

If $l\leq \f12\min\{|l'|,|k|\}$, then 
\begin{align*}
\f{|l|}{l^2(1+|t-\f{\xi}{l}|^2)}
&\lesssim \sqrt{\f{|l|}{l^2(1+|t-\f{\eta}{l}|^2)}}\sqrt{\f{|l|}{l^2(1+|t-\f{\xi}{l}|^2)}}
\langle\xi-\eta\rangle\\
&\lesssim \f{\langle\xi-\eta\rangle}{l(1+\f{\xi^2}{l^2})}
\lesssim \f{\langle\xi-\eta\rangle}{1+\f{\xi^2}{l'^2}}\lesssim  \f{\langle\xi-\eta\rangle}{1+t^2},
\end{align*}
here we use the fact that $t\approx \f{\xi}{l'}$. 

If $l>\f12\min\{|l'|,|k|\}$, 
then $\f{1}{l^2}\lesssim\f{1}{\min\{|l'|,|k|\}^2}\lesssim \f{1}{|l'k|}$ 
and by the fact that for $t\in I_{l',\xi}\cap I_{k,\eta}$ and $t\notin I_{l,\xi}\cup I_{l,\eta}$, 
it holds that $|t-\f{\eta}{l}|\geq |t-\f{\eta}{k}|$ and $|t-\f{\xi}{l}|\geq |t-\f{\xi}{l'}|$. 
Thus we get by \eqref{eq: w_l=w}, 
\begin{align*}
\f{|l|\langle\nu^{\f12}(\f{\xi}{k})^{\f12}\rangle^{-1}}{l^2(1+|t-\f{\xi}{l}|^2)}
&\lesssim \sqrt{\f{|l|}{l^2(1+|t-\f{\eta}{l}|^2)}}\sqrt{\f{|l|}{l^2(1+|t-\f{\xi}{l}|^2)}}
\langle\xi-\eta\rangle \langle\nu^{\f12}(\f{\xi}{k})^{\f12}\rangle^{-1}\\
&\lesssim \sqrt{\f{1}{k(1+|t-\f{\eta}{k}|^2)}}\sqrt{\f{1}{l'(1+|t-\f{\xi}{l'}|^2)}}
\langle\xi-\eta\rangle \langle\nu^{\f12}(\f{\xi}{k})^{\f12}\rangle^{-1}\\
&\lesssim \sqrt{\f{\eta}{k^2(1+|t-\f{\eta}{k}|^2)}}\sqrt{\f{\xi}{l'^2(1+|t-\f{\xi}{l'}|^2)}}
\langle\xi-\eta\rangle \langle\nu^{\f13}(\f{\xi}{k})\rangle^{-(1+\b)}\\
&\lesssim \sqrt{\f{\pa_tw(t,\eta)}{w(t,\eta)}}\sqrt{\f{\pa_tw(t,\xi)}{w(t,\xi)}}
\langle\xi-\eta\rangle \nu^{-\f13}
\end{align*}
In the third inequality, we use the fact that for $0<\b\leq \f12$, 
\beq\label{eq: beta<1/2}
\langle\nu^{\f12}(\f{\xi}{k})^{\f12}\rangle^{-1}\lesssim \f{\xi}{k}\langle\nu^{\f13}(\f{\xi}{k})\rangle^{-(1+\b)}.
\eeq

Thus by the fact that $t\approx \f{\xi}{k}$, we get 
\begin{align*}
|R_{N,=}^{NR,NR}|
&\lesssim \f{1}{\langle t^2\rangle}\|\As f_{\sim N}\|_{L^2}\|\As\Delta_L\Delta_t^{-1}f_N\|_{L^2}
\| f_{0}\|_{H^{4}}\\
&\quad+\left\|\sqrt{\f{\pa_tw}{w}}\As f_{\sim N}\right\|_{L^2}
\left\|\sqrt{\f{\pa_tw}{w}}\chi_R\As \Delta_L\Delta_t^{-1}f_N\right\|_{L^2}
\left\|\nu^{-\f13}(\nu^{\f12}t^{\f12}\pa_v) f_{0}\right\|_{H^{3}}\\
&\lesssim \f{\ep\nu^{\f13}}{\langle t^2\rangle}\|\As f_{\sim N}\|_{L^2}\|\As\Delta_L\Delta_t^{-1}f_N\|_{L^2}
+\ep\left\|\sqrt{\f{\pa_tw}{w}}\As f_{\sim N}\right\|_{L^2}
\left\|\sqrt{\f{\pa_tw}{w}}\chi_R\As \Delta_L\Delta_t^{-1}P_{\neq}f_N\right\|_{L^2}.
\end{align*}

Putting together all the above estimates, we conclude that
\beq\label{eq: R_N,=^NR,NR}
\begin{split}
|R_{N,=}^{NR,NR}|
&\lesssim \f{\ep\nu^{\f13}}{\langle \nu^{\f12}t^{\f32}\rangle}\|\As f_{\sim N}\|_{L^2}\|\As\Delta_L\Delta_t^{-1}f_N\|_{L^2}
+\f{\ep\nu^{\f13}}{\langle t^2\rangle}\|\As f_{\sim N}\|_{L^2}\|\As\Delta_L\Delta_t^{-1}f_N\|_{L^2}
\\
&\quad+\ep\left\|\sqrt{\f{\pa_tw}{w}}\As f_{\sim N}\right\|_{L^2}
\left\|\sqrt{\f{\pa_tw}{w}}\chi_R\As \Delta_L\Delta_t^{-1}P_{\neq}f_N\right\|_{L^2}.
\end{split}
\eeq

\subsubsection{Treatment of the zero mode $R_{N,=}^{NR,R}$}
Since $t\in I_{l,\xi}$, if $t\leq t(\eta)\approx \sqrt{|\xi|}\approx \sqrt{N}$, then $l\approx \sqrt{|\xi|}$, thus by the fact that 
\beno
\f{\As_{l}(\eta)|l|}{l^2+|lt-\xi|^2}\lesssim \f{\As_l(\xi)}{|l|}\lesssim \f{\As_l(\xi)}{\sqrt{N}},
\eeno
we obtain that
\begin{align*}
|R_{N,=}^{NR,R}|
&\lesssim \f{1}{\langle t\rangle}
\|\As f_{\sim N}\|_{L^2}
\|\As\Delta_L\Delta_t^{-1}f_N\|_{L^2}
\|\pa_vf_{0}\|_{H^{3}}\\
&\lesssim \f{\ep\nu^{\f13}}{\langle \nu^{\f12}t^{\f32}\rangle}
\|\As f_{\sim N}\|_{L^2}
\|\As\Delta_L\Delta_t^{-1}f_N\|_{L^2}.
\end{align*}

If $t\geq 2|\eta|$, then we use the same argument as  \eqref{eq:big1} and \eqref{eq:big2} and get that
\beno
\f{|l|\As_k(\eta)}{l^2+(\xi-lt)^2}\lesssim \f{\langle \xi-\eta\rangle^2}{\langle t\rangle^2}\As_l(\xi).
\eeno
Therefore we get
\begin{align*}
|R_{N,=}^{NR,R}|
&\lesssim \f{1}{\langle t^2\rangle}\|\As f_{\sim N}\|_{L^2}\|\As \Delta_L\Delta_t^{-1}f_N\|_{L^2}
\| f_{0}\|_{H^{5}}\\
&\lesssim \f{\ep\nu^{\f13}}{\langle t^2\rangle}\|\As f_{\sim N}\|_{L^2}\|\As \Delta_L\Delta_t^{-1}f_N\|_{L^2}.
\end{align*}

If $t\in I_t(\eta)$, then there is $k\neq l$ such that $t\in I_{k,\eta}\cap I_{l,\xi}$.  By the fact that for $t\in I_{k,\eta}$ and $t\notin I_{l,\eta}$, $|t-\f{\eta}{l}|> |t-\f{\eta}{k}|$, then we get by \eqref{eq: beta<1/2} that
\begin{align*}
\f{|l|\langle\nu^{\f12}(\f{\xi}{k})^{\f12}\rangle^{-1}}{l^2(1+|t-\f{\xi}{l}|^2)}
&\lesssim \sqrt{\f{|l|}{l^2(1+|t-\f{\eta}{l}|^2)}}\sqrt{\f{|l|}{l^2(1+|t-\f{\xi}{l}|^2)}}
\langle\xi-\eta\rangle \langle\nu^{\f12}(\f{\xi}{k})^{\f12}\rangle^{-1}\\
&\lesssim \sqrt{\f{1}{k(1+|t-\f{\eta}{k}|^2)}}\sqrt{\f{1}{l(1+|t-\f{\xi}{l}|^2)}}
\langle\xi-\eta\rangle \langle\nu^{\f12}(\f{\xi}{k})^{\f12}\rangle^{-1}\\
&\lesssim \sqrt{\f{\eta}{k^2(1+|t-\f{\eta}{k}|^2)}}\sqrt{\f{\xi}{l^2(1+|t-\f{\xi}{l}|^2)}}
\langle\xi-\eta\rangle \langle\nu^{\f13}(\f{\xi}{k})\rangle^{-(1+\b)}\\
&\lesssim \sqrt{\f{\pa_tw(t,\eta)}{w(t,\eta)}}\sqrt{\f{\pa_tw(t,\xi)}{w(t,\xi)}}
\langle\xi-\eta\rangle \nu^{-\f13}.
\end{align*}
Therefore, we get 
\begin{align*}
|R_{N,=}^{NR,R}|
&\lesssim 
\left\|\sqrt{\f{\pa_tw}{w}}\As f_{\sim N}\right\|_{L^2}
\left\|\sqrt{\f{\pa_tw}{w}}\As \chi_R\Delta_L\Delta_t^{-1}f_N\right\|_{L^2}
\left\|\nu^{-\f13}(\nu^{\f12}t^{\f12}\pa_v) f_{0}\right\|_{H^{3}}\\
&\lesssim 
\ep \left\|\sqrt{\f{\pa_tw}{w}}\As f_{\sim N}\right\|_{L^2}
\left\|\sqrt{\f{\pa_tw}{w}}\As \chi_R\Delta_L\Delta_t^{-1}f_N\right\|_{L^2}.
\end{align*}

Putting together all the above estimates, we conclude that
\beq\label{eq: R_N,=^NR,R}
\begin{split}
|R_{N,=}^{NR,R}|
&\lesssim \f{\ep\nu^{\f13}}{\langle \nu^{\f12}t^{\f32}\rangle}
\|\As f_{\sim N}\|_{L^2}
\|\As\Delta_L\Delta_t^{-1}f_N\|_{L^2}
+\f{\ep\nu^{\f13}}{\langle t^2\rangle}\|\As f_{\sim N}\|_{L^2}\|\As \Delta_L\Delta_t^{-1}f_N\|_{L^2}\\
&\quad+\ep\left\|\sqrt{\f{\pa_tw}{w}}\As f_{\sim N}\right\|_{L^2}
\left\|\sqrt{\f{\pa_tw}{w}}\chi_R\As \Delta_L\Delta_t^{-1}f_N\right\|_{L^2}.
\end{split}
\eeq

\subsubsection{Treatment of the zero mode $R_{N,=}^{R,NR}$}
Since $t\in I_{l,\eta}$, if $t\leq t(\xi)\approx \sqrt{|\xi|}$, then $l\approx \sqrt{|\xi|}$. By using \eqref{eq: tbd}, we get 
\beno
|R_{N,=}^{R,NR}|\lesssim \f{1}{\langle t^2\rangle}\|\As f_{\sim N}\|_{L^2}\|\As\Delta_L\Delta_t^{-1}f_N\|_{L^2}
\| f_{0}\|_{H^{5}}
\lesssim \f{\ep\nu^{\f13}}{\langle t^2\rangle}\|\As f_{\sim N}\|_{L^2}\|\As\Delta_L\Delta_t^{-1}f_N\|_{L^2}.
\eeno

Similarly if $t\geq 2|\xi|$, then by \eqref{eq:big2}, we get  
\beno
|R_{N,=}^{R,NR}|\lesssim \f{1}{\langle t^2\rangle}\|\As f_{\sim N}\|_{L^2}\|\As\Delta_L\Delta_t^{-1}f_N\|_{L^2}
\|f_{0}\|_{H^{5}}
\lesssim \f{\ep\nu^{\f13}}{\langle t^2\rangle}\|\As f_{\sim N}\|_{L^2}\|\As\Delta_L\Delta_t^{-1}f_N\|_{L^2}.
\eeno
If $t\in I_t(\xi)$, then there is $k\neq l$ such that $t\in I_{k,\xi}\cap I_{l,\eta}$.  By the fact that for $t\in I_{k,\xi}$ and $t\notin I_{l,\xi}$, $|t-\f{\xi}{l}|> |t-\f{\xi}{k}|$, then we get by \eqref{eq: beta<1/2} that
\begin{align*}
\f{|l|\langle\nu^{\f12}(\f{\xi}{k})^{\f12}\rangle^{-1}}{l^2(1+|t-\f{\xi}{l}|^2)}
&\lesssim \sqrt{\f{|l|}{l^2(1+|t-\f{\xi}{l}|^2)}}\sqrt{\f{|l|}{l^2(1+|t-\f{\eta}{l}|^2)}}
\langle\xi-\eta\rangle \langle\nu^{\f12}(\f{\xi}{k})^{\f12}\rangle^{-1}\\
&\lesssim \sqrt{\f{1}{k(1+|t-\f{\xi}{k}|^2)}}\sqrt{\f{1}{l(1+|t-\f{\eta}{l}|^2)}}
\langle\xi-\eta\rangle \langle\nu^{\f12}(\f{\xi}{k})^{\f12}\rangle^{-1}\\
&\lesssim \sqrt{\f{\xi}{k^2(1+|t-\f{\xi}{k}|^2)}}\sqrt{\f{\eta}{l^2(1+|t-\f{\eta}{l}|^2)}}
\langle\xi-\eta\rangle \langle\nu^{\f13}(\f{\xi}{k})\rangle^{-(1+\b)}\\
&\lesssim \sqrt{\f{\pa_tw(t,\eta)}{w(t,\eta)}}\sqrt{\f{\pa_tw(t,\xi)}{w(t,\xi)}}
\langle\xi-\eta\rangle \nu^{-\f13}.
\end{align*}
Therefore, we get 
\begin{align*}
|R_{N,=}^{R,NR}|
&\lesssim 
\left\|\sqrt{\f{\pa_tw}{w}}\As f_{\sim N}\right\|_{L^2}
\left\|\sqrt{\f{\pa_tw}{w}}\chi_R\As \Delta_L\Delta_t^{-1}f_N\right\|_{L^2}
\left\|\nu^{-\f13}(\nu^{\f12}t^{\f12}\pa_v) f_{0}\right\|_{H^{3}}\\
&\lesssim 
\ep \left\|\sqrt{\f{\pa_tw}{w}}\As f_{\sim N}\right\|_{L^2}
\left\|\sqrt{\f{\pa_tw}{w}}\chi_R\As \Delta_L\Delta_t^{-1}f_N\right\|_{L^2}.
\end{align*}

Putting together all the above estimates, we conclude that
\beq\label{eq: R_N,=^R,NR}
\begin{split}
|R_{N,=}^{R,NR}|
&\lesssim 
\f{\ep\nu^{\f13}}{\langle t^2\rangle}\|\As f_{\sim N}\|_{L^2}\|\As \Delta_L\Delta_t^{-1}f_N\|_{L^2}\\
&\quad+\ep\left\|\sqrt{\f{\pa_tw}{w}}\As f_{\sim N}\right\|_{L^2}
\left\|\sqrt{\f{\pa_tw}{w}}\chi_R\As \Delta_L\Delta_t^{-1}f_N\right\|_{L^2}.
\end{split}
\eeq

\subsubsection{Treatment of the zero mode $R_{N,=}^{R,R}$}
We get by \eqref{eq: beta<1/2} that
\begin{align*}
\f{|l|\langle\nu^{\f12}(\f{\xi}{l})^{\f12}\rangle^{-1}}{l^2(1+|t-\f{\xi}{l}|^2)}
&\lesssim \sqrt{\f{|l|}{l^2(1+|t-\f{\xi}{l}|^2)}}\sqrt{\f{|l|}{l^2(1+|t-\f{\eta}{l}|^2)}}
\langle\xi-\eta\rangle \langle\nu^{\f12}(\f{\xi}{l})^{\f12}\rangle^{-1}\\
&\lesssim \sqrt{\f{\xi}{l^2(1+|t-\f{\xi}{l}|^2)}}\sqrt{\f{\eta}{l^2(1+|t-\f{\eta}{l}|^2)}}
\langle\xi-\eta\rangle \langle\nu^{\f13}(\f{\xi}{l})\rangle^{-(1+\b)}\\
&\lesssim \sqrt{\f{\pa_tw(t,\eta)}{w(t,\eta)}}\sqrt{\f{\pa_tw(t,\xi)}{w(t,\xi)}}
\langle\xi-\eta\rangle \nu^{-\f13},
\end{align*}
which gives 
\begin{align}\label{eq: R_N,=^R,R}
\begin{split}
|R_{N,=}^{R,R}|
&\lesssim 
\left\|\sqrt{\f{\pa_tw}{w}}\As f_{\sim N}\right\|_{L^2}
\left\|\sqrt{\f{\pa_tw}{w}}\chi_R\As \Delta_L\Delta_t^{-1}f_N\right\|_{L^2}
\left\|\nu^{-\f13}(\nu^{\f12}t^{\f12}\pa_v)f_{0}\right\|_{H^{3}}\\
&\lesssim 
\ep\left\|\sqrt{\f{\pa_tw}{w}}\As f_{\sim N}\right\|_{L^2}
\left\|\sqrt{\f{\pa_tw}{w}}\chi_R\As \Delta_L\Delta_t^{-1}f_N\right\|_{L^2}.
\end{split}
\end{align}

\subsubsection{Treatment of $R_{N,\neq,*}$}
In this case, we get $(l,k,\xi,\eta)\notin D_1$ which means that at least one of the following two $|l|\geq |\xi|$, $|l-k,\xi-\eta|\geq \f{1}{1000}|l,\xi|$ holds. Thus we get
\beno
\f{\As_{k}(\eta)|l,\xi|}{l^2+|lt-\xi|^2}\lesssim \As_l(\xi)\langle l-k,\xi-\eta\rangle,
\eeno
which implies 
\ben\label{eq: R_N,neq,*}
\begin{split}
|R_{N,\neq,*}|
&\lesssim \|\As f_{\sim N}\|_{2}
\|\As\Delta_L\Delta_t^{-1}f_{N}\|_{2}\|f_{\neq}\|_{H^{4}}\\
&\lesssim \f{\ep\nu^{\f13}}{\langle \nu t^3\rangle}\|\As f_{\sim N}\|_{2}
\|\As\Delta_L\Delta_t^{-1}f_{N}\|_{2}
\end{split}
\een

The next 4 subsections, we restriction $(l,k,\xi,\eta)\in D_1$, which gives $|\eta|\geq \f{1}{20}|k|$ and $|\xi|\geq \f{1}{20}|l|$, thus 
\beq
w_k(t,\eta)=w(t,\eta),\quad w_l(t,\xi)=w(t,\xi). 
\eeq

\subsubsection{Treatment of $R_{N}^{NR,NR}$}
Since we restrict the integrand in $D_1$, it holds that 
\beq\label{eq: k-l<k}
\left||l,\xi|-|k,\eta|\right|\leq |k-l,\eta-\xi|\leq \f{1}{1000}|l,\xi|.
\eeq
It follows form the fact that
\beno
(\eta l-\xi k)=(\eta-\xi)l +(l-k) \xi, 
\eeno 
and Lemma \ref{lem: total-growth} that 
\begin{align*}
|R_{N}^{NR,NR}|&\lesssim \sum_{k,l\neq 0,\, k\neq l}\int_{\eta,\xi}1_{t\notin I_{k,\eta}, t\notin I_{l,\xi}}
|\As\bar{\hat{f}}_{k}(\eta)|
\left|\f{\As_{l}(\xi)|l,\xi|}{l^2+(\xi-lt)^2}\widehat{\Delta_L\Delta_t^{-1}f}_{l}(\xi)_{N}\widehat{\na f}_{k-l}(\eta-\xi)_{<N/8}\right|d\eta d\xi\\
&\lesssim \sum_{k,l\neq 0,\, k\neq l}\int_{\eta,\xi}1_{t\notin I_{k,\eta}, t\notin I_{l,\xi}}\As|\bar{\hat{f}}_{k}(\eta)|
\f{\As_{l}(\xi)|l,\xi|}{l^2(1+\f{\xi^2}{l^4})}
|\widehat{\Delta_L\Delta_t^{-1}f}_{l}(\xi)_{N}\widehat{\na f}_{k-l}(\eta-\xi)_{<N/8}|d\eta d\xi.
\end{align*}
Therefore we have 
\beq\label{R_N^NR,NR}
\begin{split}
|R_{N}^{NR,NR}|
&\lesssim \|\As f_{\sim N}\|_{2}\|\As\Delta_L\Delta_t^{-1}f_{N}\|_{2}\|f_{\neq}\|_{H^{3}}\\
&\lesssim \f{\ep\nu^{\f13}}{\langle \nu t^3\rangle}\|\As f_{\sim N}\|_{2}
\|\As\Delta_L\Delta_t^{-1}f_{N}\|_{2}.
\end{split}
\eeq

\subsubsection{Treatment of $R_{N}^{NR,R}$} 
By \eqref{eq: k-l<k}, we have $\langle k,\eta\rangle^{\s}\approx \langle l,\xi\rangle^{\s}$ and $A_l(\xi)\approx A_k(\eta)$ which gives
\begin{align*}
|R_{N}^{NR,R}|
&\lesssim \sum_{k,l\neq 0,\, k\neq l}\int_{|\eta-\xi|\leq \f{|\eta|}{100}}\chi^{NR,R}
|\As\bar{\hat{f}}_{k}(\eta)|
\f{\As_l(\xi)|l,\xi|}{l^2+(\xi-lt)^2}
|\widehat{\Delta_{L}\Delta_t^{-1}f}_{l}(\xi)_{N}\widehat{\na f}_{k-l}(\eta-\xi)_{<N/8}|d\eta d\xi\\
&\quad \sum_{k,l\neq 0,\, k\neq l}\int_{|\eta-\xi|> \f{|\eta|}{100}}\chi^{NR,R}
|\As\bar{\hat{f}}_{k}(\eta)|\f{\As_l(\xi)|l,\xi|}{l^2+(\xi-lt)^2}
|\widehat{\Delta_{L}\Delta_t^{-1}f}_{l}(\xi)_{N}\widehat{\na f}_{k-l}(\eta-\xi)_{<N/8}|d\eta d\xi\\
&=R_{N,1}^{NR,R}+R_{N,2}^{NR,R}.
\end{align*}
Let us first deal with $R_{N,1}^{NR,R}$, so that $|\eta-\xi|\leq \f{|\eta|}{100}$. 

If $t\leq t(\eta)\approx \sqrt{|\eta|}$, then $w_{k}(\eta)=1$ and by the fact that $t\geq t(\xi)\approx \sqrt{|\eta|}$, we get that $t\approx \sqrt{|\xi|}$, $|l|\approx\sqrt{|\xi|}$ and then 
\begin{align*}
|R_{N,1}^{NR,R}|
&\lesssim \sum_{k,l\neq 0,\, k\neq l}\int_{\eta,\xi}1_{t\notin I_{k,\eta}, t\in I_{l,\xi}}
|\As\bar{\hat{f}}_{k}(\eta)|
\As|\widehat{\Delta_{L}\Delta_t^{-1}f}_{l}(\xi)_{N}\widehat{\na f}_{k-l}(\eta-\xi)_{<N/8}|d\eta d\xi\\
&\lesssim \|\As f_{\sim N}\|_{2}\|\As\Delta_{L}\Delta_t^{-1}f_N\|_{2}\|f_{\neq}\|_{H^{3}}.
\end{align*}

If $t\geq 2|\eta|$, and for $|\eta-\xi|\leq \f{1}{100}|\eta|$, $|l|t-|\xi|\geq 2|\eta|-|\xi|\geq |\eta|-\big||\eta|-|\xi|\big|\gtrsim |\xi|$, which implies 
\begin{align*}
|R_{N,1}^{NR,R}|
\lesssim \|\As f_{\sim N}\|_{L^2}\|\As\Delta_{L}\Delta_t^{-1}f_N\|_{H^{\s}}\|f_{\neq}\|_{H^{3}}.
\end{align*}

If $t\in I_t(\eta)$, there exists $k'\in [1, E(\sqrt{|\eta|})]$ such that $t\in I_{k',\eta}\cap I_{l,\xi}$. Then by Lemma \ref{lem: 3.2}, we need to consider the following three cases: 

(a.) $k'=l$. In this case, by using the fact that 
\begin{align*}
\f{|\xi|}{l^2+(\xi-lt)^2}
&\lesssim \sqrt{\f{|\xi|}{l^2+(\xi-lt)^2}}\sqrt{\f{|\eta|}{l^2+(\eta-lt)^2}}\langle \xi-\eta\rangle\\
&\lesssim \langle\nu^{\f13}t\rangle^{1+\b}\nu^{-\f13}\langle \xi-\eta\rangle \sqrt{\f{\pa_tw(t,\xi)}{w(t,\xi)}}\sqrt{\f{\pa_tw(t,\eta)}{w(t,\eta)}},
\end{align*}
and $|l|\lesssim \sqrt{|\xi|}\lesssim |\xi|$, we get
\begin{align*}
&|R_{N,1}^{NR,R}|
\lesssim \sum_{k,l\neq 0,\, k\neq l}\int_{|\eta-\xi|\leq \f{|\eta|}{100}}1_{t\in I_{k',\eta}\cap I_{l,\xi}}
|\As\bar{\hat{f}}_{k}(\eta)|
\f{\As_l(t,\xi)|\xi|}{l^2+(\xi-lt)^2}
|\widehat{\Delta_{L}\Delta_t^{-1}f}_{l}(\xi)_{N}\widehat{\na f}_{k-l}(\eta-\xi)_{<N/8}|d\eta d\xi\\
&\lesssim \sum_{k,l\neq 0,\, k\neq l}\int_{|\eta-\xi|\leq \f{1}{100}|\eta|}1_{t\in I_{l,\eta}\cap I_{l,\xi}}\left|\sqrt{\f{\pa_t{w}(\eta)}{{w}(\eta)}}\As\bar{\hat{f}}_{k}(\eta)\right|\\
&\quad \quad \quad \quad\quad \quad\quad \quad
\times \left|\sqrt{\f{\pa_t{w}(\eta)}{{w}(\eta)}}\As\widehat{\Delta_{L}\Delta_t^{-1}f}_{l}(\xi)_{N}
(1+\nu^{\f13}t)^2\nu^{-\f13}\widehat{\langle\na\pa_v\rangle f}_{k-l}(\eta-\xi)_{<N/8}\right|d\eta d\xi\\
&\lesssim \left\|\sqrt{\f{\pa_t{w}(\eta)}{{w}(\eta)}}\As f_{\sim N}\right\|_{2}
\left\|\sqrt{\f{\pa_t{w}(\xi)}{{w}(\xi)}}\chi_R\As\Delta_{L}\Delta_t^{-1}f_{N}\right\|_{2}
\left\|\langle\nu^{\f13}t\rangle^{1+\b}\nu^{-\f13} f_{\neq}\right\|_{H^{4}}.
\end{align*}

(b.) $|t-\f{\eta}{k'}|\gtrsim\f{\eta}{k'^2}$ and $|t-\f{\xi}{l}|\gtrsim \f{\xi}{l^2}$. In this case, by using the fact that
\begin{align}\label{eq: bd-Delta^-1}
\f{|l,\xi|}{l^2+(\xi-lt)^2}\lesssim \f{|\xi|/l^2}{1+(\xi/l-t)^2}
\lesssim \f{|\xi|/l^2}{1+\f{\xi^2}{l^4}}\lesssim 1, 
\end{align}
we obtain that
\begin{align*}
&|R_{N,1}^{NR,R}|
\lesssim \|\As f_{\sim N}\|_{L^2}\|\As\Delta_{L}\Delta_t^{-1}P_{\neq}f_N\|_{L^2}\|f_{\neq}\|_{H^{3}}
\end{align*}

(c.) $|\eta-\xi|\gtrsim \f{|\eta|}{|l|}\approx \f{|\xi|}{|l|}$. In this case, we get 
\beno
\f{|\xi|}{l^2+(\xi-lt)^2}\lesssim \f{|\xi|}{l^2}\lesssim |\eta-\xi|,
\eeno
which gives
\begin{align*}
R_{N,1}^{NR,R}
\lesssim \|\As f_{\sim N}\|_{L^2}\|\As\Delta_{L}\Delta_t^{-1}P_{\neq}f_N\|_{L^2}\|f_{\neq}\|_{H^{4}}.
\end{align*}
Next we deal with $R_{N,2}^{NR,R}$, which we will use the fact that 
\beno
\f{|l,\xi|}{l^2+(\xi-lt)^2}\lesssim 1+\f{|\xi|}{l^2}\lesssim \langle \xi-\eta\rangle.
\eeno
Thus we get 
\begin{align*}
R_{N,2}^{NR,R}&\lesssim \sum_{k,l\neq 0,\, k\neq l}\int_{|\eta-\xi|>\f{|\eta|}{100}}\chi^{NR,R}
|\As\bar{\hat{f}}_{k}(\eta)|
\f{\As_l(\xi)|l,\xi|}{l^2+(\xi-lt)^2}
|\widehat{\Delta_{L}\Delta_t^{-1}f}_{l}(\xi)_{N}\widehat{\na f}_{k-l}(\eta-\xi)_{<N/8}|d\eta d\xi\\
&\lesssim \|\As f_{\sim N}\|_{2}\|\As \Delta_{L}\Delta_t^{-1}P_{\neq}f_N\|_{L^2}\|f_{\neq}\|_{H^{4}}.
\end{align*}

Therefore we conclude that
\beq\label{eq: R_N^NR,R}
\begin{split}
|R_{N}^{NR,R}|\lesssim &\|\As f_{\sim N}\|_{L^2}\|\As \Delta_{L}\Delta_t^{-1}f_N\|_{L^2}\| f_{\neq}\|_{H^{4}}\\
&+\left\|\sqrt{\f{\pa_t{w}}{{w}}}\As f_{\sim N}\right\|_{2}
\left\|\sqrt{\f{\pa_t{w}}{{w}}}\chi_R\As \Delta_{L}\Delta_t^{-1}f_{N}\right\|_{2}
\left\|\langle\nu^{\f13}t\rangle^{1+\b}\nu^{-\f13} f_{\neq}\right\|_{H^{4}}\\
\lesssim &\f{\ep\nu^{\f13}}{\langle \nu t^3\rangle}\|\As f_{\sim N}\|_{L^2}\|\As \Delta_{L}\Delta_t^{-1}f_N\|_{L^2}
+\ep\left\|\sqrt{\f{\pa_t{w}}{{w}}}\As f_{\sim N}\right\|_{2}
\left\|\sqrt{\f{\pa_t{w}}{{w}}}\chi_R\As \Delta_{L}\Delta_t^{-1}f_{N}\right\|_{2}.
\end{split}
\eeq

\subsubsection{Treatment of $R_{N}^{R,NR}$} In this case we have 
\beno
\f{|l,\xi|}{l^2+(\xi-lt)^2}\lesssim 1+\f{|\xi|}{l^2(1+\f{\xi^2}{l^4})}\lesssim 1,
\eeno 
which implies 
\begin{align}\label{R_N^R,NR}
R_{N}^{R,NR}
\nonumber&\lesssim \sum_{k,l\neq 0,\, k\neq l}\int_{\eta,\xi}\chi^{R,NR}
|\As\bar{\hat{f}}_{k}(\eta)|\f{\As_{l}(\xi)|l,\xi|}{l^2+(\xi-lt)^2}\left|\widehat{\Delta_L\Delta_t^{-1}f}_{l}(\xi)_{N}\widehat{\na f}_{k-l}(\eta-\xi)_{<N/8}\right|d\eta d\xi\\
\nonumber&\lesssim \|\As f_{\sim N}\|_{2}\|\As \Delta_{L}\Delta_t^{-1}P_{\neq}f_N\|_{L^2}\|f_{\neq}\|_{H^{3}}\\
&\lesssim \f{\ep\nu^{\f13}}{\langle \nu t^3\rangle}\|\As f_{\sim N}\|_{2}\|\As \Delta_{L}\Delta_t^{-1}P_{\neq}f_N\|_{L^2}.
\end{align}

\subsubsection{Treatment of $R_{N}^{R,R}$} 
In this case $t\in I_{k,\eta}\cap I_{l,\xi}$ with $k\neq l$, by Lemma \ref{lem: 3.2}, we only need to deal with the following two cases. 

(b.) $|t-\f{\eta}{k}|\gtrsim\f{\eta}{k^2}$ and $|t-\f{\xi}{l}|\gtrsim \f{\xi}{l^2}$. In this case, by the fact that $\As_l(\xi)\approx \As_k(\eta)$ and 
\beno
\f{|l,\xi|}{l^2+(\xi-lt)^2}\lesssim 1+\f{|\xi|}{l^2(1+\f{\xi^2}{l^4})}\lesssim 1,
\eeno 
we get
\begin{align*}
|R_{N}^{R,R}|&\lesssim \sum_{k,l\neq 0,\, k\neq l}\int_{\eta,\xi}\chi^{R,R}
\As|\bar{\hat{f}}_{k}(\eta)|\f{\As_{l}(\xi)|l,\xi|}{l^2+(\xi-lt)^2}
|\widehat{\Delta_L\Delta_t^{-1}f}_{l}(\xi)_{N}\widehat{\na f}_{k-l}(\eta-\xi)_{<N/8}|d\eta d\xi\\
&\lesssim \|\As f_{\sim N}\|_{L^2}\|\As\Delta_{L}\Delta_t^{-1}f_N\|_{L^2}\| f_{\neq}\|_{H^{3}}.
\end{align*}
(c.) $|\eta-\xi|\gtrsim \f{|\eta|}{|l|}\approx \f{|\xi|}{|l|}$. In this case, by using the fact that $\As_l(\xi)\approx \As_k(\eta)$ and 
\beno
\f{|l,\xi|}{l^2+(\xi-lt)^2}\lesssim 1+\f{|\xi|}{l^2}\lesssim |\eta-\xi|,
\eeno
we get
\begin{align*}
|R_{N}^{R,R}|&\lesssim \sum_{k,l\neq 0,\, k\neq l}\int_{\eta,\xi}\chi^{R,R} \As|\bar{\hat{f}}_{k}(\eta)|
\f{A_{l}(\xi)|l,\xi|}{l^2+(\xi-lt)^2}
|\widehat{\Delta_L\Delta_t^{-1}f}_{l}(\xi)_{N}\widehat{\na f}_{k-l}(\eta-\xi)_{<N/8}|d\eta d\xi\\
&\lesssim \|\As f_{\sim N}\|_{2}\|\As\Delta_{L}\Delta_t^{-1}f_N\|_{2}\|f_{\neq}\|_{H^{4}}.
\end{align*}

Therefore we conclude that
\beq\label{eq: R_N^R,R}
\begin{split}
|R_{N}^{R,R}|&\lesssim \|\As f_{\sim N}\|_{2}
\|\As\Delta_{L}\Delta_t^{-1}P_{\neq}f_N\|_{2}\|f_{\neq}\|_{H^{4}}\\
&\lesssim \f{\ep\nu^{\f13}}{\langle \nu t^3\rangle}\|\As f_{\sim N}\|_{2}\|\As\Delta_{L}\Delta_t^{-1}P_{\neq}f_N\|_{2}.
\end{split}
\eeq

Combining \eqref{eq: R_N,D}, \eqref{eq: R_N,=^NR,NR}, \eqref{eq: R_N,=^NR,R}, \eqref{eq: R_N,=^R,NR}, \eqref{eq: R_N,=^R,R}, \eqref{eq: R_N,neq,*}, \eqref{R_N^NR,NR}, \eqref{eq: R_N^NR,R}, \eqref{R_N^R,NR} and \eqref{eq: R_N^R,R} we deduce 
\beq\label{eq: R_N^1}
\begin{split}
|R_{N}^{1}|
&\lesssim \left(\f{\ep\nu^{\f13}}{\langle \nu^{\f12}t^{\f32}\rangle}
+\f{\ep\nu^{\f13}}{\langle t^2\rangle}
+\f{\ep\nu^{\f13}}{\langle \nu t^3\rangle}\right)
\|\As f_{\sim N}\|_{L^2}
\|\As\Delta_L\Delta_t^{-1}P_{\neq}f_N\|_{L^2}
\\
&\quad+\ep\left\|\sqrt{\f{\pa_tw}{w}}\As f_{\sim N}\right\|_{2}
\left\|\sqrt{\f{\pa_tw}{w}}\chi_R\As \Delta_L\Delta_t^{-1}P_{\neq}f_N\right\|_{2}.
\end{split}
\eeq

\subsection{Treatment of $R_N^2$}
We recall that 
\begin{align*}
R^2_N&=\sum_{k\neq 0}\int_{\eta,\xi}\As\bar{\hat{f}}_{k}(\eta)\As_{k}(\eta)\widehat{g}(\xi)_N\widehat{\pa_vf}_k(\eta-\xi)_{<N/8}d\eta d\xi\\
&\quad+\int_{\eta,\xi}\As\bar{\hat{f}}_{0}(\eta)\As_{0}(\eta)\widehat{g}(\xi)_N\widehat{\pa_vf}_0(\eta-\xi)_{<N/8}d\eta d\xi\\
&=R^2_{N,\neq}+R^2_{N,0}.
\end{align*}
By the fact that $|k,\eta-\xi|\leq \f{3}{16}N\leq \f{3}{8}|\xi|\approx |k,\eta|$, we have
\begin{align*}
|R^2_{N,\neq}|\lesssim \|\As(f_{\neq})_{\sim N}\|_{2}\|\langle\pa_v\rangle^{\s} g_N\|_2\|f_{\neq}\|_{H^{3}},
\end{align*}
and
\beno
|R^2_{N,0}|\lesssim \|\As(f_{0})_{\sim N}\|_{2}\|\langle\pa_v\rangle^{\s} g_N\|_2\|f_{0}\|_{H^{3}}.
\eeno

Thus we obtain that
\beq\label{eq: R_N^2}
|R^2_{N}|\lesssim \|\As f_{\sim N}\|_{2}\|g_N\|_{H^{\s}}\|f\|_{H^3}. 
\eeq

\subsection{Treatment of $R_N^3$}
$R_N^3$ is easy to dealt with, because the derivatives land on the low frequency. We then get that
\begin{align*}
|R_N^3|
&\leq \left|\sum_{k,l}\int_{\eta,\xi}\As\bar{\hat{f}}_{k}(\eta)\As_{k-l}(\eta-\xi)\hat{u}_{l}(\xi)_{N}\widehat{\na f}_{k-l}(\eta-\xi)_{<N/8}d\eta d\xi\right|\\
&\leq \sum_{k,l}\int_{\eta,\xi}\As|\bar{\hat{f}}_{k}(\eta)||l,\xi||\hat{u}_{l}(\xi)_{N}|\As_{k-l}(\eta-\xi)|\widehat{f}_{k-l}(\eta-\xi)_{<N/8}|d\eta d\xi,
\end{align*}
which gives
\beq\label{eq: R_N^3}
|R_N^3|
\lesssim \|\As f_{\sim N}\|_{2}\|u_N\|_{H^{3}}\|f\|_{H^{\s}}.
\eeq

\subsection{Corrections}\label{sec: correction}
In this section we treat $R_N^{\ep,1}$ which is higher order in $\nu^{\f13}$ than $R_N^1$. We expand $(1-v')\phi_l$ with a paraproduct only in $v$: 
\begin{align*}
R_{N}^{\ep,1}&=-\f{1}{2\pi}\sum_{M\geq 8}\sum_{k,l\neq 0}\int_{\eta,\xi,\xi'}
\As\bar{\hat{f}}_{k}(\eta)
\As_{k}(\eta)\Big((\xi-\eta)l-\xi'(k-l)\Big)\chi_N(l,\xi)\\
&\quad \quad \times\left[\widehat{(1-v')}(\xi'-\xi)\right]_{<M/8}
\widehat{\phi}_l(\xi')_M\widehat{f}_{k-l}(\eta-\xi)_{<N/8}d\eta d\xi d\xi'\\
&\quad-\f{1}{2\pi}\sum_{M\geq 8}\sum_{k,l\neq 0}\int_{\eta,\xi,\xi'}
\As\bar{\hat{f}}_{k}(\eta)
\As_{k}(\eta)\Big((\xi-\eta)l-\xi'(k-l)\Big)\chi_N(l,\xi)\\
&\quad \quad \times\left[\widehat{(1-v')}(\xi'-\xi)\right]_{M}
\widehat{\phi}_l(\xi')_{<M/8}\widehat{f}_{k-l}(\eta-\xi)_{<N/8}d\eta d\xi d\xi'\\
&\quad-\f{1}{2\pi}\sum_{M\in \mathbb{D}}\sum_{\f18M\leq M'\leq M}\sum_{k,l\neq 0}\int_{\eta,\xi,\xi'}
\As\bar{\hat{f}}_{k}(\eta)
\As_{k}(\eta)\Big((\xi-\eta)l-\xi'(k-l)\Big)\\
&\quad \quad \times\chi_N(l,\xi)\left[\widehat{(1-v')}(\xi'-\xi)\right]_{M'}
\widehat{\phi}_l(\xi')_{M}\widehat{f}_{k-l}(\eta-\xi)_{<N/8}d\eta d\xi d\xi'\\
&=R_{N,LH}^{\ep,1}+R_{N,HL}^{\ep,1}+R_{N,HH}^{\ep,1}. 
\end{align*}
We recall that $\chi_N$ denotes the Littlewood-Paley cut-off to the $N$-th dyadic shell in $\mathbb{Z}\times\R$; see Section \ref{Sec: L-P}. 

Begin first with $R_{N,LH}^{\ep,1}$. On the support of the integrand 
\beno
&&||k,\eta|-|l,\xi||\leq |k-l,\eta-\xi|\leq \f{3}{8}|l,\xi|,\\
&&||l,\xi'|-|l,\xi||\leq |\xi-\xi'|\leq \f{3}{8}|l,\xi'|.
\eeno
Thus $\As_k(\eta)\approx \As_l(\xi')$ and 
\begin{align*}
|R_{N,LH}^{\ep,1}|
&\lesssim \sum_{M\geq 8}\sum_{k,l\neq 0}\int_{\eta,\xi,\xi'}
|\As\bar{\hat{f}}_{k}(\eta)|
\Big|(\xi-\eta)l-\xi'(k-l)\Big|\chi_N(l,\xi)\\
&\quad \quad \times\Big|\left[\widehat{(1-v')}(\xi'-\xi)\right]_{<M/8}
\As_l(\xi')\widehat{\phi}_l(\xi')_M\widehat{f}_{k-l}(\eta-\xi)_{<N/8}\Big|d\eta d\xi d\xi'
\end{align*}
From here we may proceed analogous to treatment of $R_N^1$ with $(l,\xi')$ playing the role of $(l,\xi)$. We omit the details and simply conclude the result is 
\beq\label{R_N,LH^ep,1}
\begin{split}
|R_{N,LH}^{\ep,1}|
&\lesssim \|h\|_{H^{3}}\bigg[
\left(\f{\ep\nu^{\f13}}{\langle t^{\f32}\nu^{\f12}\rangle}
+\f{\ep\nu^{\f13}}{\langle t^2\rangle}
+\f{\ep\nu^{\f13}}{\langle t^{3}\nu\rangle}\right)
\|\As f_{\sim N}\|_{L^2}
\|\As\Delta_L\Delta_t^{-1}P_{\neq}f_N\|_{L^2}\\
&\quad\quad\quad\quad
+\ep\left\|\sqrt{\f{\pa_tw}{w}}\As f_{\sim N}\right\|_{2}
\left\|\sqrt{\f{\pa_tw}{w}}\As \chi_R\Delta_L\Delta_t^{-1}P_{\neq}f_N\right\|_{2}\bigg].
\end{split}
\eeq

Turn now to $R_{N,HL}^{\ep,1}$. On the support of the integrand, it holds that
\begin{align*}
\langle k,\eta\rangle^{\s}\approx \langle l,\xi\rangle^{\s}\approx \langle l,\xi'-\xi\rangle^{\s},
\end{align*}
Thus we get that
\begin{align*}
|R_{N,HL}^{\ep,1}|&\lesssim \sum_{M\geq 8}\sum_{k,l\neq 0}\int_{\eta,\xi,\xi'}\chi_{|l|\geq \f{|\xi|}{16}}
|\As\bar{\hat{f}}_{k}(\eta)|
\chi_N(l,\xi)
\bigg|\left[\widehat{(1-v')}(\xi'-\xi)\right]_{M}\\
&\quad\quad\quad
\times\langle l\rangle^{\s}|l,\xi'|\widehat{\phi}_l(\xi')_{<M/8}\widehat{\na f}_{k-l}(\eta-\xi)_{<N/8}\bigg|d\eta d\xi d\xi'\\
&+\sum_{M\geq 8}\sum_{k,l\neq 0}\int_{\eta,\xi,\xi'}
\chi_{|l|< \f{|\xi|}{16}}
|\As\bar{\hat{f}}_{k}(\eta)|
\chi_N(l,\xi)\langle \xi'-\xi\rangle^{\s}
\bigg|\left[\widehat{(1-v')}(\xi'-\xi)\right]_{M}\\
&\quad\quad\quad
\times |l,\xi'|\widehat{\phi}_l(\xi')_{<M/8}\widehat{\na f}_{k-l}(\eta-\xi)_{<N/8}\bigg|d\eta d\xi d\xi'\\
&=R_{N,HL}^{\ep,1,z}+R_{N,HL}^{\ep,1,v}. 
\end{align*}

First consider $R_{N,HL}^{\ep,1,z}$, where on the support of the integrand, $16|l|\geq |\xi|$. 
\beno
&&||k,\eta|-|l,\xi||\leq |k-l,\xi-\eta|\leq \f{3}{16}|l,\xi|,\\
&&||l,\xi|-|l,\xi'||\leq |\xi-\xi'|\leq 38|\xi|/32\lesssim |l|.
\eeno
 
Thus we divide $R_{N,HL}^{\ep,1,z}$ into two parts
\begin{align*}
R_{N,HL}^{\ep,1,z}&\lesssim
\sum_{M\geq 8}\sum_{k,l\neq 0}\int_{\eta,\xi,\xi'}\chi_{|l|\geq {16}{|\xi|}}
|\As\bar{\hat{f}}_{k}(\eta)|
\chi_N(l,\xi)
\bigg|\left[\widehat{(1-v')}(\xi'-\xi)\right]_{M}\\
&\quad\quad\quad
\times \left[\langle l\rangle^{\s+1}\widehat{\phi}_l(\xi')_{<M/8}\right]_{\sim N}\widehat{\na f}_{k-l}(\eta-\xi)_{<N/8}\bigg|d\eta d\xi d\xi'\\
&\quad+ \sum_{M\geq 8}\sum_{k,l\neq 0}\int_{\eta,\xi,\xi'}\chi_{16|\xi|\geq |l|\geq \f{|\xi|}{16}}
|\As\bar{\hat{f}}_{k}(\eta)|
\chi_N(l,\xi)
\bigg|\left[\widehat{(1-v')}(\xi'-\xi)\right]_{M\sim N}\\
&\quad\quad\quad
\times\langle l\rangle^{\s}|l,\xi'|\widehat{\phi}_l(\xi')_{<M/8}\widehat{\na f}_{k-l}(\eta-\xi)_{<N/8}\bigg|d\eta d\xi d\xi'\\
&=R_{N,HL,1}^{\ep,1,z}+R_{N,HL,2}^{\ep,1,z}
\end{align*}
To make it summable in $M$ we need more 'derivative' in higher(in $M$) frequency, luckily the all of the 'derivates' lands on the lower(in $M$) frequency. 

If $|l|\geq 16|\xi|$, then in fact $38|\xi|/32\leq |l|/4$, therefore $||k,\eta|-|l,\xi'||\leq |l|/2\leq |l,\xi'|/2$, which gives
\beno
|k,\eta|\approx |l,\xi'|\approx |l|\approx N.
\eeno
For $|l|\geq 16|\xi|\geq 16|\xi'|$, we get 
\beno
\f{|l|}{l^2(1+(\xi'/l-t)^2)}\lesssim 
\f{|l|^{-1}\langle \f{|\xi'|}{|l|}\rangle^2}{\langle \f{|\xi'|}{|l|}\rangle^2(1+(\f{\xi'}{l}-t)^2)}\lesssim \f{|l|^{-1}}{\langle t\rangle^2}
\eeno
We then get that
\begin{align*}
R_{N,HL,1}^{\ep,1,z}&\lesssim
\sum_{M\geq 8}\sum_{k,l\neq 0}\int_{\eta,\xi,\xi'}\chi_{|l|\geq {16}{|\xi|}}
|\As\bar{\hat{f}}_{k}(\eta)|
\chi_N(l,\xi)
\bigg|\left[\widehat{(1-v')}(\xi'-\xi)\right]_{M}\\
&\quad\quad\quad
\times \left[\f{\langle l\rangle^{\s+1}}{l^2+(\xi'-lt)^2}\widehat{\Delta_L\Delta_t^{-1}f}_l(\xi')_{<M/8}\right]_{\sim N}\widehat{\na f}_{k-l}(\eta-\xi)_{<N/8}\bigg|d\eta d\xi d\xi'\\
&\lesssim
\langle t\rangle^{-2}\sum_{M\geq 8}\|\As f_{\sim N}\|_2M^{-2}\|(v'-1)_{M}\|_{H^5}
\|(\Delta_L\Delta_t^{-1}P_{\neq}f)_{\sim N}\|_{H^{\s}}\|f\|_{H^3},
\end{align*}
which gives
\begin{align*}
R_{N,HL,1}^{\ep,1,z}
&\lesssim
\langle t\rangle^{-2}\|\As f_{\sim N}\|_2\|h\|_{H^5}
\|\As(\Delta_L\Delta_t^{-1}P_{\neq}f)_{\sim N}\|_2\|f\|_{H^3}\\
&\lesssim \f{\ep^2\nu^{\f23}}{\langle t\rangle^2}\|\As f_{\sim N}\|_2
\|\As(\Delta_L\Delta_t^{-1}P_{\neq}f)_{\sim N}\|_2.
\end{align*}

Next we turn to $R_{N,HL,2}^{\ep,1,z}$. 
 If $\f{1}{16}|\xi|\leq |l|\leq 16|\xi|$, then $|l|\approx |l,\xi|\approx |\xi-\xi'|\approx M\approx N$.

\begin{align*}
R_{N,HL,2}^{\ep,1,z}
&\lesssim \sum_{M\geq 8}
\|\As f_{\sim N}\|_2
M^{-1}\|(1-v')_{M\sim N}\|_{H^{\s-1}}
\|P_{\neq} \phi\|_{H^{4}}\|f\|_{H^{3}}\\
&\lesssim \|\As f_{\sim N}\|_2\|h_{\sim N}\|_{H^{\s-1}}
\|P_{\neq} \phi\|_{H^{4}}\|f\|_{H^{3}}\\
&\lesssim \f{\ep^2\nu^{\f23}}{\langle t\rangle^2} \|\As f_{\sim N}\|_2\|h_{\sim N}\|_{H^{\s-1}}
\end{align*}

Next we turn to $R_{N,HL}^{\ep,1,v}$, in which case we can consider all of the 'derivates' to be landing on $1-v'$. On the support of the integrand,
\beno
&&||k,\eta|-|l,\xi||\leq |k-l,\xi-\eta|\leq \f{3}{16}|l,\xi|,\\
&&||\xi-\xi'|-|l,\xi||\leq |l,\xi'|\leq |\xi|/16+|\xi'|\leq 67|\xi'-\xi|/100.
\eeno

Since $|l,\xi|\approx |\xi-\xi'|$, the sum only includes boundedly many terms. Therefore
\begin{align*}
R_{N,HL}^{\ep,1,v}&\lesssim \|\As f_{\sim N}\|_{2}\|\pa_vh_{\sim N}\|_{H^{\s-1}}\|\Delta_L\Delta_t^{-1}f_{\neq}\|_{H^{4}}\|f\|_{H^{3}}\\
&\lesssim \f{\ep^2\nu^{\f23}}{\langle \nu t^3\rangle}\|\As f_{\sim N}\|_{2}\|\pa_vh_{\sim N}\|_{H^{\s-1}}.
\end{align*}

We turn to the remainder term $R_{N,HH}^{\ep,1}$. 
In the case, we have 
\beno
|\xi-\xi'|\approx |\xi'|\approx M \approx M'.
\eeno
Thus we divide into two cases according to the relationship between $l$ and $\xi'$: 
\begin{align*}
|R_{N,HH}^{\ep,1}|
&\lesssim \sum_{M\in \mathbb{D}}\sum_{\f18M\leq M'\leq M}\sum_{k,l\neq 0}\int_{\eta,\xi,\xi'} 1_{|l|\geq 3|\xi'|}
\As\bar{\hat{f}}_{k}(\eta)
\chi_N(l,\xi)\left[\widehat{(1-v')}(\xi'-\xi)\right]_{M'}\\
&\quad \quad \times
\f{|l|\As_{l}(\xi)}{l^2+(\xi'-lt)^2}|\widehat{\Delta_L\Delta_t^{-1}f}_l(\xi')_{M}||\widehat{\na f}_{k-l}(\eta-\xi)_{<N/8}|d\eta d\xi d\xi'\\
&+\sum_{M\in \mathbb{D}}\sum_{\f18M\leq M'\leq M}\sum_{k,l\neq 0}\int_{\eta,\xi,\xi'} 1_{|l|<3|\xi'|}
\As\bar{\hat{f}}_{k}(\eta)
\chi_N(l,\xi)\left[\widehat{(1-v')}(\xi'-\xi)\right]_{M'}\\
&\quad \quad \times
\f{|l|\As_{l}(\xi)}{l^2+(\xi'-lt)^2}
|\widehat{\Delta_L\Delta_t^{-1}f}_l(\xi')_{M}\widehat{\na f}_{k-l}(\eta-\xi)_{<N/8}|d\eta d\xi d\xi'\\
&=R_{N,HH}^{\ep,1,z}+R_{N,HH}^{\ep,1,v}.
\end{align*}
First we consider $R_{N,HH}^{\ep,1,z}$. In this case, we have $\langle \xi'-\xi\rangle\approx \langle \xi'\rangle$, 
\beno
\As_k(\eta)\approx \As_l(\xi)\approx 
\langle l,\xi\rangle^{\s}\approx 
\langle l\rangle^{\s}+\langle \xi\rangle^{\s}
\lesssim \langle l\rangle^{\s}+\langle \xi'\rangle^{\s}+\langle \xi'-\xi\rangle^{\s}
\lesssim \langle l\rangle^{\s},
\eeno
and $N\approx |k,\eta|\approx \langle l\rangle$, then 
\beno
\f{|l|\As_{l}(\xi)}{l^2+(\xi'-lt)^2}\lesssim \f{\langle l\rangle^{\s-1}}{\langle t\rangle^{2}}. 
\eeno
Therefore, we get
\begin{align*}
R_{N,HH}^{\ep,1,z}&\lesssim \langle t\rangle^{-2}\sum_{M\in \mathbb{D}} \|\As f_{\sim N}\|_2M^{-2}\|h_{\sim M}\|_{H^4}\|\Delta_L\Delta_t^{-1}f_{\sim N}\|_{H^{\s}}\|f\|_{H^3}\\
&\lesssim \langle t\rangle^{-2}\|\As f_{\sim N}\|_2\|h\|_{H^4}\|\Delta_L\Delta_t^{-1}P_{\neq}f_{\sim N}\|_{H^{\s}}\|f\|_{H^3}\\
&\lesssim \f{\ep^2\nu^{\f23}}{\langle t\rangle^2}\|\As f_{\sim N}\|_2\|\Delta_L\Delta_t^{-1}P_{\neq}f_{\sim N}\|_{H^{\s}}.
\end{align*}

For $|l|<3|\xi'|$, we have $\langle \xi'-\xi\rangle\approx \langle \xi'\rangle$, 
\beno
\As_k(\eta)\approx \As_l(\xi)\approx 
\langle l,\xi\rangle^{\s}\approx 
\langle l\rangle^{\s}+\langle \xi\rangle^{\s}
\lesssim \langle l\rangle^{\s}+\langle \xi'\rangle^{\s}+\langle \xi'-\xi\rangle^{\s}
\lesssim \langle \xi'\rangle^{\s},
\eeno
and $N\approx \langle l,\xi\rangle\approx \langle \xi'-\xi\rangle$. Thus 
\begin{align*}
R_{N,HH}^{\ep,1,v}&\lesssim 
\sum_{M\in \mathbb{D}}\sum_{\f18M\leq M'\leq M}\sum_{k,l\neq 0}\int_{\eta,\xi,\xi'} 1_{|l|<3|\xi'|}
\As\bar{\hat{f}}_{k}(\eta)
\chi_N(l,\xi)\left[\langle \xi'-\xi\rangle^{\s-1}\widehat{(1-v')}(\xi'-\xi)\right]_{M'}\\
&\quad \quad \times
\langle \xi'\rangle^2
|\widehat{\phi}_l(\xi')_{M}\widehat{\na f}_{k-l}(\eta-\xi)_{<N/8}|d\eta d\xi d\xi'\\
&\lesssim \sum_{M\in \mathbb{D}} \|\As f_{\sim N}\|_2M^{-2}\|h_{\sim N}\|_{H^{\s-1}}\|(\Delta_L\Delta_t^{-1}f_{\neq})_{\sim M}\|_{H^{4}}\|f\|_{H^3}\\
&\lesssim \|\As f_{\sim N}\|_2\|h_{\sim N}\|_{H^{\s-1}}\|P_{\neq}\phi\|_{H^{4}}\|f\|_{H^3}\\
&\lesssim \f{\ep^2\nu^{\f23}}{\langle t\rangle^2}\|\As f_{\sim N}\|_2\|h_{\sim N}\|_{H^{\s-1}}.
\end{align*}

Therefore, we conclude by the bootstrap hypotheses and Lemma \ref{lem: lin-inv-dam}, that
\beq\label{eq: R_N^1,ep}
\begin{split}
|R_N^{1,\ep}|&\lesssim \ep\nu^{\f13}\bigg[
\left(\f{\ep\nu^{\f13}}{\langle t^{\f32}\nu^{\f12}\rangle}
+\f{\ep\nu^{\f13}}{\langle t^2\rangle}
+\f{\ep\nu^{\f13}}{\langle t^{3}\nu\rangle}\right)
\|\As f_{\sim N}\|_{L^2}
\|\As\Delta_L\Delta_t^{-1}P_{\neq}f_N\|_{L^2}\\
&\quad\quad\quad\quad
+\ep\left\|\sqrt{\f{\pa_tw}{w}}\As f_{\sim N}\right\|_{2}
\left\|\sqrt{\f{\pa_tw}{w}}\As \chi_R\Delta_L\Delta_t^{-1}P_{\neq}f_N\right\|_{2}\bigg]\\
&\quad + \f{\ep^2\nu^{\f23}}{\langle t\rangle^2} \|\As f_{\sim N}\|_{2}\bigg[\|\pa_vh_{\sim N}\|_{H^{\s-1}}
+
\|\As(\Delta_L\Delta_t^{-1}P_{\neq}f_{\neq})_{\sim N}\|_2
+\|h_{\sim N}\|_{H^{\s-1}}\bigg].
\end{split}
\eeq

We end this section by proving Proposition \ref{prop: reaction}
\begin{proof}
By \eqref{eq: R_N^1}, \eqref{eq: R_N^2}, \eqref{eq: R_N^3}, \eqref{eq: R_N^1,ep}, Lemma \ref{lem: lin-inv-dam}, Proposition \ref{Rmk: basic estimate}, \eqref{eq: L-P-ortho2} and the bootstrap hypotheses, we get that
\begin{align*}
\sum_{N\geq 8}|R_N|
&\lesssim 
\left(\f{\ep\nu^{\f13}}{\langle \nu^{\f12} t^{\f32}\rangle}
+\f{\ep\nu^{\f13}}{\langle t^2\rangle}
+\f{\ep\nu^{\f13}}{\langle \nu t^3\rangle}\right)\|f\|_{H^{\s}}\|\As(\Delta_L\Delta_t^{-1}P_{\neq}f)\|_2\\
&\quad+\ep \left\|\sqrt{\f{\pa_t{w}}{{w}}}\As f\right\|_{2}
\left\|\sqrt{\f{\pa_tw}{w}}\As \chi_R\Delta_L\Delta_t^{-1}P_{\neq}f_N\right\|_{2}\\
&\quad+\ep\nu^{\f13}\|\As f\|_{2}\|g\|_{H^{\s}}
+ \f{\ep^2\nu^{\f23}}{\langle t\rangle^2} \|\As f\|_{2}\bigg[\|\pa_vh\|_{H^{\s-1}}
+\|\As(\Delta_L\Delta_t^{-1}P_{\neq}f)\|_2
+\|h\|_{H^{\s-1}}\bigg].
\end{align*}
By Lemma \ref{lem: elliptical highest}, we have
\begin{align*}
\sum_{N\geq 8}|R_N|
&\lesssim 
\left(\f{\ep\nu^{\f13}}{\langle \nu^{\f12} t^{\f32}\rangle}
+\f{\ep\nu^{\f13}}{\langle t^2\rangle}
+\f{\ep\nu^{\f13}}{\langle \nu t^3\rangle}\right)\|f\|_{H^{\s}}\Big(\|f_{\neq}\|_{H^{\s}}+\f{\ep\nu^{\f13}}{\langle t\rangle \langle \nu t^3\rangle}\|\pa_vh\|_{H^{\s}}\Big)\\
&\quad+\ep \left\|\sqrt{\f{\pa_t{w}}{{w}}}\As f\right\|_{2}
\Big(\left\|\sqrt{\f{\pa_tw}{w}}f_{\neq}\right\|_{H^{\s}}+\f{\ep^2\nu^{\f12}}{\langle \nu t^3\rangle}\Big)\\
&\quad+\ep\nu^{\f13}\|\As f\|_{2}\|g\|_{H^{\s}}
+ \f{\ep^2\nu^{\f23}}{\langle t\rangle^2} \|\As f\|_{2}\big[\|\pa_vh\|_{H^{\s-1}}
+\|h\|_{H^{\s-1}}\big],
\end{align*}
and then by the bootstrap hypotheses and the Young's inequality, it holds that
\begin{align*}
\int_1^t\sum_{N\geq 8}|R_N(t')|dt'
&\lesssim \ep\sup_{t'\in [1,t]}\|\As f(t')\|_2^2+\ep\int_1^tCK_w(t')dt'\\
&\quad+\ep\nu^{\f13}\sup_{t'\in [1,t]}\|\As f(t')\|_2
\|\pa_vh\|_{L^2_T(H^{\s})}
\left(\int_1^t\Big(\f{\ep\nu^{\f13}}{\langle t'\rangle \langle \nu t'^3\rangle}\Big)^2dt'\right)^{\f12}\\
&\quad+\ep\left(\int_1^tCK_w(t')dt'\right)^{\f12}
\left(\int_1^t\Big(\f{\ep^2\nu^{\f12}}{\langle \nu t'^3\rangle}\Big)^2dt'\right)^{\f12}\\
&\quad +\ep\nu^{\f13}\sup_{t'\in [1,t]}\|\As f(t')\|_2\|g\|_{L^1_T(H^{\s})}
+\ep^3\nu\sup_{t'\in [1,t]}\|\As f(t')\|_2\\
&\quad+\ep^2\nu^{\f23}\sup_{t'\in [1,t]}\|\As f(t')\|_2
\|\pa_vh\|_{L^2_T(H^{\s-1})}\\
&\lesssim \ep\sup_{t'\in [1,t]}\|\As f(t')\|_2^2+\ep\int_1^tCK_w(t')dt'\\
&\quad+\ep^3\nu^{\f13}\sup_{t'\in [1,t]}\|\As f(t')\|_2
+\ep^3\nu^{\f13}\left(\int_1^tCK_w(t')dt'\right)^{\f12}\\
&\lesssim \ep\sup_{t'\in [1,t]}\|\As f(t')\|_2^2+\ep\int_1^tCK_w(t')dt'+\ep^{5}\nu^{\f23}
\end{align*}
Thus we have proved Proposition \ref{prop: reaction}.
\end{proof}

\section{Coordinate system}
\subsection{Higher regular controls}\label{Sec: higher regular controls}
In this subsection we will study the energy estimate for $g$ in $H^{\s}$ and $h,\bar{h}$ in $H^{\s-1}$ and $H^{\s}$. 
\subsubsection{Energy estimate of $g$}\label{sec: g-high}
In this section, we will prove \eqref{eq: aim2}. We need to mention that the result of \eqref{eq: aim2} is not optimal, however it is enough. It is natural to computer the time evolution of $\left\|g\right\|_{H^{\s}}^2$. 
We get 
\begin{align*}
\f{d}{dt}\left\|\langle \pa_v\rangle^{\s}g\right\|_2^2
&=2\int \langle \pa_v\rangle^{\s}g \langle \pa_v\rangle^{\s}\pa_tgdv\\
&=-\f{4}{t}\|\langle \pa_v\rangle^{\s}g\|_2^2
-2\int \langle \pa_v\rangle^{\s}g \langle \pa_v\rangle^{\s}(g\pa_vg)dv\\
&\quad-\f{2}{t}\int \langle \pa_v\rangle^{\s}g \langle \pa_v\rangle^{\s} (v'\langle \na_{z,v}^{\bot}P_{\neq}\phi\cdot\na_{z,v}\tu\rangle)dv\\
&\quad+2\nu \int \langle \pa_v\rangle^{\s}g \langle \pa_v\rangle^{\s}((v')^2\pa_v^2g)dv\\
&=-\f{4}{t}\|\langle \pa_v\rangle^{\s}g\|_2^2
+V^{H,g}_1+V^{H,g}_2+V^{H,g}_3
\end{align*}

To treat $V^{H,g}_1$, we get by using integration by parts,
\begin{align*}
|V^{H,g}_1|&\lesssim \left|\int |\pa_vg| |\langle \pa_v\rangle^{\s}g|^2dv\right|
+\|g\|_{H^{\s}}\|[\langle \pa_v\rangle^{\s},g]\pa_vg\|_{2}.
\end{align*}
By Lemma \ref{lem: commutator} and the Sobolev embedding theory, we get
\begin{align*}
|V^{H,g}_1|
\lesssim \|g\|_{H^{2}}\|\langle \pa_v\rangle^{\s}g\|_{2}^2.
\end{align*}

Next we treat $V^{H,g}_2$. We now use the fact that 
\begin{align*}
<\na_{z,v}^{\bot}P_{\neq}\phi\cdot\na_{z,v}\tu>
= <\na_{L}^{\bot}P_{\neq}\phi\cdot \na_{L}\tu_{\neq}>.
\end{align*}
Then by the bootstrap hypotheses, we get
\begin{align*}
|V^{H,g}_2|&\lesssim 
\f{2}{t}\int \langle \pa_v\rangle^{\s}g \langle \pa_v\rangle^{\s} 
\left[(v'-1)< \na^{\bot}P_{\neq}\phi\cdot \na\tu_{\neq}>\right] dv\\
&\quad+\f{2}{t}\int \langle \pa_v\rangle^{\s}g \langle \pa_v\rangle^{\s} \left[<\na_{L}^{\bot}P_{\neq}\phi\cdot \na_{L}\tu_{\neq}>\right] dv\\
&\lesssim 
\f{1}{t}\|\langle\pa_v\rangle^{\s} g\|_{2}\|\langle \pa_v\rangle^{\s}(v'-1)\|_2
\|< \na^{\bot}P_{\neq}\phi\cdot \na\tu_{\neq}>\|_{H^{\s}}\\
&\quad+\f{1}{t}\|\langle\pa_v\rangle^{\s} g\|_{2}
\|\langle\pa_v\rangle^{\s}< \na_L^{\bot}P_{\neq}\phi\cdot \na_L\tu_{\neq}>\|_{2}\\
&\lesssim 
\f{1}{t}\|\langle\pa_v\rangle^{\s} g\|_{2}
\|h\|_{H^{\s}}
\|\na_{L}^{\bot}P_{\neq}\phi\|_{H^{\s}}
\|\na_{L}\tu_{\neq}\|_{H^{2}}\\
&\quad+\f{1}{t}\|\langle\pa_v\rangle^{\s} g\|_{2}
\|h\|_{H^{\s}}
\|\na_L^{\bot}P_{\neq}\phi\|_{H^{2}}
\|\na_L\tu_{\neq}\|_{H^{\s}}\\
&\quad+\f{1}{t}\|\langle\pa_v\rangle^{\s} g\|_{2}
\| \na_L^{\bot}P_{\neq}\phi\|_{H^{\s}}
\| \na_L\tu_{\neq}\|_{H^{2}}
+\f{1}{t}\|\langle\pa_v\rangle^{\s} g\|_{2}
\|\na_L^{\bot}P_{\neq}\phi\|_{H^{2}}
\|\na_L\tu_{\neq}\|_{H^{\s}}\\
&\lesssim 
\f{1}{t}\|\langle\pa_v\rangle^{\s} g\|_{2}
\| \na_L^{\bot}P_{\neq}\phi\|_{H^{\s}}
\| \na_L\tu_{\neq}\|_{H^{2}}
+\f{1}{t}\|\langle\pa_v\rangle^{\s} g\|_{2}
\|\na_L^{\bot}P_{\neq}\phi\|_{H^{2}}
\|\na_L\tu_{\neq}\|_{H^{\s}}.
\end{align*}
At last we treat the dissipation term $V_3^{H,g}$. We have  
\begin{align*}
V_3^{H,g}
&=2\nu \int \langle \pa_v\rangle^{\s}g \langle \pa_v\rangle^{\s}(\pa_v^2g)dv\\
&\quad+2\nu \int \langle \pa_v\rangle^{\s}g \langle \pa_v\rangle^{\s}\big(((v')^2-1)\pa_v^2g\big)dv\\
&=-2\nu\|\pa_v\langle \pa_v\rangle^{\s}g \|_2^2\\
&\quad+2\nu \int \langle \pa_v\rangle^{\s}g \langle \pa_v\rangle^{\s}\big(((v')^2-1)\pa_v^2g\big)dv\\
&=-2\nu\|\pa_v\langle \pa_v\rangle^{\s}g \|_2^2+V_{3,\ep}^{H,g}
\end{align*}

The term $V_{3,\ep}^{H,g}$ is similar to $E^0$. 
We then obtain by Young's inequality that
\beq\label{V_3,ep^H,g}
\begin{split}
|V_{3,\ep}^{H,g}|
&\lesssim \nu (1+\|h\|_{H^2})\left(\|\pa_vg\|_{H^{\s}}^2\|h\|_{H^{2}}+\|\pa_vh\|_{H^{\s-2}}\|\pa_vg\|_{H^{\s}}\|g\|_{H^4}\right)\\
&\lesssim \nu (1+\|h\|_{H^2})\left(\|\pa_vg\|_{H^{\s}}^2\|h\|_{H^{2}}+\|h\|_{H^{\s-1}}\|\pa_vg\|_{H^{\s}}^2+\|\pa_vh\|_{H^{\s-2}}\|g\|_{H^4}^2\right).
\end{split}
\eeq

By the bootstrap assumption, we get 
\begin{align*}
\f{d}{dt}\|g\|_{H^{\s}}^2
&\leq -\f{4}{t}\|g\|_{H^{\s}}^2-\nu\|\pa_vg \|_{H^{\s}}^2+C\bigg(\|g\|_{H^2}\|g\|_{H^{\s}}^2\\
&\quad+\f{1}{t}\|g\|_{H^{\s}}\Big(
\| \na_L^{\bot}P_{\neq}\phi\|_{H^{\s}}
\| \na_L\tu_{\neq}\|_{H^{2}}
+\|\na_L^{\bot}P_{\neq}\phi\|_{H^{2}}
\|\na_L\tu_{\neq}\|_{H^{\s}}\Big)\\
&\quad
+\nu \|\pa_vh\|_{H^{\s-2}}\|g\|_{H^4}^2\bigg)\\
&\leq -\f{4}{t}\|g\|_{H^{\s}}^2+C\bigg[
\|g\|_{H^2}\|g\|_{H^{\s}}^2
+\nu \|\pa_vh\|_{H^{\s-2}}\|g\|_{H^{\s}}\|g\|_{H^{4}}\\
&\quad+\f{1}{t}\|g\|_{H^{\s}}\Big(
\| \Delta_{L}P_{\neq}\phi\|_{H^{\s}}
\| \na_L\tu_{\neq}\|_{H^{2}}
+\|\na_L^{\bot}P_{\neq}\phi\|_{H^{2}}
\|\na_L\tu_{\neq}\|_{H^{\s}}\Big)\bigg],
\end{align*}
which gives that for $t\geq 1$,
\begin{align*}
&\sup_{t'\in [1,t]}\Big(t'\|g(t')\|_{H^{\s}}\Big)+\int_1^{t}\|g(t')\|_{H^{\s}}dt'\\
&\leq \|g(1)\|_{H^{\s}}+C\bigg(
\|g\|_{L^1_TH^2}\sup_{t'\in [1,T]}t'\|g(t')\|_{H^{\s}}
+\nu^{\f12}\|\pa_vh\|_{L^2_T(H^{\s-2})}\nu^{\f12}\|t g\|_{L^2_{T}(H^4)}\\
&\quad+ \|f_{\neq}\|_{L^{\infty}_{T}(H^{\s})}\| \na_L\tu_{\neq}\|_{L^1_{T}H^{2}}
+\left\|\f{\ep\nu^{\f13}}{\langle t\rangle\langle \nu t^3\rangle}\|\pa_vh\|_{H^{\s}}\right\|_{L^1_{T}}\| \na_L\tu_{\neq}\|_{L^{\infty}_{T}(H^{2})}\\
&\quad +\|\na_L^{\bot}P_{\neq}\phi\|_{L^{1}_{T}(H^{2})}\| f_{\neq}\|_{L^{\infty}_{T}H^{\s}}
 + \|\na_L^{\bot}P_{\neq}\phi\|_{L^{\infty}_{T}(H^{2})} \left\|\f{\ep\nu^{\f13}}{\langle t\rangle\langle \nu t^3\rangle}\|\pa_vh\|_{H^{\s}}\right\|_{L^1_{T}}\bigg)\\
&\leq \|g(1)\|_{H^{\s}}+C\bigg(
\|g\|_{L^1_TH^2}\sup_{t'\in [1,T]}t'\|g(t')\|_{H^{\s}}
+\nu^{\f12}\|\pa_vh\|_{L^2_T(H^{\s-2})}\nu^{\f12}\|t g\|_{L^2_{T}(H^4)}\\
&\quad+ \ep\nu^{\f13}\| \na_L\tu_{\neq}\|_{L^1_{T}H^{2}}
+\ep\nu^{\f13}\|\na_L^{\bot}P_{\neq}\phi\|_{L^{1}_{T}(H^{2})}
+\ep^2\nu^{\f23}\|\pa_vh\|_{L^2_{T}(H^{\s})}\bigg).
\end{align*}

Thus by the bootstrap hypotheses, we get that
\begin{align*}
\sup_{t'\in [1,t]}\Big(t'\|g(t')\|_{H^{\s}}\Big)
&+\int_1^{t}\|g(t')\|_{H^{\s}}dt'\\
&\leq \|g(1)\|_{H^{\s}}+C\bigg(
\ep\nu^{\f13}\sup_{t'\in [1,T]}t'\|g(t')\|_{H^{\s}}+\ep^2\nu^{\f13}\bigg).
\end{align*}
By taking $\ep$ small enough, we proved \eqref{eq: aim2}. 

\subsubsection{Energy estimate of $\bar{h}$ and $h$}\label{sec: barh and h-high}
We get that
\begin{align*}
\f12\f{d}{dt}\|\bar{h}\|_{H^{\s-1}}^2
&=-\int \langle \pa_v\rangle^{\s-1}\bar{h}\langle \pa_v\rangle^{\s-1}(g\pa_v \bar{h})dv
-\f{2}{t}\|\bar{h}\|_{H^{\s-1}}^2\\
&\quad+\f{1}{t}\int \langle \pa_v\rangle^{\s-1}\bar{h}\langle \pa_v\rangle^{\s-1}\Big(v'< \na_{z,v}^{\bot}P_{\neq}\phi\cdot\na_{z,v}f_{\neq}>\Big)dv\\
&\quad+\nu\int \langle \pa_v\rangle^{\s-1}\bar{h}\langle \pa_v\rangle^{\s-1}(((v')^2-1)\pa_{vv}\bar{h})dv
-\nu\|\pa_v \bar{h}\|_{H^{\s-1}}\\
&=-\f{2}{t}\|\bar{h}\|_{H^{\s-1}}^2-\nu\|\pa_v \bar{h}\|_{H^{\s-1}}+V^{H,\bar{h}}_{1}+V^{H,\bar{h}}_{2}+V^{H,\bar{h}}_{3},
\end{align*}
and
\begin{align*}
\f12\f{d}{dt}\|{h}\|_{H^{\s-1}}^2
&=-\int \langle \pa_v\rangle^{\s-1}h\langle \pa_v\rangle^{\s-1}\Big(g\pa_v h-\bar{h}-\nu (v')^2\pa_v^2h\big)dv\\
&=-\int \langle \pa_v\rangle^{\s-1}h\langle \pa_v\rangle^{\s-1}(g\pa_v h)dv\\
&\quad+\int \langle \pa_v\rangle^{\s-1}h \langle \pa_v\rangle^{\s-1}\bar{h}dv\\
&\quad +\nu \int \langle \pa_v\rangle^{\s-1}h \langle \pa_v\rangle^{\s-1}\big(((v')^2-1)\pa_v^2h\big)dv
-\nu\|\pa_v h\|_{H^{\s-1}}^2\\
&=-\nu\|\pa_v h\|_{H^{\s-1}}^2+V^{H,{h}}_{1}+V^{H,{h}}_{2}+V^{H,{h}}_{3}
\end{align*}

We use the same argument as in the treatment of $V^{H,g}_1$ and get that
\beno
|V^{H,\bar{h}}_{1}|\lesssim \|g\|_{H^{\s-1}}\|\bar{h}\|_{H^{\s-1}}^2,
\eeno
and
\beno
|V^{H,{h}}_{1}|\lesssim \|g\|_{H^{\s-1}}\|\bar{h}\|_{H^{\s-1}}^2. 
\eeno

We also have 
\beno
|V^{H,{h}}_{2}|\leq \|{h}\|_{H^{\s-1}}\|\bar{h}\|_{H^{\s-1}}.
\eeno

To treat $V^{H,\bar{h}}_{2}$, we have
\begin{align*}
V^{H,\bar{h}}_{2}&=\f{1}{t}\int \langle \pa_v\rangle^{\s-1}\bar{h}\langle \pa_v\rangle^{\s-1}\Big(< \na_{z,v}^{\bot}P_{\neq}\phi\cdot\na_{z,v}f_{\neq}>\Big)dv\\
&\quad+\f{1}{t}\int \langle \pa_v\rangle^{\s-1}\bar{h}\langle \pa_v\rangle^{\s-1}\Big(h< \na_{z,v}^{\bot}P_{\neq}\phi\cdot\na_{z,v}f_{\neq}>\Big)dv.
\end{align*}
Thus by the fact that
\beno
<-\pa_vF\pa_zG+\pa_zF\pa_vG>=\pa_v<G\pa_zF>,
\eeno
we get 
\begin{align*}
|V^{H,\bar{h}}_{2}|
&\lesssim \f{1}{t}\|\bar{h}\|_{H^{\s-1}}(1+\|h\|_{H^2})\|< \na_{z,v}^{\bot}P_{\neq}\phi\cdot\na_{z,v}f_{\neq}>\|_{H^{\s-1}}\\
&\quad+\f{1}{t}\|\bar{h}\|_{H^{\s-1}}\|h\|_{H^{\s-1}}
\|< \na_{z,v}^{\bot}P_{\neq}\phi\cdot\na_{z,v}f_{\neq}>\|_{H^{1}}\\
&\lesssim \f{1}{t}\|\bar{h}\|_{H^{\s-1}}(1+\|h\|_{H^2})\|<\pa_{z}P_{\neq}\phi f_{\neq}>\|_{H^{\s}}\\
&\quad+\f{1}{t}\|\bar{h}\|_{H^{\s-1}}\|h\|_{H^{\s-1}}
\|< \pa_{z}P_{\neq}\phi f_{\neq}>\|_{H^{2}}\\
&\lesssim \f{1}{t}\|\bar{h}\|_{H^{\s-1}}(1+\|h\|_{H^2})
\|P_{\neq}\phi\|_{H^{\s}}\|\pa_{z}f_{\neq}\|_{H^1}\\
&\quad+\f{1}{t}\|\bar{h}\|_{H^{\s-1}}(1+\|h\|_{H^2})
\|\pa_{z}P_{\neq}\phi\|_{H^{1}}\|f_{\neq}\|_{H^{\s}}\\
&\quad+\f{1}{t}\|\bar{h}\|_{H^{\s-1}}\|h\|_{H^{\s-1}}
\|\pa_{z}P_{\neq}\phi\|_{H^2}\|f_{\neq}\|_{H^{2}}. 
\end{align*}

Next we turn to $V_3^{H,\bar{h}}$ and $V_3^{H,{h}}$
\begin{align*}
|V_{3,\ep}^{H,\bar{h}}|
&\lesssim \int_{\xi,\eta}\langle \eta\rangle^{2\s-2}
|\hat{\bar{h}}(\eta)|\left|(\widehat{1-(v')^2}(\eta-\xi))|\xi|^2\hat{\bar{h}}(\xi)\right|d\xi d\eta\\
&\lesssim \int_{\xi,\eta}1_{|\eta|\leq 1}
|\hat{\bar{h}}(\eta)|\left|(\widehat{1-(v')^2}(\eta-\xi))|\xi|(|\eta|+|\xi-\eta|)\hat{\bar{h}}(\xi)\right| d\eta d\xi\\
&\quad+\int_{\xi,\eta}1_{|\eta|\geq 1}1_{|\xi-\eta|\geq |\xi|}|\eta|^{2\s-2}
|\hat{\bar{h}}(\eta)|\left|(\widehat{1-(v')^2}(\eta-\xi))|\xi|^2\hat{\bar{h}}(\xi)\right|d\xi d\eta\\
&\quad+\int_{\xi,\eta}1_{|\eta|\geq 1}1_{|\xi-\eta|<|\xi|}|\eta|^{2\s-2}
|\hat{\bar{h}}(\eta)|\left|(\widehat{1-(v')^2}(\eta-\xi))|\xi|^2\hat{\bar{h}}(\xi)\right|d\xi d\eta\\
&\lesssim \int_{\xi,\eta}1_{|\eta|\leq 1}
|\eta||\hat{\bar{h}}(\eta)|\left|(\widehat{(1-(v')^2)}(\eta-\xi))|\xi|\hat{f}_0(\xi)\right| d\eta d\xi\\
&\quad+\int_{\xi,\eta}1_{|\eta|\leq 1}
|\hat{\bar{h}}(\eta)|\left|(\widehat{\pa_v(1-(v')^2)}(\eta-\xi))|\xi|\hat{\bar{h}}(\xi)\right| d\eta d\xi\\
&\quad+\int_{\xi,\eta}1_{|\eta|\geq 1}1_{|\xi-\eta|< |\xi|}
|\eta|^{\s}
|\hat{\bar{h}}(\eta)|\left|(\widehat{1-(v')^2}(\eta-\xi))|\xi|^{\s}\hat{\bar{h}}(\xi)\right|d\xi d\eta\\
&\quad+\int_{\xi,\eta}1_{|\eta|\geq 1}1_{|\xi-\eta|\geq |\xi|}|\eta|^{\s-1}
|\hat{\bar{h}}(\eta)|\left|(\widehat{1-(v')^2}(\eta-\xi))|\eta-\xi|^{\s-1}|\xi|^2\hat{\bar{h}}(\xi)\right|d\xi d\eta\\
&\lesssim \|(1-(v')^2)\|_{H^2}\|\pa_v\bar{h}\|_{H^{\s-1}}
+\|\bar{h}\|_{H^{\s-1}}\|(1-(v')^2)\|_{H^{\s-1}}\|\bar{h}\|_{H^4}\\
&\lesssim (1+\|h\|_{H^2})\Big(\|h\|_{H^2}\|\pa_v\bar{h}\|_{H^{\s-1}}
+\|\bar{h}\|_{H^{\s-1}}\|\pa_vh\|_{H^{\s-2}}\|\bar{h}\|_{H^4}\Big),
\end{align*}
and
\begin{align*}
|V_{3,\ep}^{H,{h}}|
&\lesssim \nu (1+\|h\|_{H^2})\|\pa_v{h}\|_{H^{\s-1}}^2\|h\|_{H^{4}}.
\end{align*}

Putting together and using the bootstrap assumption, we get
\begin{align*}
\f12\f{d}{dt}\|\bar{h}\|_{H^{\s-1}}^2
&\leq -\f{2}{t}\|\bar{h}\|_{H^{\s-1}}^2+C\bigg(
\|g\|_{H^2}\|\bar{h}\|_{H^{\s-1}}^2\\
&\quad+\f{1}{t}\|\bar{h}\|_{H^{\s-1}}(1+\|h\|_{H^2})
\|P_{\neq}\phi\|_{H^{\s}}\|\pa_{z}f_{\neq}\|_{H^1}\\
&\quad+\f{1}{t}\|\bar{h}\|_{H^{\s-1}}(1+\|h\|_{H^2})
\|\pa_{z}P_{\neq}\phi\|_{H^{1}}\|f_{\neq}\|_{H^{\s}}\\
&\quad+\f{1}{t}\|\bar{h}\|_{H^{\s-1}}\|h\|_{H^{\s-1}}
\|\pa_{z}P_{\neq}\phi\|_{H^2}\|f_{\neq}\|_{H^{2}}
+\nu\|\pa_vh\|_{H^{\s-2}}\|\bar{h}\|_{H^{\s-1}}\|\bar{h}\|_{H^4}
\bigg).
\end{align*}

We also get 
\begin{align*}
\f{d}{dt}\big(t\|\bar{h}\|_{H^{\s-1}}\big)
&\leq -\|\bar{h}\|_{H^{\s-1}}+C\bigg(
t\|g\|_{H^2}\|\bar{h}\|_{H^{\s-1}}\\
&\quad+\|P_{\neq}\phi\|_{H^{\s}}\|\pa_{z}f_{\neq}\|_{H^1}
+\|\pa_{z}P_{\neq}\phi\|_{H^{2}}\|f_{\neq}\|_{H^{\s}}
+\nu t\|\pa_vh\|_{H^{\s-2}}\|\bar{h}\|_{H^4}
\bigg),
\end{align*}
which then implies
\begin{align*}
&\sup_{t\in [1,T]}\big(t\|\bar{h}(t)\|_{H^{\s-1}}\big)
+\int_{0}^{T}\|\bar{h}(t')\|_{H^{\s-1}}dt'\\
&\leq \|\bar{h}(1)\|_{H^{\s-1}}
+C\bigg(\|g\|_{L^1_{T}(H^2)}\sup_{t\in [1,T]}\big(t\|\bar{h}(t)\|_{H^{\s-1}}\big)
+\|P_{\neq}\phi\|_{L^{\infty}_{T}H^{\s}}\|\pa_{z}f_{\neq}\|_{L^1_{T}H^1}\\
&\quad\quad\quad\quad\quad\quad\quad\quad\quad
+\|\pa_{z}P_{\neq}\phi\|_{L^1_{T}H^{2}}\|f_{\neq}\|_{L^{\infty}(H^{\s})}+\nu\|\pa_vh\|_{L^2_TH^{\s-2}}
\left\|\langle t\rangle\|\bar{h}\|_{H^4}\right\|_{L^2_T}\bigg)\\
&\leq \|\bar{h}(1)\|_{H^{\s-1}}+C\bigg(\ep\nu^{\f13}\sup_{t\in [1,T]}\big(t\|\bar{h}(t)\|_{H^{\s-1}}\big)
+\ep^2\nu^{\f13}+\ep^2\nu^{\f76}\bigg). 
\end{align*}
Then by taking $\ep$ small enough, we get \eqref{eq: aim3}. 

We also get by Young's inequality and the bootstrap hypotheses that
\begin{align*}
&\sup_{t\in [1,T]}\big(\|{h}(t)\|_{H^{\s-1}}^2\big)
+\nu\|\pa_v{h}\|_{L^2_TH^{\s-1}}^2\\
&\leq \|{h}(1)\|_{H^{\s-1}}^2+\|\bar{h}\|_{L^1_T(H^{\s-1})}\sup_{t\in [1,T]}\big(\|{h}(t)\|_{H^{\s-1}}\big)
+C\|g\|_{L^1_{T}(H^2)}\sup_{t\in [1,T]}\big(\|{h}(t)\|_{H^{\s-1}}^2\big)\\
&\leq \|{h}(1)\|_{H^{\s-1}}^2+4\|\bar{h}\|_{L^1_T(H^{\s-1})}^2
+\Big(\f14+C\|g\|_{L^1_{T}(H^2)}\Big)\sup_{t\in [1,T]}\big(\|{h}(t)\|_{H^{\s-1}}^2\big)\\
&\leq \|{h}(1)\|_{H^{\s-1}}^2+4\|\bar{h}\|_{L^1_T(H^{\s-1})}^2+\f{3}{8}\sup_{t\in [1,T]}\big(\|{h}(t)\|_{H^{\s-1}}^2\big),
\end{align*}
which implies \eqref{eq: aim4}. 

\subsection{Energy estimates of $\bar{h}$ and $h$ in $H^{\s}$}\label{sec: high-h-barh}
\begin{align*}
\f12\f{d}{dt}\|\As\bar{h}\|_2^2
&=-\int\f{\pa_tw(t,\eta)}{w(t,\eta)}\left|\f{\langle \eta\rangle^{\s}\widehat{\bar{h}}(t,\eta)}{w(t,\eta)}\right|^2d\eta
-\f{2}{t}\|\As\bar{h}\|_2^2\\
&\quad \underbrace{-\int \As \bar{h}\left[\As(g\pa_v\bar{h})-g\pa_v\As\bar{h}\right]dv}_{V_{1,\s}^{\bar{h}}}
+\f{1}{2}\int g'|\As \bar{h}|^2dv\\
&\quad +\underbrace{\f{1}{t}\sum_{l\neq 0}\int 
\f{\langle \eta\rangle^{\s}\widehat{\bar{h}}(t,\eta)}{w(t,\eta)} 
\f{\langle \eta\rangle^{\s}\eta l}{w(t,\eta)}\overline{\hat{\phi}_{-l}(\eta-\xi)\hat{f}_l(\xi)}d\xi d\eta}_{V_{2,\s}^{\bar{h}}}\\
&\quad +\underbrace{\f{1}{t}\sum_{l\neq 0}\int 
\f{\langle \eta\rangle^{\s}\overline{\widehat{\bar{h}}(t,\eta)}}{w(t,\eta)} 
\f{\langle \eta\rangle^{\s}(\eta-\xi') l}{w(t,\eta)} \hat{h}(\xi')\hat{\phi}_{-l}(\eta-\xi)\hat{f}_{l}(\xi-\xi')d\xi'd\xi d\eta}_{V_{2,\s}^{\bar{h},\ep}}\\
&\quad -\nu\|\pa_v\As\bar{h}\|_2^2
+\underbrace{\nu \int \As \bar{h} \As \left(((v')^2-1)\pa_{vv}\bar{h}\right)dv}_{V_{3,\s}^{\bar{h}}}\\
&=-CK_{w}^{\bar{h}}-\f{2}{t}\|\As\bar{h}\|_2^2 -\nu\|\pa_v\As\bar{h}\|_2^2+V_{1,\s}^{\bar{h}}+V_{2,\s}^{\bar{h}}+V_{2,\s}^{\bar{h},\ep}+V_{3,\s}^{\bar{h}}
\end{align*}
We the get 
\begin{align*}
V_{1,\s}^{\bar{h}}
&=-\sum_{M\geq 8}\int \As \bar{h}\left[\As(g_{M}\pa_v\bar{h}_{<M/8})-g_{M}\pa_v\As\bar{h}_{<M/8}\right]dv\\
&\quad -\sum_{M\geq 8}\int \As \bar{h}\left[\As(g_{<M/8}\pa_v\bar{h}_{M})-g_{<M/8}\pa_v\As\bar{h}_{M}\right]dv\\
&\quad -\sum_{M\in\mathbb{D}}\sum_{\f18M\leq M'\leq 8M}\int \As \bar{h}\left[\As(g_{M}\pa_v\bar{h}_{M'})-g_{M}\pa_v\As\bar{h}_{M'}\right]dv\\
&=V_{1,\s}^{\bar{h},HL}+V_{1,\s}^{\bar{h},LH}+V_{1,\s}^{\bar{h},HH}.
\end{align*}
We have
\begin{align*}
|V_{1,\s}^{\bar{h},HL}|&\lesssim 
\sum_{M\geq 8}
\int \As \widehat{\bar{h}}(\eta)\left[\As(\hat{g}(\xi)_{M}\widehat{\pa_v\bar{h}}(\eta-\xi)_{<M/8})-\hat{g}(\xi)_{M}\As\widehat{\pa_v\bar{h}}(\eta-\xi)_{<M/8}\right]d\xi d\eta\\
&\lesssim \sum_{M\geq 8}
\int \As \widehat{\bar{h}}(\eta)\langle\xi\rangle^{\s}\hat{g}(\xi)_{M}\widehat{\pa_v\bar{h}}(\eta-\xi)_{<M/8}d\xi d\eta\\
&\lesssim \sum_{M\geq 8} \|\As\bar{h}_{\sim M}\|_2
\|g_{M}\|_{H^{\s}}\|\pa_v\bar{h}\|_{H^1}\\
&\lesssim \|\As\bar{h}\|_2
\|g\|_{H^{\s}}\|\pa_v\bar{h}\|_{H^1}.
\end{align*}

Next we treat $V_{1,\s}^{\bar{h},LH}+V_{1,\s}$, we get
\begin{align*}
|V_{1,\s}^{\bar{h},HL}|&\lesssim 
\sum_{M\geq 8}
\int_{|\xi-\eta|\leq \f{1}{10}|\eta|} \As \widehat{\bar{h}}(\eta)\left[\As(\hat{g}(\eta-\xi)_{<M/8}\widehat{\pa_v\bar{h}}(\xi)_{M})-\hat{g}(\eta-\xi)_{<M/8}\As\widehat{\pa_v\bar{h}}(\xi)_{M}\right]d\xi d\eta\\
&\quad+\sum_{M\geq 8}
\int_{|\xi-\eta|\geq \f{1}{10}|\eta|} \As \widehat{\bar{h}}(\eta)\left[\As(\hat{g}(\eta-\xi)_{<M/8}\widehat{\pa_v\bar{h}}(\xi)_{M})-\hat{g}(\eta-\xi)_{<M/8}\As\widehat{\pa_v\bar{h}}(\xi)_{M}\right]d\xi d\eta,
\end{align*}
For the first term, by Lemma \ref{lem: w-w}, we have $|\xi|\approx |\eta|$ and then
\begin{align*}
&\left|\sum_{M\geq 8}
\int_{|\xi-\eta|\leq \f{1}{10}|\eta|} \As \widehat{\bar{h}}(\eta)\left[\As(\hat{g}(\eta-\xi)_{<M/8}\widehat{\pa_v\bar{h}}(\xi)_{M})-\hat{g}(\eta-\xi)_{<M/8}\As\widehat{\pa_v\bar{h}}(\xi)_{M}\right]d\xi d\eta\right|\\
&\lesssim 
\sum_{M\geq 8}
\int_{|\xi-\eta|\leq \f{1}{10}|\eta|} \As |\widehat{\bar{h}}(\eta)|\left|\f{\langle\eta\rangle\eta}{w(t,\eta)}-\f{\langle\xi\rangle\xi}{w(t,\xi)}\right]|\hat{g}(\eta-\xi)_{<M/8}\widehat{\bar{h}}(\xi)_{M}|d\xi d\eta\\
&\lesssim 
\sum_{M\geq 8}
\int_{|\xi-\eta|\leq \f{1}{10}|\eta|} \As |\widehat{\bar{h}}(\eta)|\left|\langle\eta\rangle^{\s}\eta-\langle\xi\rangle^{\s}\xi\right]|\hat{g}(\eta-\xi)_{<M/8}\widehat{\bar{h}}(\xi)_{M}|d\xi d\eta\\
&\quad+\sum_{M\geq 8}
\int_{|\xi-\eta|\leq \f{1}{10}|\eta|} \As |\widehat{\bar{h}}(\eta)|\left|\f{1}{w(t,\eta)}-\f{1}{w(t,\xi)}\right]\langle\xi\rangle^{\s+1}|\hat{g}(\eta-\xi)_{<M/8}\widehat{\bar{h}}(\xi)_{M}|d\xi d\eta\\
&\lesssim 
\sum_{M\geq 8}
\int_{|\xi-\eta|\leq \f{1}{10}|\eta|} \As |\widehat{\bar{h}}(\eta)|\langle\eta-\xi\rangle\langle\xi\rangle^{\s}|\hat{g}(\eta-\xi)_{<M/8}\widehat{\bar{h}}(\xi)_{M}|d\xi d\eta\\
&\quad+\sum_{M\geq 8}
\int_{|\xi-\eta|\leq \f{1}{10}|\eta|} \As |\widehat{\bar{h}}(\eta)|\left|\f{1}{w(t,\eta)}-\f{1}{w(t,\xi)}\right]\langle\xi\rangle^{\s+1}|\hat{g}(\eta-\xi)_{<M/8}\widehat{\bar{h}}(\xi)_{M}|d\xi d\eta\\
&\lesssim 
\sum_{M\geq 8}
\int_{|\xi-\eta|\leq \f{1}{10}|\eta|} \As |\widehat{\bar{h}}(\eta)|\langle\eta-\xi\rangle\langle\xi\rangle^{\s}|\hat{g}(\eta-\xi)_{<M/8}\widehat{\bar{h}}(\xi)_{M}|d\xi d\eta\\
&\quad+\sum_{M\geq 8}
\int_{|\xi-\eta|\leq \f{1}{10}|\eta|} \As |\widehat{\bar{h}}(\eta)|\f{\langle\eta-\xi\rangle}{\eta}\left(\nu^{-\f13}\chi_{t'\lesssim \nu^{-\f13}}(t')+\nu^{\f13\b}t'^{1-\b}\chi_{t'\gtrsim \nu^{-\f13}}(t')\right)\\
&\quad\quad\quad\quad \quad \quad
\times\langle\xi\rangle^{\s+1}|\hat{g}(\eta-\xi)_{<M/8}\widehat{\bar{h}}(\xi)_{M}|d\xi d\eta\\
&\lesssim \sum_{M\geq 8}\|\As\bar{h}_{\sim M}\|_{2}\|g_{M}\|_{H^{3}}\|\bar{h}\|_{H^{\s}}\left(\nu^{-\f13}\chi_{t'\lesssim \nu^{-\f13}}(t')+\nu^{\f13\b}t'^{1-\b}\chi_{t'\gtrsim \nu^{-\f13}}(t')\right)\\
&\lesssim  \|\As\bar{h}\|_{2}\|g\|_{H^{3}}\|\bar{h}\|_{H^{\s}}\left(\nu^{-\f13}\chi_{t'\lesssim \nu^{-\f13}}(t')+\nu^{\f13\b}t'^{1-\b}\chi_{t'\gtrsim \nu^{-\f13}}(t')\right). 
\end{align*}
For the second term, we have $|\xi-\eta|\geq \f{1}{10}|\eta|\approx |\xi|$, thus we get 
\begin{align*}
&\sum_{M\geq 8}
\left|\int_{|\xi-\eta|\geq \f{1}{10}|\eta|} \As \widehat{\bar{h}}(\eta)\left[\As(\hat{g}(\eta-\xi)_{<M/8}\widehat{\pa_v\bar{h}}(\xi)_{M})-\hat{g}(\eta-\xi)_{<M/8}\As\widehat{\pa_v\bar{h}}(\xi)_{M}\right]d\xi d\eta\right|\\
&\lesssim \sum_{M\geq 8}
\int \As \widehat{\bar{h}}(\eta)\langle\xi\rangle^{\s}\hat{g}(\xi)_{M}\widehat{\pa_v\bar{h}}(\eta-\xi)_{<M/8}d\xi d\eta\\
&\lesssim \sum_{M\geq 8} \|\As\bar{h}_{\sim M}\|_2
\|g_{M}\|_{H^{\s}}\|\pa_v\bar{h}\|_{H^1}
\lesssim \|\As\bar{h}\|_2
\|g\|_{H^{\s}}\|\pa_v\bar{h}\|_{H^1}.
\end{align*}
Now we turn to $V_{1,\s}^{\bar{h},HH}$, we have
\beno
|\eta|\lesssim |\eta-\xi|+|\xi|\approx |\xi|\approx|\eta-\xi|
\eeno
and then
\begin{align*}
|V_{1,\s}^{\bar{h},HH}|
&\lesssim \sum_{M\in \mathbb{D}}\|\As\bar{h}\|_2\|g_{M}\|_{H^{\s}}\|\bar{h}_{\sim M}\|_{H^3}\\
&\lesssim \|\As\bar{h}\|_2\|g\|_{H^{\s}}\|\bar{h}\|_{H^3}
\end{align*}
As suggested by Section \ref{sec: correction}, $V_{2,\s}^{\bar{h},\ep}$ is not significantly harder then $V_{2,\s}^{\bar{h}}$, in fact the primary complications that arise in the treatment of $R_{N}^{\ep,1}$ do not arise in the treatment of $V_{2,\s}^{\bar{h},\ep}$. 
Hence we focus only on $V_{2,\s}^{\bar{h}}$; the control of $V_{2,\s}^{\bar{h},\ep}$ results in, at worst, similar contributions with an additional power of $\ep$. 
Now we turn to $V_{2,\s}^{\bar{h}}$, we use Littlewood-Paley decomposition in $v$ and get that
\begin{align*}
V_{2,\s}^{\bar{h}}
&=\f{1}{t}\sum_{l\neq 0}\int \chi_{|\eta|\leq 100|l|}
\f{\langle \eta\rangle^{\s}\widehat{\bar{h}}(t,\eta)}{w(t,\eta)} 
\f{\langle \eta\rangle^{\s}\eta l}{w(t,\eta)}\overline{\hat{\phi}_{-l}(\eta-\xi)\hat{f}_l(\xi)}d\xi d\eta\\
&\quad+\sum_{M\geq 8}\f{1}{t}\sum_{l\neq 0}\int \chi_{|\eta|\geq 100|l|}
\f{\langle \eta\rangle^{\s}\widehat{\bar{h}}(t,\eta)}{w(t,\eta)} 
\f{\langle \eta\rangle^{\s}\eta l}{w(t,\eta)}\overline{\hat{\phi}_{-l}(\eta-\xi)_{<M/8}\hat{f}_l(\xi)_{M}}d\xi d\eta\\
&\quad+\sum_{M\geq 8}\f{1}{t}\sum_{l\neq 0}\int \chi_{|\eta|\geq 100|l|}
\f{\langle \eta\rangle^{\s}\widehat{\bar{h}}(t,\eta)}{w(t,\eta)} 
\f{\langle \eta\rangle^{\s}\eta l}{w(t,\eta)}\overline{\hat{\phi}_{l}(\xi)_{M}\hat{f}_{-l}(\eta-\xi)_{<M/8}}d\xi d\eta\\
&\quad+\sum_{M\in \mathbb{D}}\sum_{\f18M\leq M'\leq 8M}\f{1}{t}\sum_{l\neq 0}\int \chi_{|\eta|\geq 100|l|}
\f{\langle \eta\rangle^{\s}\widehat{\bar{h}}(t,\eta)}{w(t,\eta)} 
\f{\langle \eta\rangle^{\s}\eta l}{w(t,\eta)}\overline{\hat{\phi}_{-l}(\eta-\xi)_{M'}\hat{f}_l(\xi)_{M}}d\xi d\eta\\
&=V_{2,\s}^{\bar{h},z}
+\sum_{M\geq 8}T_{M}^{\bar{h}}
+\sum_{M\geq 8}R_{M}^{\bar{h}}
+\mathcal{R}^{\bar{h}}. 
\end{align*}
It is easy to obtain that
\begin{align*}
|V_{2,\s}^{\bar{h},z}|\lesssim \f1t\|\bar{h}\|_{H^{\s}}\|\phi_{\neq}\|_{H^{4}}\|f_{\neq}\|_{H^{\s}}
\lesssim \f{1}{t^3}\|\bar{h}\|_{H^{\s}}\|f_{\neq}\|_{H^{\s}}\|f_{\neq}\|_{H^6}
\end{align*}
There is a loss of derivate in $T_{M}^{\bar{h}}$. By using the fact that 
\beno
|\eta|=|(\eta-\xi+lt)|+|(\xi-lt)|\lesssim \langle \eta-\xi+lt\rangle \langle \xi-lt\rangle
\eeno
We get 
\begin{align*}
|T_{M}^{\bar{h}}|&\lesssim
\f{1}{t}\sum_{l\neq 0}\int 1_{|\eta|\geq 100|l|}
\f{\langle \eta\rangle^{\s}|\widehat{\bar{h}}(t,\eta)|}{w(t,\eta)} 
\f{\langle \xi\rangle^{\s}|l|}{w(t,\xi)}|\langle\eta-\xi+lt\rangle\hat{\phi}_{l}(\eta-\xi)_{<M/8}||\langle \xi-lt\rangle\hat{f}_{-l}(\xi)_{M}|d\xi d\eta\\
&\lesssim \f1t\|\As\bar{h}_{\sim M}\|_{2}\|(\sqrt{-\Delta_{L}}\As f_{\neq})_{M}\|_2\|\pa_z\na_{L}\phi_{\neq}\|_{H^2}\\
&\lesssim \f{1}{t^2}\|\As\bar{h}_{\sim M}\|_{2}\|(\sqrt{-\Delta_{L}}\As f_{\neq})_{M}\|_2\|f_{\neq}\|_{H^3}.
\end{align*}

For $R_{M}^{\bar{h}}$, we have
\begin{align*}
|R_{M}^{\bar{h}}|
&\lesssim \f{1}{t}\sum_{l\neq 0}\int 1_{|\eta|\geq 100|l|}
|\As\widehat{\bar{h}}(t,\eta)|
\As(t,\xi)\f{\xi/l^2}{\langle\f{\xi}{l}-t\rangle^2}\widehat{\Delta_{L}\Delta_t^{-1}f}_{l}(\xi)_{M}\widehat{\pa_zf}_{-l}(\eta-\xi)_{<M/8}d\xi d\eta\\
&\lesssim \f{1}{t}\sum_{l\neq 0}\int 1_{|\eta|\geq 100|l|}\chi_{|l|\geq \f{1}{10}\sqrt{|\xi|}}
|\As\widehat{\bar{h}}|
\As(t,\xi)\f{\xi/l^2}{\langle\f{\xi}{l}-t\rangle^2}\widehat{\Delta_{L}\Delta_t^{-1}f}_{l}(\xi)_{M}\widehat{\pa_zf}_{-l}(\eta-\xi)_{<M/8}d\xi d\eta\\
&\quad+\f{1}{t}\sum_{l\neq 0}\int 1_{|\eta|\geq 100|l|}
1_{|l|\leq \f{1}{10}\sqrt{|\xi|}} \left[1_{t\in I_{l,\xi}}+1_{t\notin I_{l,\xi}}\right]
|\As\widehat{\bar{h}}|\\
&\quad\quad \quad \quad\quad \quad \times
\As(t,\xi)\f{\xi/l^2}{\langle\f{\xi}{l}-t\rangle^2}\widehat{\Delta_{L}\Delta_t^{-1}f}_{l}(\xi)_{M}\widehat{\pa_zf}_{-l}(\eta-\xi)_{<M/8}d\xi d\eta\\
&=R_{M}^{\bar{h},z}+R_{M,R}^{\bar{h}}+R_{M,NR}^{\bar{h}}
\end{align*}
We have
\begin{align*}
|R_{M}^{\bar{h},z}|\lesssim \f1t\|\As\bar{h}_{\sim M}\|_{2}
\|\As(\Delta_{L}\Delta_t^{-1}f_{\neq})_{M}\|_2\|f_{\neq}\|_{H^3}
\end{align*}

Next we treat $R_{M,NR}^{\bar{h}}$, we have 
\begin{align*}
|R_{M,NR}^{\bar{h}}|\lesssim 
\f{1}{t}\left\|\As\bar{h}_{\sim M}\right\|_{2}
\left\|\As(\Delta_{L}\Delta_t^{-1}f_{\neq})_{M}\right\|_2
\|f_{\neq}\|_{H^3}
\end{align*}

If $t\in I_{l,\xi}$ with $|l|\leq \f{1}{10}\sqrt{|\xi|}$, then according the integrand we get $|\xi-\eta|\leq \f{3}{16}|\xi|$, and thus $t\in I_{k,\eta}$. By Lemma \ref{lem: 3.2}, we have the follow three cases. 

a.) $k=l$. We get 
\begin{align*}
|R_{M,R}^{\bar{h}}|\lesssim 
\f{1}{t}\left\|\sqrt{\f{\pa_tw}{w}}\As\bar{h}_{\sim M}\right\|_{2}
\left\|\sqrt{\f{\pa_tw}{w}}\chi_{R}\As(\Delta_{L}\Delta_t^{-1}f_{\neq})_{M}\right\|_2
\|\nu^{-\f13}\langle \nu^{\f13}t\rangle^{1+\b}f_{\neq}\|_{H^6}
\end{align*}

b.) $|\f{\eta}{k}-t|\gtrsim \f{\eta}{k^2}$ and $|\f{\xi}{k}-t|\gtrsim \f{\xi}{k^2}$. Then the estimate is similar to $R_{M,NR}^{\bar{h}}$ and we get
\begin{align*}
|R_{M,R}^{\bar{h}}|\lesssim 
\f{1}{t}\left\|\As\bar{h}_{\sim M}\right\|_{2}
\left\|\As(\Delta_{L}\Delta_t^{-1}f_{\neq})_{M}\right\|_2
\|f_{\neq}\|_{H^3}
\end{align*}

c.) $|\eta-\xi|\gtrsim \f{\xi}{l^2}$, then we get 
\begin{align*}
|R_{M,R}^{\bar{h}}|\lesssim 
\f{1}{t}\left\|\As\bar{h}_{\sim M}\right\|_{2}
\left\|\As(\Delta_{L}\Delta_t^{-1}f_{\neq})_{M}\right\|_2
\|f_{\neq}\|_{H^4}
\end{align*}

The remainder term is easy to dealt with. We use the fact that $|\eta|\lesssim |\eta-\xi|+|\xi|$ and $|\xi|\approx |\eta-\xi|$ and get that
\begin{align*}
|\mathcal{R}^{\bar{h}}|
&\lesssim \sum_{M\in \mathbb{D}}\sum_{\f18M\leq M'\leq 8M}\f{1}{t}\sum_{l\neq 0}\int \chi_{|\eta|\geq 100|l|}
|\As\widehat{\bar{h}}(t,\eta)|
\langle \eta\rangle^{\s-2}|\hat{\phi}_{-l}(\eta-\xi)_{M'}||\xi|^3|l||\hat{f}_l(\xi)_{M}|d\xi d\eta\\
&\lesssim \sum_{M\in \mathbb{D}}\f{1}{t}\|\As\bar{h}\|_2\|P_{\neq}\phi_{\sim M}\|_{H^{\s-2}}\|(f_{\neq})_{M}\|_{H^7}\\
&\lesssim \f{1}{t}\|\As\bar{h}\|_2\|P_{\neq}\phi\|_{H^{\s-2}}\|f_{\neq}\|_{H^7}\lesssim \f{1}{t}\|\As\bar{h}\|_2\|f_{\neq}\|_{H^{\s-2}}\|f_{\neq}\|_{H^7}. 
\end{align*}
To treat the dissipation error term $V_{3,\s}^{\bar{h}}$, we have
\begin{align*}
|V_{3,\s}^{\bar{h}}|&\lesssim 
\int \langle \eta \rangle^{\s}|\widehat{\bar{h}}(\eta)|\langle \eta \rangle^{\s}|G(\eta-\xi)||\xi|^2|\widehat{\bar{h}}(\xi)|d\xi d\eta\\
&\lesssim 
\int_{|\eta-\xi|\leq |\xi|}1_{|\eta|\leq 1} |\widehat{\bar{h}}(\eta)||G(\eta-\xi)||\xi|^2|\widehat{\bar{h}}(\xi)|d\xi d\eta\\
&\quad+\int_{|\eta-\xi|\leq |\xi|} 1_{|\eta|\geq  1}\langle \eta \rangle^{\s}|\widehat{\bar{h}}(\eta)|\langle \eta \rangle^{\s}|G(\eta-\xi)||\xi|^2|\widehat{\bar{h}}(\xi)|d\xi d\eta\\
&\quad+\int_{|\eta-\xi|\geq |\xi|}1_{|\eta|\geq 1} \langle \eta \rangle^{\s}|\widehat{\bar{h}}(\eta)|\langle \eta \rangle^{\s}|G(\eta-\xi)||\xi|^2|\widehat{\bar{h}}(\xi)|d\xi d\eta\\
&=V_{3,\s}^{\bar{h},<1}+V_{3,\s}^{\bar{h},LH}+V_{3,\s}^{\bar{h},HL}.
\end{align*}
Now we treat $V_{3,\s}^{\bar{h},<1}$, 
\begin{align*}
|V_{3,\s}^{\bar{h},<1}|\lesssim \nu\|\bar{h}\|_{L^2}\|\bar{h}\|_{H^2}\|G\|_{H^{2}}\lesssim \nu\|\bar{h}\|_{L^2}\|\bar{h}\|_{H^2}\|h\|_{H^{2}}.
\end{align*}
Next we turn to $V_{3,\s}^{\bar{h},LH}$, in which case it holds that
\beno
|\eta|\leq |\eta-\xi|+|\xi|\lesssim |\xi|. 
\eeno
Then we have
\begin{align*}
V_{3,\s}^{\bar{h},LH}
&\lesssim \nu\int_{|\eta-\xi|\leq |\xi|} |\eta|\langle \eta \rangle^{\s}|\widehat{\bar{h}}(\eta)|\langle \xi \rangle^{\s-1}|G(\eta-\xi)||\xi|^2|\widehat{\bar{h}}(\xi)|d\xi d\eta\\
&\lesssim \nu\|h\|_{H^3}\|\pa_v\As\bar{h}\|_2^2
\lesssim \ep\nu^{\f13}\nu\|\pa_v\As\bar{h}\|_2^2.
\end{align*}
Next we treat $V_{3,\s}^{\bar{h},HL}$ and get that
\begin{align*}
|V_{3,\s}^{\bar{h},HL}|
&\lesssim \nu\int_{|\eta-\xi|\geq |\xi|} 
|\eta|\langle \eta \rangle^{\s}|\widehat{\bar{h}}(\eta)|\langle \eta-\xi \rangle^{\s-1}|G(\eta-\xi)||\xi|^2|\widehat{\bar{h}}(\xi)|d\xi d\eta\\
&\lesssim \nu\|\pa_v\bar{h}\|_{H^{\s}}\|h\|_{H^{\s-1}}\|\pa_v\bar{h}\|_{H^3}.
\end{align*}

Thus we conclude that
\begin{align*}
&\f12\f{d}{dt}\|\As\bar{h}\|_2^2+CK_{w}^{\bar{h}}+\f{2}{t}\|\As\bar{h}\|_2^2 +\f12\nu\|\pa_v\As\bar{h}\|_2^2\\
&\lesssim 
\|\As\bar{h}\|_2^2
\|g\|_{H^{\s}}
+\|\As\bar{h}\|_{2}^2\|g\|_{H^{3}}\left(\nu^{-\f13}\chi_{t'\lesssim \nu^{-\f13}}(t')+\nu^{\f13\b}t'^{1-\b}\chi_{t'\gtrsim \nu^{-\f13}}(t')\right)\\
&\quad+\f{1}{t}\|\bar{h}\|_{H^{\s}}\|f_{\neq}\|_{H^{\s}}\|f_{\neq}\|_{H^7}
+\f{1}{t^2}\|\As\bar{h}\|_{2}\|(\sqrt{-\Delta_{L}}\As f_{\neq})\|_2\|f_{\neq}\|_{H^3}\\
&\quad+\f{1}{t}\left\|\As\bar{h}\right\|_{2}
\left\|\As(\Delta_{L}\Delta_t^{-1}f_{\neq})\right\|_2
\|f_{\neq}\|_{H^3}\\
&\quad +\f{1}{t}\left\|\sqrt{\f{\pa_tw}{w}}\As\bar{h}\right\|_{2}
\left\|\sqrt{\f{\pa_tw}{w}}\chi_{R}\As(\Delta_{L}\Delta_t^{-1}f_{\neq})\right\|_2
\|\nu^{-\f13}\langle \nu^{\f13}t\rangle^{1+\b}f_{\neq}\|_{H^6}\\
&\quad +\nu\|\bar{h}\|_{L^2}\|\bar{h}\|_{H^2}\|h\|_{H^{2}},
\end{align*}
which gives by the bootstrap hypotheses that
\begin{align*}
&(t^{\f32}\|\As\bar{h}(t)\|_2)^2
+\int_1^tt'^{3}CK_{w}^{\bar{h}}dt'
+\f12\int_1^tt'^{2}\|\As\bar{h}(t')\|_2^2dt'
+\f12\nu\int_1^tt'^{3}\|\pa_v\As\bar{h}(t')\|_2^2dt'\\
&\leq \|\As\bar{h}(1)\|_2^2\\
&\quad+C\left[\|g\|_{L^1_T(H^{\s})}
+\int_{1}^t\f{\ep\nu^{\f13}}{t'^2}\left(\nu^{-\f13}\chi_{t'\lesssim \nu^{-\f13}}(t')+\nu^{\f13\b}t'^{1-\b}\chi_{t'\gtrsim \nu^{-\f13}}(t')\right)dt'\right]\|t^{\f32}\|\As\bar{h}(t)\|_2\|_{L^{\infty}_T}^2\\
&\quad+C\int_{1}^{t}\f{\ep^2\nu^{\f23}t^{\f12}}{\langle\nu t'^3\rangle}dt'\|t^{\f32}\|\As\bar{h}(t)\|_2\|_{L^{\infty}_T}\\
&\quad+ C\|t^{\f32}\|\As\bar{h}(t)\|_2\|_{L^{\infty}_T}\left(\int_1^t\left[\f{\ep t'^{\f12}}{t'}\f{\nu^{-\f16}}{\langle \nu t'^3\rangle}\right]^2dt'\right)^{\f12}\nu^{\f12}\|\sqrt{-\Delta_{L}}f_{\neq}\|_{L^2_{T}(H^{\s})}\\
&\quad +C\|t^{\f32}\|\As\bar{h}(t)\|_2 \left(\int_1^t\left[t'^{\f12}\f{\ep\nu^{-\f16}}{\langle t\rangle \langle \nu t^3\rangle}\right]^2dt'\right)^{\f12}
\nu^{\f12}\|\pa_vh(t')\|_{L^2_{T}H^{\s}}\\
&\quad +C\left\|t'^{\f32}\sqrt{\f{\pa_tw}{w}}\As\bar{h}(t')\right\|_{L^2_{T}(L^2)}
\left\|\sqrt{\f{\pa_tw}{w}}f_{\neq}\right\|_{L^2_T(H^{\s})}\left\|\f{\ep^2t^{\f12}\langle \nu^{\f13}t\rangle^{1+\b}}{\langle\nu t^3\rangle}\right\|_{L^{\infty}}\\
&\quad+C\left\|t'^{\f32}\sqrt{\f{\pa_tw}{w}}\As\bar{h}(t')\right\|_{L^2_{T}(L^2)}
\left\|\f{\ep^3\nu^{\f12}t^{\f12}}{\langle \nu t^3\rangle}\right\|_{L^2_T}\\
&\leq \|\As\bar{h}(1)\|_2^2
+\f{1}{100} \|t^{\f32}\|\As\bar{h}(t)\|_2\|_{L^{\infty}_T}^2+\f{1}{100} \int_1^tt'^{3}CK_w^{\bar{h}}dt'
+C\ep^4\nu^{\f13}. 
\end{align*}

The energy for $h$ in $H^{\s}$ is similar to the estimate of $H^{\s-1}$. We have
\begin{align*}
\f12\f{d}{dt}\|{h}\|_{H^{\s}}^2
&=-\int \langle \pa_v\rangle^{\s}h\langle \pa_v\rangle^{\s}\Big(g\pa_v h-\bar{h}-\nu (v')^2\pa_v^2h\big)dv\\
&=-\int \langle \pa_v\rangle^{\s}h\langle \pa_v\rangle^{\s}(g\pa_v h)dv\\
&\quad+\int \langle \pa_v\rangle^{\s}h \langle \pa_v\rangle^{\s}\bar{h}dv\\
&\quad +\nu \int \langle \pa_v\rangle^{\s}h \langle \pa_v\rangle^{\s}\big(((v')^2-1)\pa_v^2h\big)dv
-\nu\|\pa_v h\|_{H^{\s}}^2\\
&=-\nu\|\pa_v h\|_{H^{\s}}^2+V^{H,{h}}_{1,\s}+V^{H,{h}}_{2,\s}+V^{H,{h}}_{3,\s},
\end{align*}
We get that
\beno
|V^{H,{h}}_{1,\s}|\lesssim \|g\|_{H^{\s}}\|h\|_{H^{\s}}^2,
\eeno
and
\beno
|V^{H,{h}}_{2,\s}|\lesssim \|h\|_{H^{\s}}\|\bar{h}\|_{H^{\s}},
\eeno
and
\beno
|V^{H,{h}}_{3,\s}|\lesssim \ep\nu^{\f13}\nu \|\pa_vh\|_{H^{\s}}^2.
\eeno
Thus we conclude that
\begin{align*}
\|{h}(t)\|_{H^{\s}}^2+\nu\|\pa_vh\|_{L^2_T(H^{\s})}^2
&\leq \|{h}(1)\|_{H^{\s}}^2
+C\|g\|_{L^1_{T}H^{\s}}\|{h}\|_{L^{\infty}_{T}H^{\s}}^2
+C\|{h}(1)\|_{H^{\s}}\|\bar{h}\|_{L^1_{T}(H^{\s})}\\
&\leq \|{h}(1)\|_{H^{\s}}^2+\f{1}{100} \|{h}\|_{L^{\infty}_{T}H^{\s}}^2+C\ep^3\nu^{\f13}. 
\end{align*}

\subsection{Lower energy estimate}\label{Sec: Lower energy estimate}
\subsubsection{Energy estimate of $g$ in $H^{\s-6}$}\label{sec: g-low}
\begin{align*}
\f{d}{dt}\Big( t^{4}\|\langle \pa_v\rangle^{\s-6}g\|_{2}^2\Big)
&=(4) t^{3}\|\langle \pa_v\rangle^{\s-6} g\|_2^2
+ t^{4}\f{d}{dt}\|\langle \pa_v\rangle^{\s-6}g\|_{2}^2\\
&=(4)t^{3}\|\langle \pa_v\rangle^{\s-6} g\|_2^2
+2 t^{4}\int \langle \pa_v\rangle^{\s-6}g \langle \pa_v\rangle^{\s-6}\pa_tgdv\\
&=-2t^{4}\int \langle \pa_v\rangle^{\s-6}g \langle \pa_v\rangle^{\s-6}(g\pa_vg)dv\\
&\quad-2t^{3}\int \langle \pa_v\rangle^{\s-6}g \langle \pa_v\rangle^{\s-6} (v'<\na_{z,v}^{\bot}P_{\neq}\phi\cdot\na_{z,v}\tu>)dv\\
&\quad+2t^{4}\nu \int \langle \pa_v\rangle^{\s-6}g \langle \pa_v\rangle^{\s-6}(((v')^2-1)\pa_v^2g)dv
-2t^{4}\nu \|\pa_vg\|_{H^{\s-6}}^2\\
&=-2t^{4}\nu \|\pa_vg\|_{H^{\s-6}}^2+V_1^{L,g}+V_2^{L,g}+V_3^{L,g}.
\end{align*}

As before we deal with $V_1^{L,g}$ by commutator estimate and integration by parts. 
\begin{align*}
|V_1^{L,g}|
&\lesssim \left|t^{4}\int \pa_vg |\langle \pa_v\rangle^{\s-6}g|^2 dv\right|
+\left|t^{4}\int \langle \pa_v\rangle^{\s-6}g [\langle \pa_v\rangle^{\s-6},g]\pa_vg dv\right|\\
&\lesssim t^{4}\|g\|_{H^{3}}\|g\|_{H^{\s-6}}^2.
\end{align*}

For $V_2^{L,g}$, by the fact that $<\na_{z,v}^{\bot}P_{\neq}\phi\cdot\na_{z,v}\tu>=\pa_v<\pa_zP_{\neq} \phi \tu>$, we get
\begin{align*}
|V_2^{L,g}|
&\lesssim t^{3}\|g\|_{H^{\s-6}}
\Big(\left\|((v'-1)<\na_{z,v}^{\bot}P_{\neq}\phi\cdot\na_{z,v}\tu>)\right\|_{H^{\s-6}}^2
+\left\|<\na_{z,v}^{\bot}P_{\neq}\phi\cdot\na_{z,v}\tu>\right\|_{H^{\s-6}}^2\Big)\\
&\lesssim t^{3}\|g\|_{H^{\s-6}}
(1+\|h\|_{H^{\s-6}})\|<\pa_zP_{\neq}\phi\tu> \|_{H^{\s-5}}\\
&\lesssim t^{3}\|g\|_{H^{\s-6}}
(1+\|h\|_{H^{\s-6}})
\|P_{\neq}\phi\|_{H^{\s-4}}\|\tu\|_{H^{\s-5}}.
\end{align*}
Then by the bootstrap assumption and Lemma \ref{lem: lin-inv-dam}, we get 
\begin{align*}
|V_2^{L,g}|
&\lesssim \|g\|_{H^{\s-6}}
\|f_{\neq}\|_{H^{\s-2}}\|f_{\neq}\|_{H^{\s-4}}.
\end{align*}
For the dissipation error term. We have 
\begin{align*}
|V_3^{L,g}|
&\lesssim \left|t^{4}\nu \int \langle \pa_v\rangle^{\s-6}g \langle \pa_v\rangle^{\s-6}(((v')^2-1)\pa_v^2g)dv\right|\\
&\lesssim t^{4}\nu\|((v')^2-1)\|_{H^{\s-6}}\|\pa_vg\|_{H^{\s-6}}^2 \\
&\quad+ t^4\nu \|\pa_v((v')^2-1)\|_{2}\|\pa_vg\|_{2}\|g\|_{H^1}
+ t^4\nu \|((v')^2-1)\|_{2}\|\pa_vg\|_{2}^2\\
&\lesssim t^{4}\nu(1+\|h\|_{H^2})\|h\|_{H^{\s-6}}\|\pa_vg\|_{H^{\s-6}}^2 
+ t^4\nu (1+\|h\|_{H^2})\|\pa_v h\|_{H^{2}}\|\pa_vg\|_{H^2}\|g\|_{H^1}.
\end{align*}

Thus by the bootstrap assumption, we get that 
\begin{align*}
&\sup_{t\in [1,T]}t^4\|g(t)\|_{H^{\s-6}}^2
+\nu\int_1^Tt'^4 \|\pa_vg(t')\|_{H^{\s-6}}^2dt'\\
&\leq \|g(1)\|_{H^{\s-6}}^2+C\bigg(\|g\|_{L^1_T(H^{\s-6})}\sup_{t\in [1,T]}t^4\|g(t)\|_{H^{\s-6}}^2
+\sup_{t'\in [1,t]}t'^2\|g(t')\|_{H^{\s-6}}\int_1^T\f{\|f_{\neq}\|_{L^{\infty}_T(H^{\s-2})}^2}{t^2}dt\\
&\quad +\nu\|\pa_vh\|_{L^2_T(H^2)}\|t^2\pa_vg\|_{L^2_T(H^2)}\sup_{t'\in [1,t]}t'^2\|g(t')\|_{H^{\s-6}}\bigg)\\
&\leq \|g(1)\|_{H^{\s-6}}^2+C\ep\nu^{\f13}\sup_{t\in [1,T]}t^4\|g(t)\|_{H^{\s-6}}^2+\ep^2\nu^{\f23}\sup_{t'\in [1,t]}t'^2\|g(t')\|_{H^{\s-6}}. 
\end{align*}

Thus by taking $\ep$ small enough, we proved \eqref{eq: aim6}.

\subsubsection{Energy estimate of $\bar{h}$ in $H^{\s-6}$}\label{sec: bar h}
The estimate is same as the estimates of $g$ in lower Sobolev spaces. We have
\begin{align*}
\f{d}{dt}\Big( t^{4}\|\langle \pa_v\rangle^{\s-6}\bar{h}\|_{2}^2\Big)
&=4 t^{3}\|\langle \pa_v\rangle^{\s-6} \bar{h}\|_2^2
+ t^{4}\f{d}{dt}\|\langle \pa_v\rangle^{\s-6}\bar{h}\|_{2}^2\\
&=(4)t^{3}\|\langle \pa_v\rangle^{\s-6} \bar{h}\|_2^2
+2 t^{4}\int \langle \pa_v\rangle^{\s-6}\bar{h} \langle \pa_v\rangle^{\s-6}\pa_t\bar{h}dv\\
&=-2t^{4}\int \langle \pa_v\rangle^{\s-6}\bar{h} \langle \pa_v\rangle^{\s-6}(g\pa_v\bar{h})dv\\
&\quad-2t^{3}\int \langle \pa_v\rangle^{\s-6}\bar{h} \langle \pa_v\rangle^{\s-6} (v'<\na_{z,v}^{\bot}P_{\neq}\phi\cdot\na_{z,v}f>)dv\\
&\quad+2t^{4}\nu \int \langle \pa_v\rangle^{\s-6}\bar{h} \langle \pa_v\rangle^{\s-6}(((v')^2-1)\pa_v^2\bar{h})dv
-2t^{4}\nu \|\pa_v\bar{h}\|_{H^{\s-6}}^2\\
&=-2t^{4}\nu \|\pa_v\bar{h}\|_{H^{\s-6}}^2+V_1^{L,\bar{h}}+V_2^{L,\bar{h}}+V_3^{L,\bar{h}}.
\end{align*}
Then we have 
\beno
|V_1^{L,\bar{h}}|\lesssim \|g\|_{H^{\s-6}}\|t^2\bar{h}\|_{H^{\s-6}}^2. 
\eeno
For $V_2^{L,\bar{h}}$, we get by Lemma \ref{lem: lin-inv-dam} that, 
\begin{align*}
|V_2^{L,\bar{h}}|
&\lesssim t\|t^2\bar{h}\|_{H^{\s-6}}(1+\|h\|_{H^{\s-6}})\|<\pa_zP_{\neq}\phi f_{\neq}>\|_{H^{\s-5}}\\
&\lesssim t\|t^2\bar{h}\|_{H^{\s-6}}(1+\|h\|_{H^{\s-6}})\|P_{\neq}\phi\|_{H^{\s-6}}\|f_{\neq}\|_{H^{2}}\\
&\quad+t\|t^2\bar{h}\|_{H^{\s-6}}(1+\|h\|_{H^{\s-6}})\|P_{\neq}\phi\|_{H^3} \|f_{\neq}\|_{H^{\s-5}}\\
&\lesssim t^{-1}\|t^2\bar{h}\|_{H^{\s-6}}(1+\|h\|_{H^{\s-6}})\|P_{\neq}f\|_{H^{\s-4}}\|f_{\neq}\|_{H^{2}}\\
&\quad+t^{-1}\|t^2\bar{h}\|_{H^{\s-6}}(1+\|h\|_{H^{\s-6}})\|P_{\neq}f\|_{H^5} \|f_{\neq}\|_{H^{\s-5}}.
\end{align*}
At last for the dissipation error term, we have 
\begin{align*}
|V_3^{L,\bar{h}}|\lesssim t^{4}\nu(1+\|h\|_{H^2})\|h\|_{H^{\s-6}}\|\pa_v\bar{h}\|_{H^{\s-6}}^2 
+ t^4\nu (1+\|h\|_{H^2})\|\pa_v h\|_{H^{2}}\|\pa_v\bar{h}\|_{H^2}\|\bar{h}\|_{H^1}.
\end{align*}
Thus by the bootstrap hypotheses, we obtain that
\begin{align*}
&\sup_{t\in [1,T]}t^4\|\bar{h}(t)\|_{H^{\s-6}}^2
+\nu\int_1^Tt'^4 \|\pa_v\bar{h}(t')\|_{H^{\s-6}}^2dt'\\
&\leq \|\bar{h}(1)\|_{H^{\s-6}}^2+C\bigg(\ep\nu^{\f13}\sup_{t\in [1,T]}t^4\|\bar{h}(t)\|_{H^{\s-6}}^2
+\ep^2\nu^{\f13}\sup_{t\in [1,T]}t^2\|\bar{h}(t)\|_{H^{\s-6}}\bigg)\\
&\leq \|\bar{h}(1)\|_{H^{\s-6}}^2+\f{1}{100}\sup_{t\in [1,T]}t^4\|\bar{h}(t)\|_{H^{\s-6}}^2+C\ep^4\nu^{\f23}. 
\end{align*}
Therefore by taking $\ep$ small enough, we proved \eqref{eq: aim7}.

\section{Decay estimate of vorticity}\label{Sec: Decay estimate of vorticity}

\subsection{Decay estimate of nonzero mode: Enhanced dissipation}\label{sec: f-low-enha}
Up to an adjustment of the constants in the bootstrap argument, it suffices to consider only $t$ such that $\nu t^3\geq 1$ (say), as otherwise the decay estimate follows trivially from the higher regularity energy estimate. 

Recall that
\beno
\|A^s_Ef\|_2^2=\sum_{k\neq 0}\int_{\eta} \langle k,\eta\rangle^{2s} |D(t,\eta)\hat{f}_k(t,\eta)|^2d\eta.
\eeno

Computing the time evolution of $\|A^s_Ef\|_2$ 
\begin{align*}
\f12\f{d}{dt}\|A^s_Ef\|_2^2
&=\sum_{k\neq 0}\int_{\eta} \f{\pa_tD(t,\eta)}{D(t,\eta)}|A^s_E\hat{f}_k(t,\eta)|^2d\eta\\
&\quad-\int A^s_EfA^s_E(u\cdot\na f)dvdz
+\nu\int A^s_EfA^s_E(\tilde{\Delta}_tf)dvdz\\
&\leq \f{1}{8}\nu t^2\left\|1_{t\geq 2|\eta|}A^s_E\widehat{f}_k(t,\eta)\right\|_2^2\\
&\quad-\int A^s_EfA^s_E(u\cdot\na f)dvdz
+\nu\int A^s_EfA^s_E(\tilde{\Delta}_tf)dvdz
\end{align*}
We write the dissipation term as follows
\begin{align*}
\nu\int A^s_EfA^s_E(\tilde{\Delta}_tf)dvdz
&=-\nu\left\|\sqrt{-\Delta_L}A^s_Ef\right\|_2^2
-\nu\int A^s_EfA^s_E\left(((v')^2-1)(\pa_v^2-t\pa_z^2)f\right)dvdz\\
&=-\nu\left\|\sqrt{-\Delta_L}A^s_Ef\right\|_2^2+E^{\nu}
\end{align*}
First, we need to cancel the growing term cause by $D(t,\eta)$. Indeed, we have
\begin{align*}
&\f{1}{8}\nu t^2\left\|1_{t\geq 2|\eta|}A^s_E\widehat{f}_k(t,\eta)\right\|_2^2-\nu\left\|\sqrt{-\Delta_L}A^s_Ef\right\|_2^2\\
&=\sum_{k\neq 0}\int \nu\left(\f18 t^21_{t\geq 2\eta}-k^2-(\eta-kt)^2\right)|A^s_E\hat{f}_k(\eta)|^2d\eta\\
&\leq -\f{1}{8}\nu\left\|\sqrt{-\Delta_L}A^s_Ef\right\|_2^2.
\end{align*}
which gives that
\beq
\f12\f{d}{dt}\|A^s_Ef\|_2^2\leq -\int A^s_EfA^s_E(u\cdot\na f)dvdz-\f{1}{8}\nu\left\|\sqrt{-\Delta_L}A^s_Ef\right\|_2^2+E^{\nu}
\eeq

\subsubsection{Euler nonlinearity}
We first divide into zero and non-zero frequency contributions, as they will be treated differently: 
\begin{align*}
-\int A^s_EfA^s_E(u\cdot\na f)dvdz
&=-\int A^s_EfA^s_E(g\pa_vf)dvdz
-\int A^s_EfA^s_E\left(v'\na^{\bot}P_{\neq}\phi\cdot\na f\right)dvdz\\
&=E_1+E_2. 
\end{align*}
For $E_1$ we use the commutator trick and the paraproduct (in both $z$ and $v$)
\begin{align*}
E_1&=\f12\int \pa_vg|A^s_Ef|^2dvdz
+\int A^s_Ef\left[g\pa_vA^s_Ef-A^s_E(g\pa_vf)\right]dvdz\\
&=\f12\int \pa_vg|A^s_Ef|^2dvdz
+\sum_{N\geq 8}T_N^0+\sum_{N\geq 8}R_N^0+\mathcal{R}^0,
\end{align*}
where
\begin{align*}
&T_N^0=\int A^s_Ef\left[g_{<N/8}\pa_vA^s_Ef_N-A^s_E(g_{<N/8}\pa_vf_N)\right]dvdz\\
&R_N^0=\int A^s_Ef\left[g_{N}\pa_vA^s_Ef_{<N/8}-A^s_E(g_{N}\pa_vf_{<N/8})\right]dvdz\\
&\mathcal{R}^0=\sum_{N\in\mathbb{D}}\sum_{N/8\leq N'\leq 8N}\int A^s_Ef\left[g_{N'}\pa_vA^s_Ef_{N}-A^s_E(g_{N'}\pa_vf_{N})\right]dvdz
\end{align*}

\no{\it Treatment of $T_N^0$.}\\
We get that
\begin{align*}
T_N^0
&=-i\sum_{k\neq 0}\int_{\eta,\xi}
A^s_E\bar{\hat{f}}_k(\eta)
D(\eta)(\langle k,\eta\rangle^{s}-\langle k,\xi\rangle^{s})
\hat{g}(\eta-\xi)_{<N/8}\xi\hat{f}_k(\xi)_Nd\eta d\xi\\
&\quad-i\sum_{k\neq 0}\int_{\eta,\xi}
A^s_E\bar{\hat{f}}_k(\eta)
\langle k,\xi\rangle^{s}(D(\eta)-D(\xi))
\hat{g}(\eta-\xi)_{<N/8}\xi\hat{f}_k(\xi)_Nd\eta d\xi\\
&=T_N^{0,1}+T_N^{0,2}.
\end{align*}

For the term $T_N^{0,1}$, by Lemma \ref{lem: D-D} and the fact that $|k,\eta|\approx |k,\xi|$ and 
\beno
|\langle k,\eta\rangle^{s}-\langle k,\xi\rangle^{s}|\lesssim \f{|\xi-\eta|}{\langle \eta\rangle+\langle \xi\rangle}\langle k,\xi\rangle^{s}
\eeno
we have
\begin{align*}
|T_N^{0,1}|&\lesssim \sum_{k\neq 0}\int_{\eta,\xi}
|A^s_E\bar{\hat{f}}_k(\eta)|
\langle\eta-\xi\rangle^4|\hat{g}(\eta-\xi)_{<N/8}|\f{|\xi|}{\langle \xi\rangle} A_E^s\hat{f}_k(\xi)_Nd\eta d\xi\\
&\lesssim \|A_E^sf_{\sim N}\|_{2}\|A_E^sf_{N}\|_2\|g\|_{H^5}
\end{align*}
We turn to $T_N^{0,2}$, by Lemma \ref{lem: D-D}, we get 
\begin{align*}
|T_N^{0,2}|&\lesssim \sum_{k\neq 0}\int_{\eta,\xi}
|A^s_E\bar{\hat{f}}_k(\eta)|
\langle\eta-\xi\rangle|\hat{g}(\eta-\xi)_{<N/8}|\f{|\xi|}{\langle \xi\rangle} A_E^s\hat{f}_k(\xi)_Nd\eta d\xi\\
&\lesssim \|A_E^sf_{\sim N}\|_{2}\|A_E^sf_{N}\|_2\|g\|_{H^2}.
\end{align*}

Thus we get by \eqref{eq: L-P-ortho} and \eqref{eq: L-P-ortho2} that
\beq\label{eq: T_N^0}
\sum_{N\geq 8}|T_N^0|\leq \|g\|_{H^5}\|A_E^sf\|_2^2
\eeq

\no{\it Treatment of $R_N^0$. }\\
The ‘reaction’ term $R_N^0$ is dealt with easily by 'moving' the derivative to $g$. We have 
\begin{align*}
\left|\int A^s_Efg_{N}\pa_vA^s_Ef_{<N/8}dvdz\right|
\lesssim \|A^s_Ef_{\sim N}\|_{2}\|g_{N}\|_{H^2}\|A^s_Ef\|_{2},
\end{align*}
and by Lemma \ref{lem: D-D} and the fact that $A^s_E(g_{N}\pa_v(f)_{<N/8})=A^s_E(g_{N}\pa_v(f_{\neq})_{<N/8})$, we get 
\begin{align*}
&\left|\int A^s_EfA^s_E(g_{N}\pa_v(f_{\neq})_{<N/8})dvdz\right|\\
&\lesssim \left|\sum_k\int_{\eta,\xi} A^s_E|\hat{f}_k(\eta)|
\langle\eta-\xi\rangle^3\langle k,\eta\rangle^{s}
|\hat{g}_{N}(\xi-\eta)||\xi||\hat{f}_k(\xi)_{<N/8}|d\eta d\xi\right|.
\end{align*}
On the support of the integrand, we have $|k,\eta|\approx |\xi-\eta|\gtrsim |k,\xi|$ and thus
\begin{align*}
\left|\int A^s_EfA^s_E(g_{N}\pa_v(f_{\neq})_{<N/8})dvdz\right|
\lesssim \|A^s_Ef_{\sim N}\|_2\|A^s_Ef\|_{2}\|g_{N}\|_{H^{s+4}}. 
\end{align*}

The treatment of the remainder terms is similar to the reaction term. We have 
\begin{align*}
\mathcal{R}^0
&\lesssim \sum_{N\in \mathbb{D}}\|A^s_Ef\|_2\|A^s_Ef_{N}\|_{2}\|g_{\sim N}\|_{H^{s+4}}\\
&\lesssim \|g\|_{H^{s+4}}\|A^s_Ef\|_2^2.
\end{align*}

Therefore by the bootstrap hypotheses, we conclude that
\beq\label{eq: E_1}
|E_1|\lesssim \|g\|_{H^{s+4}}\|A^s_Ef\|_2^2\leq \f{\ep\nu^{\f13}}{\langle t\rangle^2}\|A^s_Ef\|_2^2. 
\eeq

\no{\it Treatment of $E_2$.}\\
Next turn to $E_2$. Now we need use the inviscid damping to obtain decay in time. Roughly speaking, if $f$ is of zero mode, we will land the operator $A^s_E$ on $P_{\neq}\phi$ and use the lossy elliptic estimate for $A^s_E$. 

Thus we get by Lemma \ref{lem: product} that
\begin{align*}
|E_2|&\lesssim \|A_E^sf\|_2\|A_E^s((v'-1)\na^{\bot}P_{\neq}\phi\cdot\na f)\|_2
+\|A_E^sf\|_2\|A_E^s(\na^{\bot}P_{\neq}\phi\cdot\na f)\|_2\\
&\lesssim \|A_E^sf\|_2(1+\|h\|_{H^{s+3}})\|A_E^s(\na^{\bot}P_{\neq}\phi \cdot\na f_{\neq})\|_2\\
&\lesssim \|A_E^sf\|_2(1+\|h\|_{H^{s+3}})\|A_E^sP_{\neq}\phi\|_{2}\|f\|_{H^{s+5}}+\|A_E^sf\|_2^2(1+\|h\|_{H^{s+3}})\|P_{\neq}\phi\|_{H^{s+5}}. 
\end{align*}
Thus by the bootstrap hypotheses and Lemma \ref{lem: loss-elliptic-A_E^s}, we have
\begin{align}\label{eq: E_2}
|E_2|\lesssim \f{\ep\nu^{\f13}}{\langle t\rangle^2}(\|A_E^sf\|_2^2+\|\As f\|_2\|A_E^sf\|_2). 
\end{align}
\subsubsection{Dissipation error term}
By Lemma \ref{lem: D-D} and the fact that
\beno
|\xi-kt|\leq |\xi-\eta|+|\eta-kt|\leq \langle \xi-\eta\rangle \sqrt{k^2+|\eta-kt|^2},
\eeno
we have
\begin{align*}
|E^{\nu}|
&\lesssim 
\nu \sum_{k\neq 0}\int_{\eta,\xi}\left|
A_E^s\hat{f}_k(\eta)A_E^s(k,\eta)\widehat{(1-(v')^2)}(\eta-\xi)|\xi-kt|^2\hat{f}_k(\xi)\right|d\eta d\xi\\
&\lesssim 
\nu \sum_{k\neq 0}\int_{\eta,\xi}\left|
A_E^s\sqrt{k^2+|\eta-kt|^2}\hat{f}_k(\eta)A_E^s(k,\eta)\langle \xi-\eta\rangle\widehat{(1-(v')^2)}(\eta-\xi)|\xi-kt|\hat{f}_k(\xi)\right|d\eta d\xi\\
&\lesssim 
\nu \sum_{k\neq 0}\int_{\eta,\xi}\left|
A_E^s\sqrt{k^2+|\eta-kt|^2}\hat{f}_k(\eta)\langle \xi-\eta\rangle^4\widehat{(1-(v')^2)}(\eta-\xi)|\xi-kt|A_E^s(k,\xi)\hat{f}_k(\xi)\right|d\eta d\xi\\
&\lesssim \nu\left\|\sqrt{-\Delta_L}A_E^sf\right\|_2^2\|(1-(v')^2)\|_{H^6}\lesssim \nu(1+\|h\|_{H^2})\|h\|_{H^6}\left\|\sqrt{-\Delta_L}A_E^sf\right\|_2^2.
\end{align*}

Thus we get that
\begin{align*}
\f12\f{d}{dt}\|A^s_Ef\|_2^2
&\leq E_1+E_2-\f{1}{8}\nu\left\|\sqrt{-\Delta_L}A^s_Ef\right\|_2^2+E^{\nu}\\
&\leq \f{C\ep\nu^{\f13}}{\langle t\rangle^2}\|A^s_Ef\|_2^2
+\f{C\ep\nu^{\f13}}{\langle t\rangle^2}\|\As f\|_2^2\\
&\quad-\f{1}{8}\nu\left\|\sqrt{-\Delta_L}A^s_Ef\right\|_2^2
+C\nu\ep\nu^{\f13}\left\|\sqrt{-\Delta_L}A^s_Ef\right\|_2^2, 
\end{align*}
which gives that
\beno
\|A^s_Ef(t)\|_2^2
+\int_1^t\f{1}{5}\nu\left\|\sqrt{-\Delta_L}A^s_Ef(t')\right\|_2^2dt'
\leq \|A^s_Ef(1)\|_2^2+C\ep\nu^{\f13}\|A^s_Ef(t)\|_2^2+C\ep^3\nu.
\eeno
Thus by taking $\ep$ small enough, we proved \eqref{eq: aim5}.

\subsection{Decay of zero mode}\label{sec: f-low-0}
Here we start the proof of \eqref{eq: aim8}. The zero mode $f_0$ satisfies 
\beq
\pa_tf_0+g\pa_v f_0+v'<\na^{\bot}_{z,v}P_{\neq}\phi \cdot \na_{z,v}f>-\nu (v')^2\pa_v^2f_0=0.
\eeq 

We want to prove that the zero mode slightly decays. It is nature to study the time evolution of 
\beno
\mathcal{E}_{L,0}(t)=\|\langle \pa_v\rangle^{s}f_0\|_2^2+\f{t\nu}{2}\|\langle \pa_v\rangle^{s}\pa_vf_0\|_2^2.
\eeno 
We get 
\begin{align*}
\f{d}{dt}\mathcal{E}_{L,0}(t)
&=\f12\nu \|\langle \pa_v\rangle^{s}\pa_vf_0\|_2^2
+t\nu\f12\f{d}{dt}\left(\|\langle \pa_v\rangle^{s}\pa_vf_0\|_2^2\right)+\f{d}{dt}\left(\|\langle \pa_v\rangle^{s}f_0\|_2^2\right)\\
&=-\f32\nu \|\langle \pa_v\rangle^{s}\pa_vf_0\|_2^2
-\nu^2 t\|\pa_v^2f_0\|_{H^s}^2\\
&\quad-\nu t \int \langle \pa_v\rangle^{s}\pa_vf_0 \langle \pa_v\rangle^{s}\pa_v(g\pa_v f_0)dv
-2\int \langle \pa_v\rangle^{s}f_0 \langle \pa_v\rangle^{s}(g\pa_v f_0)dv\\
&\quad-\nu t\int \langle \pa_v\rangle^{s}\pa_vf_0 \langle \pa_v\rangle^{s}\pa_v\left( v'< \na^{\bot}_{z,v}P_{\neq}\phi \cdot \na_{z,v}f>\right) dv \\
&\quad-2\int \langle \pa_v\rangle^{s}f_0 \langle \pa_v\rangle^{s}\left( v'< \na^{\bot}_{z,v}P_{\neq}\phi \cdot \na_{z,v}f>\right) dv \\
&\quad+\nu^2 t\int \langle \pa_v\rangle^{s}\pa_vf_0  \langle \pa_v\rangle^{s}\pa_v( ((v')^2-1)\pa_v^2f_0)dv\\
&\quad+2\nu \int \langle \pa_v\rangle^{s}f_0  \langle \pa_v\rangle^{s}(((v')^2-1)\pa_v^2f_0)dv\\
&=-\f32\nu \|\langle \pa_v\rangle^{s}\pa_vf_0\|_2^2
-\nu^2 t\|\pa_v^2f_0\|_{H^s}^2\\
&\quad+ V_{1,1}+V_{1,2}+V_{2,1}+V_{2,2}+V_{3,1}+V_{3,2}. 
\end{align*}

To treat $V_{1,1}$ and $V_{1,2}$, we use commutator estimate and integration by part. 
\begin{align*}
|V_{1,1}|+|V_{1,2}|
&\lesssim \|\pa_v g\|_{L^{\infty}}(\|f_0\|_{H^s}^2+\nu t\|\pa_vf_0\|_{H^s}^2)\\
&\quad+\|f_0\|_{H^s}\|[\langle \pa_v\rangle^{s},g]\pa_vf_0\|_{L^2}
+\nu t\|\pa_vf_0\|_{H^s}\|[\langle \pa_v\rangle^{s}\pa_v,g]\pa_vf_0\|_{L^2}\\
&\lesssim \| g\|_{H^{s}}\left(\|f_0\|_{H^s}^2+\f12\nu t\|\pa_vf_0\|_{H^s}^2\right).
\end{align*}

Next we turn to $V_{2,1}, V_{2,2}$, by using the fact that 
\beno
< \na^{\bot}_{z,v}P_{\neq}\phi \cdot \na_{z,v}f>=\pa_v<\pa_z P_{\neq}\phi f_{\neq}>,
\eeno
we get that
\begin{align*}
|V_{2,1}|&\lesssim \nu t\|\pa_vf_0\|_{H^s}
\left(\left\|h\pa_v<\pa_z P_{\neq}\phi f_{\neq}>\right\|_{H^{s+1}}+\left\|\pa_v<\pa_z P_{\neq}\phi f_{\neq}>\right\|_{H^{s+1}}\right)\\
&\lesssim \nu t(1+\|h\|_{H^{s+1}})\|\pa_vf_0\|_{H^s}
\left\|P_{\neq}\phi\right\|_{H^{s+3}} \left\|f_{\neq}\right\|_{H^{s+2}},
\end{align*}
and similarly 
\begin{align*}
|V_{2,2}|&\lesssim \|f_0\|_{H^s}
\left(\left\|h\pa_v<\pa_z P_{\neq}\phi f_{\neq}>\right\|_{H^{s}}+\left\|\pa_v<\pa_z P_{\neq}\phi f_{\neq}>\right\|_{H^{s+1}}\right)\\
&\lesssim (1+\|h\|_{H^{s}})\|f_0\|_{H^s}
\left\|P_{\neq}\phi\right\|_{H^{s+2}} \left\|f_{\neq}\right\|_{H^{s+1}}.
\end{align*}

Finally we turn to $V_{3,1},V_{3,2}$, as before, we have
\begin{align*}
|V_{3,1}|\lesssim 
&\nu^2t \|\pa_v^2f_0\|_{H^s}^2\|(v')^2-1\|_{H^{s+1}}
+\nu^2t \|\pa_v((v')^2-1)\|_{2}\|\pa_v^2f_0\|_{2}\|\pa_vf_0\|_{H^1},
\end{align*}
and
\begin{align*}
|V_{3,2}|\lesssim 
&\nu\|\pa_v^2f_0\|_{H^s}^2\|(v')^2-1\|_{H^{s}}+
\nu\|\pa_v((v')^2-1)\|_{2}\|\pa_vf_0\|_{2}\|f_0\|_{H^1}
\end{align*}

Thus by the bootstrap assumption, we get that
\begin{align*}
&\sup_{t'\in [1,t]}\mathcal{E}_{L,0}(t')
+\nu\int_1^T\|\pa_vf_0(t)\|_{H^s}^2dt\\
&\leq \left(\|\langle \pa_v\rangle^{s}f_0(1)\|_2^2+\f{\nu}{2}\|\langle \pa_v\rangle^{s}\pa_vf_0(1)\|_2^2\right)\\
&\quad
+C\bigg[\|g\|_{L^1_{T}(H^s)}\sup_{t'\in [1,t]}\mathcal{E}_{L,0}(t')
+\left[\sup_{t'\in [1,t]}\mathcal{E}_{L,0}(t')\right]^{\f12}\|f_{\neq}\|_{L^{\infty}_T(H^{s+5})}^2\int_1^T\f{\sqrt{\nu t}+1}{t^2}dt\\
&\quad+\nu \|\pa_vh\|_{L^2_{T}H^1} \|\sqrt{\nu t}\pa_v^2f_0\|_{L^2_{T}(L^2)}\|\sqrt{\nu t}\pa_vf_0\|_{L^{\infty}_{T}(H^1)}
+\nu\|\pa_vh\|_{L^2_T(L^2)}\|\pa_vf_0\|_{L^2_T(L^2)}\|f_0\|_{L^{\infty}_T(H^1)}\bigg]\\
&\leq \left(\|\langle \pa_v\rangle^{s}f_0(1)\|_2^2+\f{\nu}{2}\|\langle \pa_v\rangle^{s}\pa_vf_0(1)\|_2^2\right)
+C\ep\nu^{\f13}\sup_{t'\in [1,t]}\mathcal{E}_{L,0}(t')
+C\ep^2\nu^{\f23}\left[\sup_{t'\in [1,t]}\mathcal{E}_{L,0}(t')\right]^{\f12}
+C\ep^3\nu\\
&\leq \left(\|\langle \pa_v\rangle^{s}f_0(1)\|_2^2+\f{\nu}{2}\|\langle \pa_v\rangle^{s}\pa_vf_0(1)\|_2^2\right)
+C\ep\sup_{t'\in [1,t]}\mathcal{E}_{L,0}(t')+C\ep^3\nu.
\end{align*}
Thus by taking $\ep$ small enough, we proved \eqref{eq: aim8}.

\section{Appendix}
\subsection{Littlewood-Paley decomposition and paraproducts}\label{Sec: L-P}
In this section we fix conventions and notation regarding Fourier analysis, Littlewood-Paley and paraproduct decompositions. See e.g. \cite{BCD-book,Bony} for more details. 

For $f(z,v)$ in the Schwartz space, we define the Fourier transform $\hat{f}_k(\eta)$ where $(k,\eta)\in \mathbb{Z}\times \R$, 
\beno
\hat{f}_k(\eta)=\f{1}{2\pi}\int_{\mathbb{T}\times R}f(z,v)e^{-ikz-iv\eta}dzdv,
\eeno
and the Fourier inversion formula,
\beno
f(z,v)=\f{1}{2\pi}\sum_{k\in\mathbb{Z}}\int_{\R}\hat{f}_k(\eta)e^{ikz+iv\eta}d\eta.
\eeno 
With this definition, we have
\beno
&&\int f(z,v)g(z,v)dzdv=\sum_{k}\int \hat{f}_k(\eta)\hat{g}_k(\eta)d\eta,\\
&&\hat{fg}=\hat{f}\ast \hat{g}. 
\eeno

This work makes heavy use of the Littlewood-Paley dyadic decomposition. Here we fix conventions and review the basic properties of this classical theory, see e.g. \cite{BCD-book} for more details. First we define the Littlewood-Paley decomposition only in the $v$ variable. Let $\psi\in C_0^{\infty}(\R; \R)$ be such that $\psi(\xi)=1$ for $|\xi|\leq \f12$ and $\psi(\xi)=0$ for $|\xi|\geq \f34$ and define $\chi(\xi)=\psi(\xi/2)-\psi(\xi)$ supported in the range $\xi\in (\f12,\f32)$. Then we have the partition of unity
\beno
1=\psi(\xi)+\sum_{M\in 2^{\mathbb{N}}}\chi_{M}(\xi)
\eeno
where we mean that the sum runs over the dyadic numbers $M=1,2,4,8,...,2^j,...$ and $\chi_{M}(\xi)=\chi(M^{-1}\xi)$ which has the compact support $M/2\leq |\xi|\leq 3M/2$. For $f\in L^2(\R)$, we define
\beno
&&f_M=\big(\chi_{M}(\xi)\hat{f}(\xi)\big)^{\vee},\\
&&f_{\f12}= \big(\psi(\xi)\hat{f}(\xi)\big)^{\vee},\\
&&f_{<M}=f_{\f12}+\sum_{K\in 2^{\mathbb{N}},K<M}f_K
\eeno
which defines the decomposition
\beno
f=f_{\f12}+\sum_{K\in 2^{\mathbb{N}}}f_K.
\eeno
There holds the almost orthogonality and the approximate projection property
\beq\label{eq: L-P-ortho}
\begin{split}
&\|f\|_2^2\approx \sum_{K\in \mathbb{D}}\|f_K\|_2^2,\\
&\|f_{M}\|_2^2\approx \|(f_{M})_{M}\|_2^2. 
\end{split}
\eeq
The following is also clear for $M\geq 1$
\beno
\||\pa_v|f_{M}\|_2^2\approx M\|f_{M}\|_2^2.
\eeno
We make use of the notation
\beno
f_{\sim M}=\sum_{K\in \mathbb{D}:\ \f{1}{C}M\leq K\leq CM}f_{K},
\eeno
for some constant $C$ which is independent of $M$. Generally the exact value of $C$ which is being used is not important; what is important is that it is finite and independent of $M$. With this notation, we also have 
\beq\label{eq: L-P-ortho2}
\|f\|_2^2\approx_{C} \sum_{K\in \mathbb{D}}\|f_{\sim K}\|_2^2
\eeq
During much of the proof we are also working with Littlewood-Paley decompositions defined in the $(z,v)$ variables, with the notation conventions being analogous. Our convention is to use $N$ to denote Littlewood-Paley projections in $(z,v)$ and $M$ to denote projections only in the $v$ direction.

Another key Fourier analysis tool employed in this work is the paraproduct decomposition, introduced by Bony \cite{Bony}(see also \cite{BCD-book}).  Given suitable functions $f,g$ we may define the paraproduct decomposition (in either $(z, v)$ or just $v$), 
\begin{align*}
fg&=T_fg+T_gf+\mathcal{R}(f,g)\\
&=\sum_{N\geq 8}f_{<N/8}g_N+\sum_{N\geq 8}f_{N}g_{<N/8}
+\sum_{N\in \mathbb{D}}\sum_{N/8\leq N'\leq 8N}g_{N'}f_N,
\end{align*}
where all the sums are understood to run over $\mathbb{D}$. In our work we do not employ the notation in the first line since at most steps in the proof we are forced to explicitly write the sums and treat them term-by-term anyway. This is due to the fact that we are working in non-standard regularity spaces and, more crucially, are usually applying multipliers which do not satisfy any version of $AT_fg\approx T_fAg$. Hence, we have to prove almost everything ‘from scratch’ and can only rely on standard para-differential calculus as a guide. 

We also show some product estimates(or Young's inequality) based on Sobolev embedding. It holds for $s>1$ that
\beq\label{eq: prod1}
\begin{split}
&\|fg\|_{H^s(\mathbb{T}\times R)}
\lesssim \|f\|_{H^s(\mathbb{T}\times R)}\|g\|_{H^{s}(\mathbb{T}\times R)},\\
&\|f\ast g\|_{2}\lesssim \|f\|_2\|g\|_{H^s(\mathbb{T}\times R)},\\
&\|f\ast g\ast h\|_2\lesssim \|f\|_2\|g\|_2\|h\|_{H^s(\mathbb{T}\times R)}. 
\end{split}
\eeq 

We end this subsection by introducing the commutator estimate which can be found in \cite{KatoPonce}. 
\begin{lemma}[\cite{KatoPonce}]\label{lem: commutator}
Let $J=(1-\Delta)^{\f12}$, for $1<p<\infty$ and $s\geq 0$, it holds that
\beno
\|J^s(fg)-f(J^sg)\|_p\lesssim_{p,s} 
\|\na f\|_{\infty}\|J^{s-1}g\|_p+\|J^sf\|_p\|g\|_{\infty}.
\eeno
\end{lemma}

\subsection{Composition lemma}
According to the coordinate transform, we need the following composition lemma. 
\begin{lemma}\label{lem: composition}
Suppose that $\g>1$, let $F\in H^{\g }:\mathbb{T}\times\R\to \R$, $G:\ \mathbb{T}\times\R\to\mathbb{T}\times\R$ be such that $\|\na G-I_{2\times 2}\|_{L^{\infty}}\leq \f14$ and $\na G-I_{2\times 2}\in H^{\g }:\mathbb{T}\times\R\to \mathcal{M}_{2\times 2}$. Then there exists $C=C(\|\na G-I_{2\times 2}\|_{H^\g },\g )$ such that
\beno
\|F\circ G\|_{H^\g }\leq C\|F\|_{H^\g }.
\eeno
\end{lemma}
\begin{proof}
First we have $\|F\circ G\|_2^2\approx \|F\|_2^2$. Then 
by the fact that $\na (F\circ G)=[(\na F)\circ G](\na G-I_{2\times 2})+(\na F)\circ G$, we have
\beno
\|F\circ G\|_{H^\g }\lesssim \|(\na F)\circ G\|_{H^{\g -1}}\|(\na G-I_{2\times 2})\|_{H^\g }+\|(\na F)\circ G\|_{H^{\g -1}}\lesssim \|(\na F)\circ G\|_{H^{\g -1}}. 
\eeno
Let $\g =[\g ]+\{\g \}$ with $\{\g \}\in [0,1)$, then by the equivalent definition of the fractional order Sobolev spaces, we get that
\beno
\|F\|_{H^\g }\approx \|F\|_{H^{[\g ]}}+\sum_{\g_1+\g_2=[\g]}\left(\int_{(\mathbb{T}\times\R)^2}\f{|\pa_{x}^{\g_1}\pa_{y}^{\g_2}F(x_1,y_1)-\pa_{x}^{\g_1}\pa_{y}^{\g_2}F(x_2,y_2)|^2}{((x_1-x_2)^2+(y_1-y_2)^2)^{1+\{\g\}}}dx_1dy_1dx_2dy_2\right)^{\f12}.
\eeno 
Therefore, we only need to prove 
\beno
\|F\circ G\|_{H^{\{\g\}}}\leq C\|F\|_{H^{\{\g\}}}.
\eeno
Indeed, we have
\begin{align*}
&\|F\circ G\|_{H^{\{\g\}}}^2\approx \int_{\mathbb{T}\times\R}\int_{\mathbb{T}\times\R}\f{|F(G(x_1,y_1))-F(G(x_2,y_2))|^2}{((x_1-x_2)^2+(y_1-y_2)^2)^{1+\{\g\}}}dx_1dy_1dx_2dy_2\\
&\lesssim \int_{\mathbb{T}\times\R}\int_{\mathbb{T}\times\R}\f{|F(G(x_1,y_1))-F(G(x_2,y_2))|^2}{|G(x_1,y_1)-G(x_1,y_1)|^{2+2\{\g\}}}\f{|G(x_1,y_1)-G(x_1,y_1)|^{2+2\{\g\}}}{((x_1-x_2)^2+(y_1-y_2)^2)^{1+\{\g\}}}dx_1dy_1dx_2dy_2\\
&\lesssim \|\na G\|_{L^{\infty}}\int_{\mathbb{T}\times\R}\int_{\mathbb{T}\times\R}\f{|F(G(x_1,y_1))-F(G(x_2,y_2))|^2}{|G(x_1,y_1)-G(x_1,y_1)|^{2+2\{\g\}}}dx_1dy_1dx_2dy_2.
\end{align*}

By assumption $\|\na G-I_{2\times 2}\|_{L^{\infty}}\leq \f14$, we have $(x,y)\to (z,v)=G(x,y)$ is invertible and thus
 \begin{align*}
\|F\circ G\|_{H^{\{\g\}}}^2
&\lesssim \|\na G\|_{L^{\infty}}\int_{\mathbb{T}\times\R}\int_{\mathbb{T}\times\R}\f{|F(z_1,v_1)-F(z_1,v_1)|^2}{((z_1-z_2)^2+(v_1-v_2)^2)^{1+\{\g\}}}dz_1dv_1dz_2dv_2\\
&\lesssim (\|\na G-I_{2\times 2}\|_{H^2}+1)\|F\|_{H^{\{\g\}}}^2.
\end{align*}
Thus we proved the lemma. 
\end{proof}

\end{document}